\documentclass[a4paper,11pt,reqno]{amsart}

\usepackage{amssymb}
\usepackage{amsmath}
\usepackage{mathtools}
\usepackage{latexsym}
\usepackage{eepic}
\usepackage{epsfig}
\usepackage{graphicx}
\usepackage{amscd}
\usepackage{color}
\usepackage[bookmarks=false]{hyperref}
\usepackage{tikz}
\usetikzlibrary{arrows.meta}
\usetikzlibrary{calc}
\usetikzlibrary{positioning}
\usepackage{fancyhdr}
\usepackage{caption}
\usepackage{subcaption}
\usepackage[ruled]{algorithm2e}

\usepackage{mathrsfs}
\usepackage[english]{babel}
\usepackage[all,cmtip]{xy}

\numberwithin{figure}{section}

\DeclareMathAlphabet{\mathpzc}{OT1}{pzc}{m}{it}

\newtheorem{theorem}{Theorem}[section]
\newtheorem*{theorem*}{Theorem}
\newtheorem{theorem-definition}[theorem]{Theorem-Definition}
\newtheorem{lemma-definition}[theorem]{Lemma-Definition}
\newtheorem{definition-prop}[theorem]{Proposition-Definition}
\newtheorem{corollary}[theorem]{Corollary}
\newtheorem{prop}[theorem]{Proposition}
\newtheorem*{prop*}{Proposition}
\newtheorem{lemma}[theorem]{Lemma}
\newtheorem{cor}[theorem]{Corollary}
\newtheorem{definition}[theorem]{Definition}
\newtheorem*{definition*}{Definition}

\newtheorem{assumption}[theorem]{Assumption}

\theoremstyle{definition}
\newtheorem{example}[theorem]{Example}

\newtheorem{remark}[theorem]{Remark}

\newcommand{\Z}{\ensuremath{\mathbb{Z}}}

\newcommand{\R}{\ensuremath{\mathbb{R}}}

\renewcommand{\R}{\ensuremath{\mathbb{R}}}

\newcommand{\Spec}{\ensuremath{\mathrm{Spec}\,}}

\numberwithin{equation}{section}

\title{Degenerations of generalized Kummer varieties}

\author[L.~H.~Halle]{Lars H. Halle}
\address{University of Bologna\\
Department of Mathematics\\
Piazza Porta S.~Donato 5\\
40126 Bologna\\
Italy}
\email{larshalvard.halle@unibo.it}

\author[K.~Hulek]{Klaus Hulek}
\address{Leibniz Universit\"at Hannover\\
Institut f\"ur algebraische Geometrie\\
Welfengarten 1\\
30167 Hannover\\
Germany}
\email{hulek@math.uni-hannover.de}

\author[Z.~Zhang]{Ziyu Zhang}
\address{ShanghaiTech University\\
Institute of Mathematical Sciences\\
393 Middle Huaxia Road\\
201210 Shanghai\\
P.R.China}
\email{zhangziyu@shanghaitech.edu.cn}



\begin{document}

\begin{abstract}
We present a method to construct explicit
degenerations of higher-dimensional generalized Kummer varieties.
We start with a simple
degeneration $f: \mathcal Y \to C$  of abelian surfaces.  
Then $ \mathcal{Y} \setminus \mathcal{Y}_0$ is an abelian scheme over $C \setminus 0$ and we can form 
the relative generalized Kummer variety
$K^{n-1}_{\circ} = \mathrm{Kum}^{n-1}(\mathcal{Y} \setminus \mathcal{Y}_0) \to C \setminus 0$. This is naturally a closed subscheme of the relative Hilbert scheme $\mathrm{Hilb}^{n}(\mathcal{Y} \setminus \mathcal{Y}_0) \to C \setminus 0$. In previous work (joint with Gulbrandsen) we had constructed a compactification $I^n_{\mathcal{Y}/C}$ over $C$ of the latter scheme.
The closure $K^{n-1}_{\mathcal{Y}/C}$ of $K^{n-1}_{\circ}$ inside $I^n_{\mathcal{Y}/C}$ yields a canonical way to degenerate the family of generalized Kummer varieties, 
and is the degeneration we propose. This paper contains a detailed study of the geometry of the scheme 
$K^{n-1}_{\mathcal{Y}/C}$ and its natural stratification. For $n=2$ we
obtain a projective Kulikov model of Kummer surfaces, whereas already for $n=3$ new phenomena occur. We study in detail the dual complex of $K^{2}_{\mathcal{Y}/C}$ and show that this is 
PL-homeomorphic to the standard $2$-simplex.
\end{abstract}

\maketitle

\section{Introduction}

\subsection{Degenerations of curves and surfaces}
Studying degenerations of $1$-parameter families of smooth varieties is an important and much studied subject in algebraic geometry. Apart from the intrinsic interest, one can find numerous motivations for the wealth of work in this area. One obvious and prominent reason for investigating degenerations is to define compactifications of moduli spaces parametrizing a given class of varieties. Degenerations can also be extremely useful for computational purposes. Namely, if a variety $V$ degenerates to, say, a normal crossing union $ \cup_i W_i$ of smooth varieties, the geometry of each irreducible component $W_i$ tends to be simpler than that of $V$, and one has in some sense traded geometric complexity for \emph{combinatorial} complexity. This can be often exploited in computations not depending on the choice of a fiber in a family. Degenerations also have many applications for arithmetic questions. To give one example, we mention that if $X$ is a variety with trivial canonical sheaf over a complete discretely valued field $K$, then the asymptotic behavior of the set of rational points of $X$ defined over finite extensions of $K$ is encoded by the so-called motivic zeta function of $X$. This is a certain rational function which is computable from sufficiently nice \emph{proper} degenerations of $X$ over the valuation ring of $K$ (see for instance \cite{HN}).

Constructions of degenerations are of a highly non-unique nature. More precisely, if one starts with a smooth family $X$ defined over a punctured smooth curve $C \setminus 0$, there will in general be infinitely many ways to fill in a fiber over $0 \in C $ while still preserving nice properties of the family (e.g flatness). It is usually too optimistic to hope for a smooth filling; the best one can hope for in general, is to fill in with a fiber that is a strict normal crossing union of smooth proper varieties. In this case, we say that the family has \emph{semi-stable reduction} over $0 \in C$. 
Starting with the work of Deligne and Mumford \cite{DM69}, which has since been extended by many mathematicians, we now know that semi-stable reduction holds in equal characteristic zero, possibly after a finite ramified base change. However, it should be pointed out that, especially in high dimension, it can be very difficult to understand the geometry of the irreducible components in a semi-stable degeneration, as well as how they intersect. Depending on the situation, it might therefore be desirable to also seek other types of degenerations, say, with mild singularities.

The best understood case is that of curves. 
Here, the semi-stable reduction theorem is known in arbitrary characteristic. Moreover, one has well understood geometrically meaningful compactifications of the moduli space $\mathcal{M}_g$ of 
curves of genus $g \geq 2$. The best known of these is the Deligne-Mumford compactification $\overline{\mathcal M}_g$ parametrizing stable curves of genus $g$.

By necessity, the picture is more complicated in higher dimensions. The situation is well-understood for surfaces, at least for surfaces with torsion canonical bundle. This goes back to the seminal work of Kulikov 
\cite{Ku77} and Pinkham-Persson \cite{PP81}. They proved, that if $ 0 \in C$ is a pointed smooth curve, and $X \to C \setminus \{0\}$ is a smooth family of proper surfaces with trivial canonical bundle, then, possibly after a finite base change, the family can be completed to a \emph{Kulikov model} $\mathcal{X}$ over $C$. Roughly, this means that the relative canonical bundle is trivial and the fiber $\mathcal{X}_0$ over $0$ is a reduced divisor with strict normal crossings; we refer to Section \ref{sec:KulikovNeron} for the definition, and further details. 

The geometry and combinatorial structure of the special fiber $\mathcal{X}_0$ of a Kulikov model is very restricted. In particular, it turns out that the \emph{dual complex} $\Delta(\mathcal{X}_0)$ must be one of three types. If the general fiber $X_t$ is a K3 surface, the dual complex is either a point, a partion of an interval, or a triangulation of the sphere $S^2$. If $X_t$ is an abelian surface, the dual complex is either a point, a partition of the circle $S^1$ or a triangulation of the torus $S^1 \times S^1$. In either case these are called degenerations of type $I$, $II$ or $III$, respectively. It is moreover known that the type corresponds precisely to the level of unipotency for the monodromy representation on the middle cohomology of $X_t$.

\subsection{Hyperk\"ahler varieties and their degenerations}
Hyperk\"ahler varieties, also known as irreducible holomorphic symplectic manifolds, can, in many ways, be seen as higher dimensional analogues of K3 surfaces, and have been studied in great depth over the last four decades. There are two known families of examples, namely deformations of  Hilbert schemes $\mathrm{Hilb}^n(S)$ of $n$ points on a K3 surface $S$, and of generalized Kummer varieties $\mathrm{Kum}^{n}(A)$ obtained as the kernel of the summation map $\mathrm{Hilb}^{n+1}(A) \to A$ where $A$ is an abelian surface. For fixed $n$, both of these are smooth of dimension $2n$. These families depend on 21 and 5 parameters, which is one more than the dimension of the versal deformation space of K3 and abelian surfaces respectively. Besides these families, there are two sporadic examples, namely varieties of OG6  and OG10 type. These have dimensions
6 and 10 and are named after O'Grady who first constructed them. These varieties have 5 and 22 moduli respectively. 
Note that these are moduli of K\"ahler manifolds (not necessarily projective). Imposing a polarization reduces the number of moduli by 1 in each case. 
It is a wide open problem whether other examples of hyperk\"ahler manifolds exist.
Hyperk\"ahler manifolds and K3 surfaces share many properties, but, typically, the situation is more complex and involved in the higher-dimensional case. One example is their Hodge theory. Period maps can also be constructed for hyperkahler manifolds, but the global Torelli theorem, developed for polarized hyperk\"ahler manifolds by Verbitsky and Markman, is considerably more subtle. As a result, moduli spaces of polarized hyperk\"ahler manifolds can be constructed, but show many new properties (e.g. have several components)
which do not appear for K3 surfaces. 
A very natural question, not least in view of moduli theory, is to 
try and understand \emph{degenerations} of hyperk\"ahler manifolds. The study of minimal dlt models of hyperk\"ahler varieties, their associated dual complexes and the relation to the monodromy representation on the cohomology of a general fiber was studied by Koll{\'a}r, Laza, Sacca and Voisin in \cite{KLSV}. Their results form a natural generalization of (a part of) the work of Kulikov and Persson-Pinkham.

\subsection{Previous work on degenerations}
Relatively little is known about explicit degenerations of hyperk\"ahler varieties. The first study goes back to Nagai \cite{Na08} 
who studied degenerations of Hilbert schemes of K3 surfaces. It is indeed a natural idea to start with a well behaved degeneration of 
K3 surfaces of type II, and then take the associacted family of Hilbert schemes of cycles of length $n$. This naive approach does not yield 
degenerations with good properties. In \cite{Na08} Nagai found a suitable modification of the total space to obtain good degenerations for $n=2$. 

Together with M.~Gulbrandsen, the first two authors studied degenerations of $\mathrm{Hilb}^n(X)$, with $X$ either a K3 or an abelian surface, in a series of papers (cf. e.g. \cite{GHH}, \cite{GHHZ21}). In fact, the degenerations studied there are obtained from taking specific degenerations of $X$ as input; we recall the main features of this work. Let $f \colon \mathcal{X} \to (0 \in C)$ be a \emph{simple} degeneration of surfaces over a smooth pointed curve (see Definition \ref{def:simpledegen}). In particular, the fiber $\mathcal{X}_0$ over $0$ is semi-stable, without triple intersections, and all other fibers are smooth surfaces. 

The relative Hilbert scheme of $n$ points on a non-smooth family is not well understood beyond relative dimension one. To remedy this fact, we instead used Jun Li's theory of \emph{expanded degenerations} (\cite{Li01}, \cite{Li13}). Li replaces the original family with $f[n] \colon \mathcal{X}[n] \to C[n]$, where $C[n]$ has an \'etale map to $\mathbb{A}^{n+1}$. For the purpose of this introduction, it is useful to recall that this map is an isomorphism over the union of the coordinate hyperplanes in $\mathbb{A}^{n+1}$. The morphism $f[n]$ has an equivariant action by a rank $n$ torus $\mathbb \mathbb G[n]$. The fiber of $\mathcal{X}[n]$ over a point in the base where $b$ coordinates are zero is obtained from $\mathcal{X}_0$ by replacing each component $D$ of the double locus by a chain of $b-1$ $\mathbb{P}^1$-bundles over $D$ intersecting along the sections at $0$
 and $\infty$.

 Under mild assumptions, the map $\mathcal{X}[n] \to C[n]$ is projective, provided this was the case for the original degeneration. The relative Hilbert scheme $\mathrm{Hilb}^n(\mathcal{X}[n]/C[n])$ therefore has a suitable relatively ample line bundle. It moreover inherits an action by the torus $\mathbb G[n]$, so one can study GIT-stability. The main result in \cite{GHH} yields an explicit description of the GIT-stable locus $\mathcal{H}^n$. In particular, a subscheme $Z$ in a fiber of $f[n]$ is stable if and only if it is contained in the smooth locus, and if the underlying cycle of $Z$ satisfies an explicit combinatorial criterion. Moreover, stability and semi-stability coincide for $Z$.

The GIT-quotient $I^n_{\mathcal{X}/C} = \mathcal{H}^n/\mathbb G[n]$ is the desired degeneration; it is projective over $C$, and one checks that it contains $\mathrm{Hilb}^n(\mathcal{X}^{sm}/C)$ as an open subscheme whose complement has codimension $2$. It also has good geometric properties. In \cite{GHHZ21}, we showed among other things that $I^n_{\mathcal{X}/C} \to C $ is a dlt model, and that the dual complex of its special fiber is the $n$-th symmetric product of the dual complex of $\mathcal{X}_0$. 

Independently, Nagai developed his approach further and also showed that his degenerations and ours coincide, see \cite{Na18}, \cite{Nagai-GHH}.

All of these constructions start with type II degenerations of K3 surfaces and again yield type II degenerations of hyperk\"ahler varieties.
Type III degenerations of Hilbert schemes were first considered successfully by C.~Tschanz. The local construction 
is discussed in work of Tschanz  \cite{Tsch23}, \cite{Tsch24}, whereas the global picture was completed by Shafi and Tschanz in \cite{ST25}.

\subsection{Goal of this paper}
Until now, all known constructions of degenerations of hyperk\"ahler varieties in dimension greater than two have been of $K3^{[n]}$-type. 
The goal of this paper is to improve this situation and to study degenerations of generalized Kummer varieties. 
Here we shall address two aspects of this problem. First, we shall give an explicit construction of degenerations of generalized Kummer varieties taking as input a type II degeneration of abelian surfaces.
This works in great generality. Second, we will 
illustrate our approach by a detailed analysis of the geometry of our degenerations in the case of $n=2$ and $n=3$, namely Kummer surfaces and 
generalized Kummer $4$-folds. We shall see that in the latter case a number of new phenomena appear.

More precisely, we construct degenerations, compatible with the Hilbert scheme degenerations outlined above. Indeed, if we take $f \colon \mathcal{X} \to C$ to be a simple degeneration of abelian surfaces, then $ \mathcal{X} \setminus \mathcal{X}_0$ is an abelian scheme over $C \setminus 0$ and we can form 
the relative generalized Kummer variety
$K^{n-1}_{\circ} = \mathrm{Kum}^{n-1}(\mathcal{X} \setminus \mathcal{X}_0) \to C \setminus 0$. This is naturally a closed subscheme of the relative Hilbert scheme $\mathrm{Hilb}^{n}(\mathcal{X} \setminus \mathcal{X}_0) \to C \setminus 0$, which is in turn a dense open subscheme of $I^n_{\mathcal{X}/C}$, obtained via restriction along the open inclusion $C \setminus 0 \subset C$. The closure $K^{n-1}_{\mathcal{X}/C}$ of $K^{n-1}_{\circ}$ inside $I^n_{\mathcal{X}/C}$ therefore yields a canonical way to degenerate the family of generalized Kummer varieties over $C \setminus 0$. The degeneration $K^{n-1}_{\mathcal{X}/C} \to C$ is indeed the main object of study of this paper.

To study $K^{n-1}_{\mathcal{X}/C}$, we view it as a GIT quotient $\mathcal{K}^{n-1}/\mathbb G[n]$, where the so-called \emph{Kummer locus} $\mathcal{K}^{n-1} $ denotes the reduced preimage of $K^{n-1}_{\mathcal{X}/C}$ under the quotient map $\mathcal{H}^n \to I^n_{\mathcal{X}/C}$. In this way, we reduce the problem to understanding the closed subvariety $\mathcal{K}^{n-1}$ of the GIT stable locus $\mathcal{H}^n$, where the latter has been extensively studied in our previous work. Nevertheless, it requires a substantial amount of work to describe the Kummer locus precisely. As an indication that this task is somewhat involved, we mention, for instance, that $\mathcal{K}^{n-1}$ is not flat over $C[n]$ (when $n>2$). Another interesting, and in fact related, aspect is the following phenomenon. For the families we take as input, it will always be the case that 
also the relative smooth locus $\mathcal{X}^{sm}/C$ has a group structure -- in fact, it is the N\'eron model of its generic fiber. Hence one can define $\mathrm{Kum}^{n-1}(\mathcal{X}^{sm}/C)$. This is an open subscheme of $K^{n-1}_{\mathcal{X}/C}$, but its complement will in general be of codimension \emph{one} (and not of codimension two as in the Hilbert scheme case).

In other words, not all irreducible components in the special fiber of our degeneration are visible at the level of $\mathrm{Kum}^{n-1}(\mathcal{X}_0^{sm})$, but, as we will explain, arise for more subtle reasons.

\subsection{Main results}
Next, we explain some central features of our approach, and we discuss the main results of the paper. Before we begin, let us point out that an important technical assumption (which was not necessary for the Hilbert-scheme case) is that $\mathcal{X} \to C$ is a (projective) Kulikov model of type II, constructed through Mumford's technique of \emph{uniformization} in the generality developed in the work of Faltings-Chai \cite{FaltingsChai} and K\"unnemann \cite{Kunnemann}. This is explained in more detail in Section \ref{sec:KulikovNeron}.

\subsubsection{Kummer locus and its stratification} 
In Section \ref{sec:Kummerlocus}, we describe the Kummer locus $ \mathcal{K}^{n-1} $ in terms of a stratification compatible with the natural stratification of $C[n]$ induced by the coordinate hyperplanes in $\mathbb{A}^{n+1}$. Over the complement $C[n]^*$ of the coordinate hyperplanes, the restriction $ \mathcal{K}^{n-1}_{\circ} $ is simply the pullback $K^{n-1}_{\circ} \times_C C[n]^*$. Any other stratum of $ \mathcal{K}^{n-1} $ lives over a stratum of $C[n]$ where some coordinates, say indexed by $I \subset \{1, \ldots, n+1\}$, are zero and the rest are non-zero. Assuming for the moment that the coordinates indexed by the complement of $I$ are all equal to $1$, we are looking at what Li calls a \emph{standard embedding} $ \tau_I \colon C[b] \to C[n]$ for some $b$. It has the property that $\tau_I^* \mathcal{X}[n] = \mathcal{X}[b]$. Now assume that $S \to \mathcal{K}^{n-1} $ is a map from the spectrum of a complete discrete valuation ring, where the generic point maps to $ \mathcal{K}^{n-1}_{\circ} $ and the special point maps to the fiber over $0 \in C[b]$. Then the composition $S \to C[b] \to \mathbb{A}^{b+1}$ is a parametrized
curve in $\mathbb{A}^{b+1}$. If it is, in fact, contained in a \emph{line} $L$,
we explain in Section \ref{sec:KulikovNeron} that the restriction $\mathcal{X}[b]_S$ is a Kulikov model, and, consequently, the smooth locus $\mathcal{X}[b]^{sm}_S$ is a N\'eron model (and hence a group scheme). It is not hard to see, in this situation, that $S$ maps into $\mathrm{Kum}^{n-1}(\mathcal{X}[b]^{sm}_S)$. The action of $\mathbb G[b]$ acts on the set of lines, and this corresponds to an action on $\mathrm{Kum}^{n-1}(\mathcal{X}[b]^{sm}_S)$. The resulting orbit is precisely a stratum. Obviously $S \to C[b] \to \mathbb{A}^{b+1}$ does not in general factor through a line, and a more elaborate version of the above analysis is needed, see Section \ref{sec:Kummerlocus}.

We now formulate the main theorem of the first part of the paper (see Theorem \ref{theorem:mainstrat}). In the statement, we put $ 1 \leq b \leq n$, and denote by $ \mathbf{m} $ a tuple of non-zero integers encoding the distribution of $n$ points on the fiber $\mathcal{X}[b]^{sm}_0$. We say that a pair $(b,\mathbf{m})$ is \emph{unobstructed} if the corresponding connected component $ \mathscr{S}^{\mathbf{m}} $ of $\mathrm{Hilb}^{n}(\mathcal{X}[b]^{sm}_0)$ actually supports part of the Kummer locus. This can in fact be checked by a numerical condition (Definition \ref{def:admisstuple}).

\begin{theorem}
The boundary of the Kummer locus admits the stratification
$$ \mathcal{K}^{n-1} \setminus \mathcal{K}_{\circ}^{n-1} = \bigcup_{b, \mathbf{m}} \mathcal{K}(b, \mathbf{m}), $$
where $(b, \mathbf{m})$ runs over all unobstructed pairs and where each stratum $\mathcal{K}(b, \mathbf{m})$ is smooth, irreducible and $\mathbb G[n]$-stable.    
\end{theorem}

We can give more specific information about the strata $\mathcal{K}(b, \mathbf{m})$. In fact, they are of the form $ O(b, \mathbf{m}) \times \mathbb G[n-b]$, where the second factor is a torus of rank $n-b$. The first factor is obtained from an addition map $\mu \colon \mathscr{S}^{\mathbf{m}} \to G^{\circ}$, where the target is a semi-abelian surface with $1$-dimensional toric part $T$. More precisely, $ O(b, \mathbf{m}) $ is either $\mu^{-1}(e_{G^{\circ}})$ or $\mu^{-1}(T)$, see Proposition-Definition \ref{prop:stratumdfn}. 

\subsubsection{Combinatorial analysis} 

The second part of the paper is devoted to combinatorial and geometric understanding of our degenerations. As a combinatorial tool to understand the strata, and the intersections of their closures, we introduce the notation called \emph{line charts}, which is an enhancement of the notion of ``expanded graphs'' introduced in \cite[Section 1.3.1]{GHH}.

Indeed, each stratum in the relative Hilbert scheme $I^n_{\mathcal{Y}/C}$ is associated with a \emph{line chart} to display intuitively how the $n$ points are distributed on different components of the expansions. It detects whether a certain stratum is unobstructed or not, hence tells whether such a stratum survives in $K^{n-1}_{\mathcal{Y}/C}$. This is presented in detail in Sections \ref{subsec:deepeststrata} and \ref{subsec:higherdimstrata}. From this information we will be able to count the number of the deepest strata in $K^{n-1}_{\mathcal{Y}/C}$, which is equal to the number of cells in top dimension in the dual complex; see Proposition \ref{prop:number-deepest}. In the case of $n=3$, we can even give a full computation for the number of strata in each dimension; see Proposition \ref{prop:number}.

Moreover, we will also see that the so-called \emph{line subcharts} correspond to smoothings of a certain stratum, which further tells the boundary relation for cells in the dual complex of $I^n_{\mathcal{Y}/C}$ and $K^{n-1}_{\mathcal{Y}/C}$. This is presented in Subsection \ref{subsec:numerical-smoothing}x. In particular, by classifying all smoothings of an arbitrary deepest stratum, we obtain the closure of the corresponding top-dimensional cell in the dual complex, which becomes a building block of the dual complex for the global degeneration. This point of view further allows us give a local comparison between the dual complexes of $I^n_{\mathcal{Y}/C}$ and $K^{n-1}_{\mathcal{Y}/C}$; more precisely, each cell in the latter dual complex is obtained as the intersection of a hyperplane and the corresponding cell in the former; see Proposition \ref{prop:from-H-to-K}.

\subsubsection{} We use the tools developed in Section \ref{Sec:linechart} to study in detail our degenerations for $n=2$ and $n=3$. When $n=2$ the generic fiber is a Kummer surface, and some features of the Kummer locus are different compared to when $n > 2$. In particular, we mention that $\mathcal{K}^1$ is smooth over $C[2]$ and that its fiber over the origin $0 \in C[2]$ is \emph{empty}.

Our main result for Kummer surfaces is as follows (Theorem \ref{theorem:Kummerdeg}).
\begin{theorem}
The GIT quotient $ K^1_{\mathcal{Y}/C} $ is a projective Kulikov model over the pointed curve $C$.    
\end{theorem}
We remark that degenerations of Kummer surfaces had already been studied by Overkamp \cite{Overkamp}. He also used uniformization techniques, but allowed also the type III case. 

In contrast, the case $n=3$ shows several new features which are not visible when discussing Kummer surfaces. We already proved in Section \ref{Sec:linechart} that the dual complex is built up from $2$-simplices, which were classified into four fundamental types. Describing how these $2$-simplices glue is a somewhat intricate task, essentially of a combinatorial nature. This is treated in detail in Section \ref{sec:dualcplxIII} where we
give a complete description of the dual complex. This finally allows us to deduce the following result (Theorem \ref{thm:dual-complex-n3}):

\begin{theorem}
The dual complex associated to $ K^2_{\mathcal{Y}/C} $ is PL-homeomorphic to a standard $2$-simplex, in particular, it is $2$-dimensional.    
\end{theorem}

We would like to mention that this result fits well with the work of Brown and Mazzon in \cite{BrownMazzon}. There, they proved that if $A$ is an abelian surface over a complete discretely valued field $K$, then the so-called essential skeleton of the Berkovich analytification of $\mathrm{Kum}^n(A)$ is $PL$-homeomorphic to the standard $n$-simplex for any $n$. We remark that the essential skeleton coincides with the dual complex of any minimal $dlt$-model of $X$ over the ring of integers of $K$.

For both, $n=2$ and $n=3$, we describe the birational geometry of the irreducible components of the special fiber of our degeneration. This is done in Section \ref{sec:stratageom}.

\subsection{Open questions and future directions} 
At the end of this introduction, we would like to mention some interesting questions arising from our work, that we would like to investigate in the future. First, we assume that $\mathcal{Y}/C$ is a type II Kulikov degeneration of abelian surfaces satisfying Assumption \ref{assum:Uniformization}, so that all results in the paper are applicable. We write $ K^{n-1} = K^{n-1}_{\mathcal{Y}/C}$ for simplicity. Clearly, it is natural to ask what geometric properties the total space of our degeneration $K^{n-1}\to C$ has. The methods in this paper are effective in describing the stratification of the special fiber of $ K^{n-1} $, but it would be nice to have a better understanding of the pair $(K^{n-1}, (K^{n-1})_0)$. We expect that it is \emph{dlt} for $n=3$, whereas for $n>3$ this can not hold in general by Example \ref{eg:4d-example}. It seems reasonable that this could be amended by some blowup of the Kummer locus (or related objects), but this needs to be studied more systematically. 

It would also be interesting to have a direct link between the dual complexes of the degenerations $I^n_{\mathcal{Y}/C}$ and $K^{n-1}_{\mathcal{Y}/C}$ over $C$. A natural guess is that there is an involution on the former, whose invariant locus yields (at least up to some modification) the latter. 

The assumption that $\mathcal{Y}/C$ is constructed through uniformization (Assumption \ref{assum:Uniformization}) is probably not necessary. However, for certain computational aspects (see Lemma \ref{lem:fundpropuniformization}), the technique of uniformization provides hands-on control that is crucial for our construction, and it is currently not clear to us how to avoid it. 

Lastly, a natural problem is to extend the results in this paper to the case where the degeneration $\mathcal{Y}/C$ of abelian surfaces is Kulikov of type III. We remark that all results concerning uniformization, as outlined in \ref{subsec:Uniformization}, are equally valid in the type III case.

\subsection{Structure of paper} The article is structured as follows. In Section \ref{sec:prelandnot}, we recall some fundamental properties of Li's expanded degenerations of a simple degeneration, as well as the analysis of GIT stability on the $n$-th Hilbert scheme of the $n$-th expansion. We also establish some useful results on (partial) resolutions of finite base changes of simple degenerations. In Section \ref{sec:KulikovNeron}, we discuss Kulikov models and N\'eron models of $1$-parameter families of abelian varieties. We explain that Kulikov models can always be found using Mumford's technique of \emph{uniformization}, which also guarantees other properties useful for our purposes. In Section \ref{sec:expandedNeron}, we explain how certain restrictions of expanded degenerations of families of abelian varieties are related to Kulikov and N\'eron models. In Section \ref{sec:Kummerlocus}, we define the Kummer locus and study its properties. We explain that it admits a partition into smooth, irreducible and stable parts. In Section \ref{Sec:linechart}, we develop the combinatorial framework for studying the degeneration obtained by taking the GIT quotient of the Kummer locus. In particular, we introduce the notion of line charts. In Section \ref{subsec:Kummersurfaces}, we discuss the case of Kummer surfaces, and we demonstrate that our degeneration is a Kulikov model. In Section \ref{sec:dualcplxIII} we study the case when $n=3$ in some detail. We compute the dual complex of our degeneration explicitly, and show that it is PL-homeomorphic to the standard $2$-simplex. In Section \ref{sec:stratageom}, we study the geometry of the strata of the boundary of our degeneration, when $n=2$ or $3$. In Appendix \ref{appendix-new}, we compute the closure of each stratum in the boundary of the Kummer locus for $n=3$.

\subsection{Convention}
Throughout the paper we will work over an algebraically closed field $k$ of characteristic $0$. 

\subsection{Acknowledgments} 
    The first author is partially supported by the PRIN2022 project 2022PEKBJ: Symplectic varieties: their interplay with Fano manifolds and derived categories. The second author was partially supported by DFG grant Hu 337/7-2. The third author is partially supported by National Natural Science Foundation of China (Grant No.~12371046)

\section{Preliminaries}\label{sec:prelandnot}
In this section, we recall the construction of J.~Li's \emph{expanded 
degenerations}, first introduced in \cite{Li01}, and prove a few properties fundamental for the applications in later parts of the paper.

\subsection{Expanded degenerations}\label{subsec:expandeddegen}
The main reference for the material presented in this subsection is \cite{GHH}. The reader is advised to consult \emph{loc. cit.} for further details.

\subsubsection{}\label{subsubsec:strictsimpledegen}

Let $0 \in C$ denote a smooth curve with a distinguished point. 

\begin{definition}\label{def:simpledegen}
A flat morphism $ f \colon \mathcal{X} \to C$ from a smooth algebraic space is called a strict simple degeneration if
\begin{enumerate}
    \item The map $f$ is smooth outside $\mathcal{X}_0 = f^{-1}(0)$,
    \item The central fiber $\mathcal{X}_0$ is reduced, has normal crossing singularities and the singular locus $\mathcal{D} \subset \mathcal{X}_0$ is smooth,
    \item All irreducible components of $\mathcal{X}_0$ are smooth.
\end{enumerate}
\end{definition}
In particular, a strict simple degeneration is a strict semi-stable degeneration. The condition on the singular locus ensures that at most two components can intersect in a point, hence the combinatorics of the irreducible components and their intersections can be encoded in the dual graph $\Gamma(\mathcal{X}_0)$.

In this paper, we additionally assume that $C$ is affine and connected, that $f$ is projective, and, moreover, that $\Gamma(\mathcal{X}_0)$ is \emph{bipartite}. This means that we can write $\mathcal{X}_0 = Y \cup Y'$, where $Y$, resp.~$Y'$, is a non-empty disjoint union of irreducible components. Let $[Y]$, resp. $[Y']$, denote the set of vertices in the dual complex which correspond to the components of $Y$, resp. $Y'$. Being bipartite is equivalent to the existence of an orientation on $\Gamma(\mathcal{X}_0)$ such that all arrows are oriented from $[Y]$ to $[Y']$ (or the other way around).

\subsubsection{}

Fix an \'etale morphism $ C \to \mathbb{A}^1 = \mathrm{Spec}(k[t])$. Shrinking $C$ if necessary, we can assume that $0 \in C$ is the unique pre-image of the origin $t=0$. We frequently view $\mathcal{X}$ as an $\mathbb{A}^1$-scheme via this map. For any $ n \geq 0$, let $ \mathbb{A}^{n+1} \to \mathbb{A}^1 $ be the map corresponding to 
$$ k[t] \to k[t_1, \ldots, t_{n+1}];~~~~t \to t_1 \cdot \ldots \cdot t_{n+1}. $$
We set $C[n] = \mathbb{A}^{n+1} \times_{\mathbb{A}^1} C $. We briefly explain how Li's $n$-th expanded degeneration of $ f[n] \colon \mathcal{X}[n] \to C[n] $ is constructed. 

The scheme $ \mathcal{X}[1] $ is defined as the blowup 
of $\mathcal{X} \times_{\mathbb{A}^1} \mathbb{A}^2 $ in the Weil divisor $ Y \times V(t_2) $, and composition with projection onto the first factor yields $p_1 \colon \mathcal{X}[1] \to \mathcal{X}$. Assume by induction that $\mathcal{X}[n-1] \to C[n-1]$ and $p_{n-1} \colon \mathcal{X}[n-1] \to \mathcal{X} $ have been constructed. Let $ \mathbb{A}^{n+1} \to \mathbb{A}^n$ be defined by multiplication of the last two coordinates. Then we let $\mathcal{X}[n]$ be the blowup of $\mathcal{X}[n-1] \times_{\mathbb{A}^n} \mathbb{A}^{n+1}$ in the Weil divisor $p_{n-1}^{-1}(Y) \times V(t_{n+1})$, and composing the projection onto the first factor with $p_{n-1}$ yields $p_n$. One moreover checks that $\mathcal{X}[n] \to \mathbb{A}^{n+1} $ factors canonically through a \emph{projective} morphism $ f[n] \colon \mathcal{X}[n] \to C[n]$. 

\subsubsection{} We denote by $0 \in C[n]$ the unique preimage of the origin in $\mathbb{A}^{n+1}$. The fiber of $f[n] \colon \mathcal{X}[n] \to C[n]$ over $0$ can be described as follows. If we replace $\mathcal{X}$ by the smooth locus $\mathcal{X}^{sm}$ of $f$,
the map $ \mathcal{X}[n] \to \mathcal{X} $ is just base change along $C[n] \to C$. However, each component $D$ of the singular locus of $\mathcal{X}_0$ gets replaced by a chain of $\mathbb{P}^1$-bundles $\Delta^1_D, \ldots, \Delta^n_D$ over $D$, intersecting along the sections at $0$ and $\infty$. For a fixed $i$, we write $\Delta^i$ for the (disjoint) union of the $\Delta^i_D$, as $D$ runs over the irreducible components of the singular locus $\mathcal{D}$.

\subsubsection{}

It is a fundamental feature of expanded degenerations that the torus $\mathbb G[n] = (\mathbb{G}_m)^n $ acts equivariantly on $\mathcal{X}[n] \to C[n]$. The action on $C[n]$ is simply the base change of the action on $ \mathbb{A}^{n+1} $ given by 
$$ (\sigma_1, \ldots, \sigma_n) * (t_1, \ldots, t_{n+1}) = (\sigma_1 t_1, \sigma_2 \sigma_1^{-1} t_2, \ldots, \sigma_n^{-1} t_{n+1}). $$
The construction of the action on $\mathcal{X}[n]$ is more involved, here we only explain how $\mathbb G[n]$ acts in the fibers. If we let $(u_i:v_i)$ denote homogeneous coordinates for the $\mathbb{P}^1$-fibers of $\Delta^i$, the action is described by
$$ (\sigma_1, \ldots, \sigma_n) * (u_i:v_i) = (\sigma_i u_i:v_i). $$

\subsubsection{}\label{subsubsec:GITstability}
The choice of a relatively very ample line bundle $\mathcal{L}$ on $\mathcal{X}$ gives rise to a relatively ample line bundle $\mathcal{M}$ on $\mathrm{Hilb}^n(\mathcal{X}[n]/C[n])$, see \cite[Section 2.2.1]{GHH}. One can therefore define the (semi)stable locus for the action of $\mathbb G[n]$ on the Hilbert scheme, induced by the action on $f[n]$. It turns out that the (semi)stable locus does not depend directly on $\mathcal{L}$, and that it admits a combinatorial description, which we recall next. 

Let $Z \in \mathrm{Hilb}^n(\mathcal{X}[n]/C[n])$ correspond to a subscheme $Z \subset \mathcal{X}[n]_q$ for some $q \in C[n]$. We write
$$ \{a_1, \ldots, a_r\} = \{i \mid t_i(q) = 0 \}, $$
and moreover set $a_0 = 1$ and $a_{r+1} = n+1$. Then the \emph{numerical support} of $[Z]$ is defined as the vector $\mathbf{v}_{\mathbf{a}}$ whose $i$-th component is $a_i - a_{i-1}$, for $ 1 \leq i \leq r+1 $.

One says that $Z$ has \emph{smooth support} if it is a subscheme of the smooth locus $\mathcal{X}[n]^{sm}$ of $f[n]$. In this case, each $P \in \mathrm{Supp}(Z)$ belongs to $\Delta^{a_{i(P)}}$ for a \emph{unique} integer $ 0 \leq i(P) \leq r$. The \emph{combinatorial support} of $Z$ is defined as 
$$ \mathbf{v}(Z) = \sum_P n_P \mathbf{e}_{i(P)} \in \mathbb{Z}^{r+1}, $$
where $\mathbf{e}_{j} = (0, \ldots, 1, \ldots, 0)$ denotes the $j$-th standard basis vector.

\begin{theorem}\label{theorem:stabilitycondition}
A point $Z \in \mathrm{Hilb}^n(\mathcal{X}[n]/C[n])$ is stable iff it has smooth support and its numerical support coincides with its combinatorial support, i.e., if $\mathbf{v}(Z) = \mathbf{v}_a$. A point is moreover stable if it is semi-stable.  
\end{theorem}

\begin{remark} A subscheme $Z$ in the fiber $\mathcal{X}[n]_0$ is stable if and only if it has smooth support, and $Z \cap \Delta^i$ is non-empty for each $i = 1, \ldots, n$. The combinatorial criterion in less degenerate fibers can be seen as a \emph{smoothing} of this.

To illustrate, assume that $\mathcal{X}_0$ consists of two irreducible components intersecting along $D$. Then $\mathcal{X}[1]_0$ is a chain of three components $\Delta^0, \Delta^1, \Delta^2$ where $\Delta^1$ is a $\mathbb{P}^1$-bundle over $D$ (the unique inserted component) and where $\Delta^0$, resp.~$\Delta^2$, intersects along the $0$-section, resp.~$\infty$-section. A point $P \in \mathcal{X}[1]_0$ is stable if and only if $ P \in \Delta^1 $, away from the double loci. The coordinate $t_1$ is a smoothing parameter for $\Delta^0 \cup \Delta^1$, and a point in a fiber $\mathcal{X}[1]_{(*, 0)}$, where $ (*,0) \in C[1] \cap \{ t_1 \neq 0, t_2 = 0\}$, is stable if and only if it is supported on the smoothing, away from $\Delta^2$. Likewise, $t_2$ is a smoothing parameter for $\Delta^1 \cup \Delta^2$, and a point in a fiber $\mathcal{X}[1]_{(0,*)}$ is stable if and only if it is supported on the smoothing, away from $\Delta^0$. 
\end{remark}

We denote by $\mathcal{H}^n = \mathrm{Hilb}^n(\mathcal{X}[n]/C[n])^{st}(\mathcal{M})$ the stable locus, and by
$$ I^n_{X/C} = \mathcal{H}^n/\mathbb G[n] $$
the GIT quotient. By \cite{GHHZ21}, it is a proper dlt-model when the relative dimension of $\mathcal{X} \to C$ is at most $2$, in which case it also forms a relative compactification of $\mathrm{Hilb}^n(\mathcal{X}^{sm}/C)$, whose complement is closed of codimension $2$.

\subsubsection{}\label{subsubseq:standardemb}

We recall some useful facts from \cite[Section 1]{Li01} and \cite[Section 2]{Li13} on so-called \emph{standard embeddings}. Let $ I \subset \{1, \ldots, n+1 \}$ be a subset of cardinality $m+1$, and let $I^{\circ}$ denote its complement. We denote by $ \mathbb{A}^{n+1}_{U(I)}$ the open subset of $ \mathbb{A}^{n+1} $ defined by requiring that $t_i \neq 0 $ for each $i \in I^{\circ}$. 

There is an isomorphism
$$ \tilde{\tau}_I \colon \mathbb{A}^{m+1} \times \mathbb G[n-m] \to \mathbb{A}^{n+1}_{U(I)} $$
given by the (order-preserving) map sending the coordinates of $\mathbb{A}^{m+1}$ to the subset of coordinates of $\mathbb{A}^{n+1}$ indexed by $I$, and the coordinates of $\mathbb G[n-m]$ to the subset of coordinates indexed by $I^{\circ}$. Restricting $ \tilde{\tau}_I$ to the identity element of the torus yields the \emph{standard embedding}
$$ \tau_I \colon \mathbb{A}^{m+1} \to \mathbb{A}^{n+1}. $$
Similar maps are defined for $C[m]$ and $C[n]$ by pullback, these will still be denoted by $ \tilde{\tau}_I $ and $ \tau_I $ respectively. 

The following properties will be particularly important for us. The standard embedding $ \tau_I $ induces an isomorphism (over $C[m]$)
\begin{equation}\label{equation:standardembeq} 
(\tau_I^*\mathcal{X}[n], \tau_I^*p_n) \cong (\mathcal{X}[m], p_m). 
\end{equation}
The inverse of $ \tilde{\tau}_I $ gives rise to a map $ \iota_I \colon C[n]_{U(I)} \to C[m] $ such that 
\begin{equation}\label{equation:pullbackiso}
    \mathcal{X}[n] \vert_{C[n]_{U(I)}} \cong \iota_I^{*} \mathcal{X}[m]
\end{equation}
compatible with projection to $\mathcal{X}$ (see \cite[Corollary 1.7]{Li01}).
\subsubsection{}
Let $0 \in C[m]$ denote the origin, and put $ \mathcal{U}(I)_0 = \iota_I^{-1}(0)$. As a consequence of (\ref{equation:pullbackiso}), we see that
\begin{equation}\label{equation:restrictionsplitting}
\mathcal{X}[n] \vert_{\mathcal{U}(I)_0} = \mathcal{X}[m]_0 \times \mathcal{U}(I)_0. 
\end{equation}
This splitting is moreover compatible with the torus action, as we now explain. In fact, recall that we can identify $\mathbb G[n]$ with the subtorus of elements in $\mathbb{G}_m^{n+1}$ multiplying to $1$. Then the permutation of the $n+1$ factors implicitly defined by the construction of $ \tilde{\tau}_I $ yields an isomorphism
$$ \kappa \colon \mathbb{G}_m^{n+1} \to \mathbb{G}_m^{n+1}, $$
restricting to an isomorphism of the rank $n$ subtori of elements multiplying to $1$. Then we have $\mathbb G[n] = \mathbb G[m]' \times \mathbb G[n-m]'$, where $\mathbb G[m]'$, resp.~$\mathbb G[n-m]'$, is the pre-image under $\kappa$ of the product of the $m$ first factors, resp.~$n-m$ last factors. Here $\mathbb G[m]'$ is the torus acting on $\mathcal{X}[m]$ in (\ref{equation:standardembeq}) whereas in (\ref{equation:restrictionsplitting}) $\mathbb G[m]'$ and $\mathbb G[n-m]'$ act on the first and second factor respectively.

\subsection{Normal forms of maps to $C[n]$}
Let $R = k[[T]]$, and consider a map 
$$ \alpha[n] \colon S = \mathrm{Spec}(R) \to C[n]. $$
In this subsection, we shall show that there is an element $\sigma \in \mathbb G[n](R)$ such that $\alpha[n]^{\sigma}$ has a particularly simple shape.

\subsubsection{}
For simplicity, we denote by $\alpha[n]$ also the induced map $ S \to \mathbb{A}^{n+1}$. We assume that this is given by $ t_i \mapsto d_i T^{a_i}$, with $d_i$ a unit and $a_i \geq 0$, where moreover $\sum_{i=1}^{n+1} a_i > 0$. Note that by Hensel's lemma we can find a unit $d \in k[[T]]$ such that 
$$ d^{a_1 + \ldots + a_{n+1}} = d_1 \cdot \ldots \cdot d_{n+1}. $$

\begin{lemma}\label{lemma:standardformmaptobase}
Define $ \sigma = (\sigma_1, \ldots, \sigma_n) \in \mathbb G[n](R)$ by 
$$ \sigma_l = \prod_{i=1}^{l} \frac{d^{a_i}}{d_i}. $$
Then $ \alpha[n]^{\sigma} \colon S \to \mathbb{A}^{n+1} $ is the map given by
$$ t_i \mapsto d^{a_i} T^{a_i} = (dT)^{a_i}$$
for every $i$.
\end{lemma}
\begin{proof}
This follows from elementary computations, we omit the details.   
\end{proof}

Writing $T' = d T$, the map $\alpha^{\sigma}$ is described by $t_i \mapsto (T')^{a_i}$. Note that $T'$ is again a uniformizer of $R$ and that $R = k[[T']]$ as well. In other words, after acting by an $R$-valued element of the torus $\mathbb G[n]$, we can assume that the coordinates of the map are pure powers of the uniformizer of $R$.

\subsubsection{}
Let $ I \subset \{1, \ldots, n+1\} $ be the subset such that $ i \in I $ if and only if $a_i > 0$. Let $ \vert I \vert = b+1 $. Then $\alpha[n]^{\sigma}$ factors through the standard embedding 
$$ \tau_I \colon \mathbb{A}^{b+1} \to \mathbb{A}^{n+1}. $$

Since the compositions of $\alpha[n]$ and $\alpha[n]^{\sigma}$ with the multiplication map $ \mathbb{A}^{n+1} \to \mathbb{A}^{1}$ are equal, $\alpha[n]^{\sigma}$ lifts to a map $S \to C[n]$, which, similarly as above, factors through $\tau_I \colon C[b] \to C[n]$.

\subsection{Generic lines}\label{subsec:genlinesbasics}

\subsubsection{}
By a \emph{generic line} $L$ in $\mathbb{A}^{n+1}$ we mean a $1$-dimensional linear subspace not contained in any coordinate hyperplane $t_i=0$. It can be represented by a basis vector $(a_1, \ldots, a_{n+1})$, where each $a_i \in k^*$. We will refer to the generic line corresponding to $(1, \ldots, 1)$ as the \emph{diagonal} line.

\subsubsection{}

If $\sigma \in \mathbb G[n](k)$, and $L$ is the generic line corresponding to $(a_1, \ldots, a_{n+1})$ then $\sigma $ maps $L$ to the generic line $L^{\sigma}$ represented by $ (\sigma_1 a_1, \sigma_2 \sigma_1^{-1} a_2, \ldots, \sigma_n^{-1} a_{n+1})$. This action is \emph{transitive} on the set of generic lines.

Indeed, let $(a_1, \ldots, a_{n+1})$ and $(b_1, \ldots, b_{n+1})$ be two tuples of nonzero elements of $k$. Then, if we define $\sigma = (\sigma_1, \ldots, \sigma_n) \in \mathbb G[n](k)$ by
$$ \sigma_l = x^l \prod_{i = 1}^l \frac{b_i}{a_i} $$ 
for every $ l \in \{1, \ldots, n\}$, it is straightforward to verify that 
$$ (\sigma_1, \ldots, \sigma_n) \cdot (a_1, \ldots , a_{n+1}) = x (b_1, \ldots, b_{n+1}), $$
where
$$ x = (\prod_{i = 1}^{n+1} \frac{a_i}{b_i})^{\frac{1}{n+1}} \in k^*. $$

\subsubsection{}

Let $R = k[[T]]$, and assume we are given a map
$$\alpha[n] \colon S \to \mathbb{A}^{n+1} $$ 
where $t_i \mapsto a_i T$, with $a_i$ a unit in $R$ for every $i$. By Lemma \ref{lemma:standardformmaptobase}, we can find $\sigma \in \mathbb G[n](R)$ such that $\alpha[n]^{\sigma}$ is the map $ t_i \mapsto a T$ for some unit $a \in R $. Hence it factors through the diagonal line in $ \mathbb{A}^{n+1} $.

\subsubsection{} 
Let $ L \subset \mathbb{A}^{n+1} $ be a generic line. Then the preimage of $L$ under the \'etale map $C[n] \to \mathbb{A}^{n+1}$ is a smooth curve, and we shall also call the connected component containing the (unique) preimage of $0 \in \mathbb{A}^{n+1}$ a generic line. The torus $\mathbb G[n]$ acts transitively on the set of generic lines in $C[n]$, since the preimage of $0 \in \mathbb{A}^{n+1}$ is preserved by each $\sigma \in \mathbb G[n]$. 

\subsection{Restriction of expansions and blowing up}\label{subsec:restrictionblowup}

Let $\mathcal{X} \to C$ denote a strict simple degeneration. We keep the notation and assumptions from Subsection \ref{subsec:expandeddegen}.

\subsubsection{}
Let $R = k[[\pi]] $ and assume that a non-constant map  $ \alpha \colon S = \mathrm{Spec}(R) \to C $ is given. Replacing $\pi$ with another uniformizer if necessary, we can assume that the composition of $\alpha$ with $ C \to \mathbb{A}^1$ is given by $ t \mapsto \pi^r $ for some $r > 0$. 

Let $b \geq 1$ and $r_i > 0$, for $i=1, \ldots, b+1$, be integers such that $ \sum_i r_i = r $. Then $S \to \mathbb{A}^{1}$ factors through the map $S \to \mathbb{A}^{b+1}$ given by $ t_i \to \pi^{r_i} $.

Together with $\alpha$ this latter map induces
$$ \alpha[b] \colon S \to C[b]. $$ 
The pull back of $\mathcal{X}[b]$ along $\alpha[b]$ will be denoted $ \mathcal{X}[b]_S $ if no confusion can occur.

\subsubsection{}

The canonical morphism $ \mathcal{X}[b] \to \mathcal{X} $ yields $ \mathcal{X}[b] \to \mathcal{X} \times_C C[b]  $. Pulling back along $S \to C[b]$, we obtain 
$$ \mathcal{X}[b]_S \to \mathcal{X}_S. $$ 
Below, we shall, among other things, show that this map is a composition of blowing up maps. To do this, we define a series of blowups
$$ \mathcal{X}_{b+1} \to \ldots \to \mathcal{X}_2 \to \mathcal{X}_1 $$ 
where $ \mathcal{X}_1 = \mathcal{X}_S $.

Let us put $ a_k = r - \sum_{i=1}^k r_i $, and denote by $Y_1$ the inverse image of $Y$ under $ \mathcal{X}_S \to \mathcal{X} $. We define $ \mathcal{X}_2 \to \mathcal{X}_1 $ to be the blowup in the center $ I(Y_1) + (\pi^{a_1}) $. Let $Y_2$ be the closed subscheme of $ \mathcal{X}_2 $ defined as the pre-image of $Y$ under the map $ \mathcal{X}_2 \to \mathcal{X} $. By induction, assume that $ Y_i \subset \mathcal{X}_i $ has been constructed. Then we define $ \mathcal{X}_{i+1} \to \mathcal{X}_i $ to be the blowup with center $ I(Y_i) + (\pi^{a_i}) $, and we let $Y_{i+1}$ be the pre-image of $Y$ under the induced map $ \mathcal{X}_{i+1} \to \mathcal{X} $.

\subsubsection{}

Alternatively, we can lift $\alpha \colon S \to C $ to $C[r-1]$ by setting $t_i' \mapsto \pi $ for each $ 1 \leq i \leq r $. We will see that the resulting $ \mathcal{X}_{r}' = \mathcal{X}[r-1]_S$ is a strict simple degeneration over $S$. Let moreover $C[r-1] \to C[b]$ be defined by 
$$ t_i \mapsto \prod_{l=1}^{r_i} t'_{r_1 + \ldots + r_{i-1} + l}. $$
Then $\alpha[b]$ factors via $\alpha[r-1]$, and in fact we have a natural map $ \mathcal{X}_{r}' \to \mathcal{X}_{b+1}$ that will be described in more detail below.

\subsubsection{}
The facts we shall use in the sequel are gathered in Proposition \ref{prop:main-blow-up}. We first state a lemma which will be useful for the proof. Consider the ring $ B = R[x,y]/(\pi^r - xy) $. If $r=1$, $B$ is regular, so we assume $r > 1$ and let $ a < r $ be a strictly positive integer. Since the elements $ y, \pi^a $ form a regular sequence in $R[x,y]$, it is easy to describe the blowup of the spectrum of $B$ in the ideal $(y, \pi^a)$. 

\begin{lemma}\label{lemma:basicblowup}
The blowup of $ \mathrm{Spec}(B)$ in $ (y, \pi^a) $ is the the projective spectrum of
$$ B' = R[x,y][U_1,V_1]/(yV_1-\pi^aU_1, \pi^{r-a}V_1-xU_1). $$    
\end{lemma}
\begin{proof}
This is a straightforward computation, we omit the details.   
\end{proof}

The facts we shall use in the sequel are gathered in the following result:

\begin{prop}\label{prop:main-blow-up}
Keep the notation above.
\begin{enumerate}
\item Let $ \mathcal{X}_{b+1} \to \mathcal{X}_S $ be the blowup corresponding to the sequence $r_1, \ldots, r_b$. Then $ \mathcal{X}_{b+1} \cong \mathcal{X}[b]_S $.

\item Let $ \mathcal{X}_{r}' \to \mathcal{X}_S $ be the blowup corresponding to the sequence $ r_i' = 1 $, for $ i = 1, \ldots, m-1 $, where $m = r-1$. Then $ \mathcal{X}_{r}' $ is semi-stable over $S$, and the exceptional locus over each component $D$ of the double locus in the special fiber of $ \mathcal{X}_S $ is a chain of $r-1$ $ \mathbb{P}^1 $-bundles over $D$ (intersecting along sections).

\item The map $ \mathcal{X}_{r}' \to \mathcal{X}_S $ factors uniquely through the blowup $ \mathcal{X}_{b+1} \to \mathcal{X}_S $. Moreover, the induced morphism $ \gamma \colon \mathcal{X}_{r}' \to \mathcal{X}_{b+1} $ restricts to an isomorphism over $ (\mathcal{X}_{b+1})^{\mathrm{sm}} $.

\end{enumerate}
\end{prop}
\begin{proof}
We explain the main ingredients of the proof, leaving some straightforward verifications to the reader. 

First, we pick an open cover of $\mathcal{X}$, similarly as in \cite[1.1.3]{GHH}. In particular, we can arrange this in such a way that any point $P$ in the double locus of $\mathcal{X}_0$ has an open neighborhood $U$ where $U_0$ consists of exactly two irreducible components components which are moreover principal divisors on $U$. Then $U$ admits an \'etale map $ U \to V = \mathrm{Spec} (k[t,x,y]/(t-xy)) $ sending $P$ to the origin $O$ and where $x$ and $y$ pull back to the local equations (with the same names) on $U$. (We omit the other smooth coordinates of $V$ from our notation.)

We will first prove our statements for the local models $V$. Then, by \'etale functoriality of pullback, expanded degenerations \cite[Prop.~1.13]{GHH} and blowups, it will follow that our statements are true also for the open neighborhoods $U$ (compatible with overlaps). Concerning \'etale functoriality of expanded degenerations, we remark that $U_S \to V_S$ is still \'etale, and pulling back $V[b]_S \to V_S$ along this map yields $U[b]_S \to U_S$.

Put $ \mathcal{V}_1 = V_S $. Then we have a blowup tower
$$ \mathcal{V}_{b+1} \to \ldots \to \mathcal{V}_1 $$
where $ \mathcal{V}_{k+1} \to \mathcal{V}_{k} $ is the blowup in $ (y, \pi^{a_k}) $. By repeated use of Lemma \ref{lemma:basicblowup} above, one shows that $ \mathcal{V}_{k+1} $ is the projective spectrum of 
$$ R[x,y][U_1,V_1] \ldots [U_k,V_k] $$
modulo the relations $ \pi^{r_1} V_1 - x U_1 $, $ U_i V_{i-1} - V_i U_{i-1} \pi^{r_i} $ for each $ 1 < i \leq k $, and $ y V_k - \pi^{a_k} U_k $. In particular, we immediately see from \cite[Prop.~1.7]{GHH} that $V[b]_S = \mathcal{V}_{b+1}$. Assertions (1) and (2) follow from the local case.

Concerning Assertion (3), the existence of $ \gamma \colon \mathcal{X}_{r}' \to \mathcal{X}_{b+1} $ is immediate from the universal property of blowups. Indeed, at each step $ \mathcal{X}_{i+1} \to \mathcal{X}_{i} $ we blow up $ I(Y_i) + (\pi^{a_i}) $, but by construction this becomes a Cartier divisor on $ \mathcal{X}_{r}' $. To prove the last part of Assertion (3), it again suffices to treat the local case.

To describe $\gamma \colon \mathcal{V}_{r}' \to \mathcal{V}_{b+1}$, we use the fact that the blowup $ V[b]_S $ is covered by affine charts $ W_k $, $k \in \{1, \ldots, b\} $, where $ W_k $ is determined by the conditions $ U_i \neq 0 $ for $ i \leq k-1 $ and $ V_i \neq 0 $ for $ i \geq k $. Likewise, $\mathcal{V}_{r}'$ is covered by charts that will be denoted $ W_l' $. 

For a fixed $k$ and 
$$ l \in \{r_1 + \ldots + r_{k-1} + 1, \ldots, r_1 + \ldots + r_{k}\}, $$
the map $\gamma$ locally corresponds to the ring homomorphism
$$ A(W_k) = R[u_k, v_{k-1}]/(u_k v_{k-1} - \pi^{r_k}) \to A(W_l') = R[u_l',v_{l-1}']/(u_l' v_{l-1}' - \pi), $$
where $v_{k-1} = \pi^{l - 1 - (r_1 + \ldots + r_{k-1})} \cdot v_{l-1}' $ and $u_k = \pi^{r_1 + \ldots + r_{k} - l} \cdot u_l'$. The last part of the statement is easily derived from this.
\end{proof}

\section{Kulikov models and N\'eron models}\label{sec:KulikovNeron}

Let $f \colon \mathcal{X} \to C$ be a strict simple degeneration as defined in \ref{subsubsec:strictsimpledegen}. For the rest of the paper, we will be concerned with the case when $\mathcal{X} \setminus f^{-1}(0)$ is an abelian scheme over $C \setminus 0$. In this section, we explain that (at least after a finite ramified base change) $f$ can be chosen in such a way that its relative smooth locus is a group scheme. We will also need to make sure that $f$ satisfies a few additional fine properties. Our tool for producing a suitable degeneration is the theory of \emph{uniformization}, in the sence of Mumford and Faltings-Chai.

\subsection{N\'eron models} 
In this subsection we recall some generalities about N\'eron models. We assume that $ R $ is a Dedekind domain with fraction field $K$, and consider an abelian $K$-variety $A$. We write $S = \mathrm{Spec}(R)$.

\subsubsection{}
We  denote the N\'eron model of $A$ by $ \mathcal{A} $. This is a smooth group scheme of finite type over $S$ 
such that $ \mathcal{A} \times_S K \cong A $ and such that, for any smooth $T \to S$, the canonical map
$$ \mathrm{Hom}(T,\mathcal{A}) \to \mathrm{Hom}(T \times_S K, A) $$
is bijective. The latter is often called the \emph{N\'eron mapping property}.

The N\'eron model exists, it is unique up to isomorphism and the group structure of $A$ extends uniquely to $ \mathcal{A} $. We refer to \cite{neron} for details. 

\subsubsection{} The formation of the N\'eron model does not in general commute with base change along flat morphisms $\phi \colon S' \to S$ of Dedekind schemes. However, for our purposes we point out that it does commute when $\phi$ either is \'etale, or if $R$ is local and $\phi$ is the canonical map from the formal completion.

\subsubsection{}
Let $s \in S$ be a closed point. For simplicity we assume the residue field $k(s)$ to be algebraically closed (this will always be the case for us). Then the fiber $ \mathcal{A}_s $ is a smooth group scheme of finite type over $k(s)$. It fits into an exact sequence
\begin{equation}\label{seq:idcpt} 0 \to \mathcal{A}_s^{\circ} \to \mathcal{A}_s \to \Phi(A) \to 0 
\end{equation}
where $\mathcal{A}_s^{\circ}$ is the connected component of identity, and $\Phi(A)$ is a finite abelian group - the \emph{component group}.

\subsubsection{}
We denote by $\mathcal{A}^{\circ}$ the open subscheme of $\mathcal{A}$ characterized by the property that $(\mathcal{A}^{\circ})_s = \mathcal{A}_s^{\circ}$ for every $s \in S$. Replacing, if necessary, $K$ by a finite Galois extension, we can and will make the assumption 
that $A$ has semi-abelian reduction over $S$, i.e. that $\mathcal{A}_s^{\circ}$ is a semi-abelian variety for every closed point $s \in S$. By definition, this means that $\mathcal{A}_s^{\circ}$ is an extension
\begin{equation} 0 \to \mathcal{T}_s \to \mathcal{A}_s^{\circ} \to \mathcal{B}_s \to 0
\end{equation}
of an abelian variety $\mathcal{B}_s$ by a torus $\mathcal{T}_s$. 

\subsubsection{}
Since $\mathcal{A}_s^{\circ}$ is semi-abelian, it is well known that \eqref{seq:idcpt} is split exact (see for instance \cite{LiuLor}). Hence, there exists an isomorphism
$$ \mathcal{A}_s = \mathcal{A}_s^{\circ} \times \Phi(A) $$
of $k(s)$-group schemes. When $\Phi(A)$ is \emph{cyclic}, which will always be the case in this paper, any splitting corresponds to lifting a generator $\gamma$ to a point of $\mathcal{A}_s(k(s))$ of the same order.

\subsubsection{}
Let $ S' \to S $ be a flat map of Dedekind schemes, and let $K'$ be the function field of $S'$. We denote by $ \mathcal{A}_{S'} $ the pullback of $\mathcal{A}$ and by $\mathcal{A}'$ the N\'eron model of $ A' = A \times_K K'$. Then, by the N\'eron mapping property of $\mathcal{A}'$, there is a unique morphism
$$ b_{S'/S} \colon \mathcal{A}_{S'} \to \mathcal{A}' $$
extending the canonical isomorphism on generic fibers. The map $ b_{S'/S} $ is called the \emph{base change} map. It is a homomorphism of $S'$-group schemes, but in general not an \emph{isomorphism}. However when $ \mathcal{A}^{\circ} $ is semi-abelian, it is well known that $ b_{S'/S} $ is an open immersion, and restricts to an isomorphism 
$$ \mathcal{A}^{\circ} \cong (\mathcal{A}')^{\circ}. $$

\subsubsection{}
Assume that $ \mathcal{A} $ is smooth over $S \setminus \{s\}$, and that  
$s' \in S'$ is the unique pre-image of $s$. Then $b_{S'/S}$ induces an inclusion 
$$ \Phi(A) \hookrightarrow \Phi(A') := \mathcal{A}'_{s'}/(\mathcal{A}'_{s'})^{\circ}. $$
The cokernel has cardinality equal to $[K':K]^{\tau}$, where $\tau = \mathrm{dim}(\mathcal{T}_s)$ (see \cite{HaNi-comp}).

\subsubsection{} In this paper, we will always work in the setting where $\mathcal{A}$ is smooth over the complement of a unique point $s$, and that $\mathcal{A}^{\circ}$ is a semi-abelian $S$-scheme.

\subsection{Kulikov models}
Let $A$, $K$ and $S$ be as at the beginning of the previous subsection. We assume in addition that these objects are defined over an algebraically closed field $k$ of characteristic zero. 

\subsubsection{}
Let $ \mathcal{Y}/S $ be an algebraic space which is flat and proper over $S$. We assume that $ \mathcal{Y} $ is a model of $A$, i.e., there exists an isomorphism $ \mathcal{Y} \times_S K \cong A $.

\begin{definition}\label{def:KulikovMod}
We call $ \mathcal{Y}$ a Kulikov model of $A$ over $S$ if:
\begin{enumerate}
\item $ \mathcal{Y} $ is non-singular and, for any closed point $s$, $ \mathcal{Y}_s $ is a reduced snc-divisor.

\item The relative canonical sheaf $ \omega_{\mathcal{Y}/S} $ is trivial.
\end{enumerate}
\end{definition}

Replacing $K$ by a finite separable extension if necessary, any abelian variety $A$ admits a Kulikov model, and we can moreover assume that it is projective over the base (this will be explained in \ref{subsubsec:projectiveKulikov} below). 
Shrinking the base if necessary, we can, and will, assume there is a unique degenerate fiber.

\subsubsection{}

Contrary to the case of the N\'eron model, Kulikov models are not unique in general. However, if $ \mathcal{Y} $ is a Kulikov model of $A$, it is nevertheless closely related to the N\'eron model $ \mathcal{A} $. Namely, let $\mathcal{Y}^{sm}$ denote the relative smooth locus, and let 
$$ \alpha \colon \mathcal{Y}^{sm} \to \mathcal{A} $$
be the unique morphism extending the identity on generic fibers. Then $\alpha$ is an isomorphism. When $R$ is a complete discrete valuation ring, this is proved in \cite[Prop.~5.1.8]{HN}. From this one easily derives that the statement also holds when $S$ is a smooth curve of finite type over $k$, see \cite{CY3}.

\subsection{Type II degenerations}\label{subsec:typeII}
From now on, we assume that $ \delta = \mathrm{dim}(A) \in \{1,2\}$. 

\subsubsection{} The possibilities for the degenerate fiber of a Kulikov model $\mathcal{Y}$ have been classified into $\delta + 1$ distinct  types according to the complexity of the intersections of the components of $\mathcal{Y}_s$ (cf. e.g. \cite{ChiLaz}). If $\mathcal{Y}$ is smooth over $S$, it is called type I. 
If it is not smooth, but $\mathcal{Y}_s$ is without triple intersections, it is called type II. Note that $\mathcal{Y}$ is then also a simple degeneration. (For abelian surfaces, one calls the case with triple intersection type III.) 

\begin{remark}
Even though the Kulikov model is not unique, the \emph{type} of the degeneration is unique.    
\end{remark} 

\subsubsection{}
\label{subsubsection:explanationcycle}
Assume $\mathcal{Y} \to S $ is a type II Kulikov model as above. Then the special fiber 
$$ \mathcal{Y}_s = \sum_{i=0}^{M-1} G_i $$
has the following structure. Each component $G_i$ is a $\mathbb{P}^1$-bundle over an abelian variety $D$ of dimension $\delta - 1$, and the 
components $G_i$ intersect pairwise along sections, forming a cycle. More precisely, 
let $D_{i,0}$ and $D_{i,\infty}$ denote the $0$ and $\infty$ sections of $G_i$. Then $G_i$ is glued to $G_{i+1} $ along $D_{i,\infty}$ and $D_{i+1,0}$ (where we compute $i$ modulo $M$).

For each $i$, the open part $ G_i^{\circ} = G_i \setminus \{ D_{i,0} \cup D_{i, \infty} \}$ is a semi-abelian variety, with extension
\begin{equation}
0 \to T \to G_i^{\circ} \to D \to 0 
\end{equation}
where $T$ is a rank $1$ torus.

\subsubsection{}
Performing a quadratic base change if necessary, we can, and will, assume that $ M = 2N $, i.e., that $\mathcal{Y}_s $ has an even number of components.
This assumption is equivalent to the dual complex $\Delta(\mathcal{Y}_s)$ being bipartite, and in particular allows us to use all results from Section \ref{sec:prelandnot}.
Since we can identify $\mathcal{Y}^{sm}$ with the N\'eron model $ \mathcal{A} $, we shall always choose indices such that $G_0^{\circ}$ is the identity component of $ \mathcal{A}_s $.

\subsection{Uniformization}\label{subsec:Uniformization}

We next explain how one can use Mumford's construction, in the formulation of Faltings-Chai \cite{FaltingsChai} and K\"unnemann \cite{Kunnemann}, to construct a suitable Kulikov model $\mathcal{Y}$ of $A$. We refer to these sources for an exhaustive treatment, and only introduce the notation and properties necessary for our purposes. 

 Unless mentioned otherwise, we assume in this subsection that $S$ is the spectrum of a complete discrete valuation ring $R$ with maximal ideal $\mathfrak{m}$ and residue field $k$. For any $S$-scheme $\mathcal{W}$, we can consider the formal completion with respect to $\mathcal{W}_0 = \mathcal{W} \times_R R/\mathfrak{m}$.

\subsubsection{}

Assume that $A$ admits \emph{some} type II Kulikov model over $S$, as in the previous subsection. As before, we denote the N\'eron model of $A$ by $ \mathcal{A}$. The formal completion of the semi-abelian scheme $\mathcal{A}^{\circ}$ is canonically an extension of a formal abelian scheme by a formal torus. As explained in \cite[Chapter II]{FaltingsChai}, this extension is algebraizable, which gives the so-called \emph{Raynaud extension} 
$$ 0 \to \mathcal{T} \to \widetilde{\mathcal{A}}^{\circ} \to \mathcal{B} \to 0. $$
Here $\mathcal{B}$ is an abelian scheme and $\mathcal{T}$ is a split torus with character lattice $X$. In our situation, the torus has relative dimension one over $S$, hence $X$ is a lattice of rank $1$. There is also a Raynaud extension associated with the dual abelian variety $A^t$, and we denote by $Y$ the character lattice of its torus part. The polarization gives rise to a bilinear map
$$ b \colon X \times Y \to \mathbb{Z}, $$
see \cite[Chapter II, Remark 6.3]{FaltingsChai}.

\subsubsection{}
Following \cite[Chapter III]{FaltingsChai} and \cite[Section 2]{Kunnemann}, we explain how one can produce an explicit Kulikov model $\mathcal{Y}$ (necessarily of type II) by means of Mumford's construction.  
As a first step, one constructs a relatively complete model $\widetilde{\mathcal{Y}}$ of $\widetilde{\mathcal{A}}^{\circ}$. This is a (non-unique) partial compactification with the property that the action of $\widetilde{\mathcal{A}}^{\circ}$ on itself extends to $\widetilde{\mathcal{Y}}$. Additionally, there is an action by $Y$ on $\widetilde{\mathcal{Y}}$.

For each $n \geq 0$, write $S_n = \mathrm{Spec}(R/\mathfrak{m}^{n+1})$. The $Y$-quotient of $\widetilde{\mathcal{Y}} \times_S S_n$, as an \emph{fpqc}-sheaf, is representable by a projective $S_n$-scheme $\mathcal{Y}_n$. Taking the limit, we obtain an \'etale morphism $\widetilde{\mathfrak{Y}} \to \mathfrak{Y} $ of formal $S$-schemes, where the target is algebraizable by a unique projective $S$-scheme $\mathcal{Y}$. We say that $\mathcal{Y}$ is \emph{uniformized} by $\widetilde{\mathcal{Y}} $.

To construct $\widetilde{\mathcal{Y}}$, one first selects a suitable cone decomposition $ \{ \sigma_{\alpha} \}$ of 
$$ \mathscr{C} = \left(X^{\vee}_{\mathbb{R}} \times \mathbb{R}_{>0} \right) \cup \{0\}. $$ 
The affine toric schemes $\mathcal{Z}(\sigma_{\alpha})$ glue to a partial compactification, denoted $\mathcal{Z}$, of $ \mathcal{T}$. The model $\widetilde{\mathcal{Y}}$, in turn, is obtained by gluing the contraction products 
$\widetilde{\mathcal{Y}}(\sigma_{\alpha}) = \widetilde{\mathcal{A}}^{\circ} \times^{\mathcal{T}} \mathcal{Z}(\sigma_{\alpha}) $ (cf. e.g. \cite[Paragraphs 1.19-20]{Kunnemann}). 

\subsubsection{}
By \cite[Thm. 4.6]{Kunnemann}, after replacing $S$ by a finite ramified extension if needed, we can assume that $\widetilde{\mathcal{Y}}$ and $\mathcal{Y}$ are regular schemes with special fibers being reduced strict normal crossings divisors. Then $\mathcal{Y}$ is a Kulikov model \cite[Thm. 5.1.6]{HN} and the action of $\mathcal{A}^{\circ}$ on itself extends to $\mathcal{Y}$ \cite[Thm. 4.2]{Kunnemann}. 

The special fiber $\mathcal{Y}_0$ is a quotient of $\widetilde{\mathcal{Y}}_0$ under the action of $Y$. This action  simultaneously moves forward $2N$ steps, where we recall that $2N$ is the number of components in the special fiber, and translates "sideways" by a fixed element of the elliptic curve (the latter property will not be directly relevant for us). The irreducible components $G_i$ inherit from $\widetilde{\mathcal{Y}}_s$ a natural notion of $0$-section and $\infty$-section, where the $\infty$-section of $G_i$ gets identified with the $0$-section of $G_{i+1}$.

\subsubsection{}\label{subsubsec:projectiveKulikov}
When $S$ is a curve over $k$, and $A$ an abelian scheme over $S \setminus \{ 0\}$, it is explained in \cite[Thm. 4.2]{Kunnemann} that we can "glue in" the above construction over the local complete ring at $0 \in S$. This gives a projective Kulikov model, which we shall also denote $\mathcal{Y}$.

\begin{assumption}\label{assum:Uniformization}
For the rest of this paper, we assume that $\mathcal{Y}$ is a Kulikov model constructed by uniformization, as outlined above.
\end{assumption}

\subsubsection{}
When $\mathcal{Y}$ is uniformizable it is easier to relate certain aspects of the global geometry of $\mathcal{Y}$ and the group structure of $\mathcal{A}$. 

\begin{lemma}\label{lem:fundpropuniformization}
If $ \mathcal{Y} $ is uniformizable, the following properties hold:

\begin{enumerate}
    \item There is an isomorphism 
    $$\Phi(A) \cong \mathbb{Z}/ 2N \mathbb{Z} $$ 
    and $ \mathcal{A}_s \to \Phi(A) $ maps $ [G^{\circ}_i] $ to $ [i] $ in $ \mathbb{Z}/2N \mathbb{Z} $.

\item Let $ P  \in T \subset G_0^{\circ}$. 
For any $Q \in G_i^{\circ}$, translation by $P$ has the property that if $P \to 0$ (resp.~$\infty$) in its $\mathbb{G}_m$-fiber on $ G_0^{\circ} $, then $Q + P \to 0$ (resp.~$\infty$) in its $\mathbb{G}_m$-fiber on $ G_i^{\circ} $.

\item The involution $-1_{\mathcal{A}} \colon \mathcal{A} \to \mathcal{A}$ extends to an involution 
$$ j \colon \mathcal{Y} \to \mathcal{Y}. $$
 
 \end{enumerate}
\end{lemma}
\begin{proof}
The bilinear map $b \colon Y \times X \to \mathbb{Z} $ induces an injective map $ Y \to X^{\vee}$, and \cite[Chapter III, Cor. 8.2]{FaltingsChai} asserts that its cokernel is isomorphic to $\Phi(A)$. By assumption, the cone decomposition $ \{ \sigma_{\alpha} \}$ of 
$$ \mathscr{C} = \left(X^{\vee}_{\mathbb{R}} \times \mathbb{R}_{>0} \right) \cup \{0\} $$ 
is semi-stable, meaning that the resulting toroidal compactification is semi-stable.
This implies that the primitive generator of each $1$-dimensional cone has the form $(l,1)$ where $l \in X^{\vee}$. If $\sigma_{\alpha}$ is $2$-dimensional, with primitive ray generators $(l_1,1)$ and $(l_2,1)$, then regularity implies that $ l_1 \cdot 1 - l_2 \cdot 1 = \pm 1$. Therefore, the $1$-dimensional cones are indexed precisely by $X^{\vee}$. Moreover, the cone with generator $(0,1)$ corresponds to $\widetilde{\mathcal{A}}^{\circ}_0$, see \cite[1.15]{Kunnemann} and $\widetilde{\mathcal{A}}^{\circ}_0$ is mapped to $\mathcal{A}^{\circ}_0$ under the quotient map \cite[III.5.3]{FaltingsChai}. Statement (1) is immediate from this description.

To prove Assertion (2) we use \cite[Proposition 3.5 (iii)]{Kunnemann}. 
More precisely, by this result it suffices to show that the action of $\mathcal{T}$ on the toroidal partial compactification $\mathcal{Z}$ of $\mathcal{T}$ has the analogous property. This is an explicit computation using the construction of $\mathcal{Z}$ given in \cite[1.12-1.14]{Kunnemann}.

Assertion (3) is proved in \cite[Thm.~4.2]{Kunnemann}. 
\end{proof}

\subsubsection{} 

We will also need a simple base change property of $\mathcal{Y} \to S $. Let us write $ R = k[[t]] $ and $ R' = k[[\pi]]$,
and let
$$ \alpha \colon S' = \mathrm{Spec}(R) \to S$$
be defined by $ t \mapsto \pi^r $, for some $r>0$. We consider the blowup map
$$ \mathcal{Y}' \to \mathcal{Y}_{S'} $$
from Proposition \ref{prop:main-blow-up} (2) (leaving out $r$ from the notation). In terms of the obvious toroidal structure on $ \mathcal{Y}_{S'} $ (being a base change of the semi-stable $S$-scheme $\mathcal{Y}$), this map can also be interpreted as the toridal resolution minimally refining the two-dimensional cones in the cone decomposition associated with $ \mathcal{Y}_{S'} $.

\begin{lemma}\label{lem:unifblowup}
Let $\widetilde{\mathcal{Y}}'$ denote the relatively complete model obtained from the pullback $\widetilde{\mathcal{Y}}_{S'}$ by the minimal toroidal desingularization. Then $\mathcal{Y}'$ is uniformized by $\widetilde{\mathcal{Y}}'$.

\end{lemma}
\begin{proof}
In this proof, we will use some standard facts on formal schemes and admissible formal blowups, see for instance \cite[Part II]{Boschbook}.

The first step, consisting of the base change from $S$ to $S'$, is essentially discussed in the proof of \cite[Thm.~4.6]{Kunnemann}. In particular, one notes that the lattice $Y$ does not change. Let us explain why $\mathcal{Y}_{S'}$ is uniformized by $\widetilde{\mathcal{Y}}_{S'}$. 
This is because $\widetilde{\mathcal{Y}}_n \to \mathcal{Y}_n$ is an \emph{fpqc} quotient (by $Y$), and, therefore, commutes with arbitrary base change. 
The statement therefore follows after pulling back $ \widetilde{\mathfrak{Y}} \to \mathfrak{Y}$ along $S' \to S$, 
and considering truncations modulo powers of the maximal ideal.

The second step is to consider the blowup $ \mathcal{Y}' \to \mathcal{Y}_{S'} $. After completion, it determines an admissible formal blowup $\mathfrak{Y}' \to \mathfrak{Y}_{S'}$. 
Pulling back along the \'etale morphism $ \widetilde{\mathfrak{Y}}_S \to \mathfrak{Y}_S$ yields an admissible formal blowup $\widetilde{\mathfrak{Y}}' \to \widetilde{\mathfrak{Y}}_S $. 

Clearly, the latter is also the formal completion of the blowup $ \widetilde{\mathcal{Y}}' \to \widetilde{\mathcal{Y}}_S$ (defined similarly to $ \mathcal{Y}' \to \mathcal{Y}_S $). To conclude, we use again universality of the formation of \emph{fpqc}-quotients, together with the previous case already established. I.e., for each $n > 0$ we can identify $\mathcal{Y}'_n$ with the quotient of $\widetilde{\mathcal{Y}}'_n$ under the action of $Y$.
\end{proof}

\subsubsection{} 

For later use, we record the fact that the bipartite structure of $ \mathcal{Y}_s $ is compatible with the involution $j$. We write $G_{even}$ and $G_{odd}$ for the disjoint union of the $G_i$ with $i$ even and odd, respectively.

\begin{lemma}\label{lemma:bipar-pres}
The subschemes $G_{even}$ and $G_{odd}$ are preserved by the involution $j$.   
\end{lemma}
\begin{proof}
It suffices to prove the statement after restricting to $\mathcal{Y}^{sm} = \mathcal{A}$. Let $P$ be a point in $ \mathcal{A}_s $, and let us assume that $ P \in G_i^{\circ} $. By the group law on $ \mathcal{A}_s $, we must have that $ - P \in G_{2N-i}^{\circ}$, thus the lemma follows immediately. 
\end{proof}

\section{Expanded degenerations and N\'eron models}\label{sec:expandedNeron}
In this section, we assume that $ \mathcal{Y} \to C $ is a projective Kulikov model of type II, with $C$ a smooth affine connected curve of finite type over $k$. We also assume that the special fiber $ \mathcal{Y}_0 $ has an even number of components. Then, for each $n > 0$, we can form the $n$-th expanded degeneration $ \mathcal{Y}[n] \to C[n] $ as explained in Section \ref{sec:prelandnot}. 

\subsection{N\'eron models over generic lines}
We shall first establish some useful facts about the restriction of expanded degenerations to generic lines.

\subsubsection{}

\begin{lemma}\label{lem:smallres}
The relative canonical sheaf of $ \mathcal{Y}[n] \to C[n] $ is trivial.
\end{lemma}
\begin{proof}
This follows from the fact that 
$$ \mathcal{Y}[n] \to \mathcal{Y} \times_C C[n] $$ 
is a small resolution.
\end{proof}

Let $ L \to C[n] $ be a generic line.

\begin{prop}\label{prop:LNeronmod}
The restriction $ \mathcal{Y}[n]_L \to L $ of $ \mathcal{Y}[n] \to C[n]$ to the generic line $ L $ is a Kulikov model. In particular, 
$$ \mathcal{Y}[n]^{sm}_L \to L $$ 
is a N\'eron model.
\end{prop}
\begin{proof}
The claims that $ \mathcal{Y}[n]_L \to L $ has regular total space and that the special fiber is a reduced normal crossings divisor are easily checked from the usual local equations \cite[Prop.~1.5]{GHH}. Since the special fiber coincides with $\mathcal{Y}[n]_0$, we know the special fiber to be strict normal crossing. By Lemma \ref{lem:smallres} above $\mathcal{Y}[n] \to C[n]$ has trivial relative canonical sheaf, so the same holds for the base change $ \mathcal{Y}[n]_L \to L $.
\end{proof}

To simplify notation, we shall often write $ \mathcal{A}(L) $ for the N\'eron model $ \mathcal{Y}[n]^{sm}_L $ over $L$.

\begin{remark}
The model $ \mathcal{X}_{r}' \to S $ constructed in Proposition \ref{prop:main-blow-up} is in particular a Kulikov model, since
$$ \mathcal{X}_{r}' \cong X[r-1]_S,$$
where the lift $ S \to C[r-1]$ factors through the \emph{diagonal line} in $C[r-1]$.  
\end{remark}

\subsubsection{}
The small resolution 
\begin{equation}\label{eq:smallres}
\mathcal{Y}[n] \to \mathcal{Y} \times_C C[n]   
\end{equation}
restricts to an isomorphism over $ \mathcal{Y}^{sm}_{C[n]} = \mathcal{Y}^{sm} \times_C C[n]$. Thus we obtain an open immersion
$$
\mathcal{Y}^{sm}_{C[n]} \to \mathcal{Y}[n]^{sm}.  
$$

Restricting \eqref{eq:smallres} to any generic line $ L \to C[n]$ gives a  resolution (which is no longer small)
$$ \mathcal{Y}[n]_L \to \mathcal{Y}_L. $$
Similar to above, this resolution is an isomorphism over $ \mathcal{Y}^{sm}_L $, 
which gives rise to an open immersion
\begin{equation}\label{eq:basechmap}
\mathcal{Y}^{sm}_L \to \mathcal{Y}[n]^{sm}_L.    
\end{equation}

\subsubsection{}
The composition $L \to C[n] \to C$ is a flat and ramified map of smooth curves. Therefore, we can identify (\ref{eq:basechmap}) with the base change morphism 
$$ b_{L/C} \colon \mathcal{A} \times_C L \to \mathcal{A}(L) $$
extending the canonical isomorphism over $L \setminus \{0\}$. 

\subsection{Relating N\'eron models over different generic lines.}

Let $L$ be a generic line in $C[n]$, and pick a $k$-valued point $\sigma \in \mathbb G[n]$. Then the automorphism
$$ \sigma \colon C[n] \to C[n] $$
restricts to an isomorphism $L \to L'$ of $L$ with some other generic line $L'$. For simplicity, this map will also be denoted $\sigma$; note that it commutes with the respective maps to $C$. 

We shall next compare the two N\'eron models $\mathcal{A}(L) $ and $\mathcal{A}(L') $.

\subsubsection{} 
Recall that the map $\mathcal{Y}[n] \to C[n]$ is $\mathbb G[n]$-equivariant. Thus by restriction we obtain a diagram
\begin{equation}\label{diagram:LLprime}
\xymatrix{
    \mathcal{Y}[n]_L^{sm}  \ar[r]^{\sigma} \ar[d] & \mathcal{Y}[n]_{L'}^{sm}   \ar[d] \\
   L   \ar[r]^{\sigma} & L'.
}    
\end{equation}
We point out that this diagram is not Cartesian. However, 
if we replace $\mathcal{Y}$ by $ \mathcal{Y}^{sm} $ at the beginning, and perform the same constructions, then $ \mathcal{Y}^{sm}[n] $ is simply the base change $\mathcal{Y}^{sm} \times_C C[n] $. In this situation the diagram is Cartesian and reduces to the isomorphism 
$$ \mathcal{Y}^{sm}_L \cong \mathcal{Y}_{L'}^{sm} \times_{L'} L $$
of $L$-group schemes. 
Note that the unit section of $ \mathcal{Y}^{sm} \to C $ pulls back to the unit sections of $\mathcal{Y}^{sm}_L$ and $\mathcal{Y}_{L'}^{sm}$, respectively.

\subsubsection{}
Diagram (\ref{diagram:LLprime}) gives rise (by the universal property of the fiber product) to an $L$-morphism
$$ \sigma_{L,L'} \colon \mathcal{A}(L) \to \mathcal{A}(L') \times_{L'} L. $$

\begin{prop}\label{prop:genline-hom}
The map $\sigma_{L,L'}  $ is an isomorphism of $L$-group schemes.
\end{prop}
\begin{proof}
First, we note that the pullback $ \mathcal{A}(L') \times_{L'} L $ is the N\'eron model of its generic fiber, since $ L \to L'$ is an isomorphism, and the formation of N\'eron models is preserved under \'etale base change (see \cite[Prop.~1.2/2]{neron}).

Second, the map $\sigma_{L,L'} $ extends the canonical isomorphism in the generic fiber, which, as we noted above, is a homomorphism. 
Therefore, by the N\'eronian property of $ \mathcal{A}(L') \times_{L'} L $, 
 the homomorphism on the generic fiber extends \emph{uniquely} to an $L$-morphism between the group schemes, and this map is necessarily also a homomorphism (again because of the N\'eronian property of the target, by a standard argument). 
\end{proof}

\subsubsection{} 

Over $ 0 \in L $, the map $\sigma_{L,L'} $ is precisely the automorphism
$$ \sigma_0 \colon \mathcal{Y}[n]_0^{sm} \to \mathcal{Y}[n]_0^{sm} $$
on the (smooth part of) the fiber of the expanded degeneration over $0 \in C[n]$, or, what amounts to the same, the fiber over $0$ of $\sigma$ in \eqref{diagram:LLprime}. We point out that, since $\sigma_{L,L'}$ is a homomorphism, the group structures on $\mathcal{Y}[n]_0^{sm}$ induced by the N\'eron models $ \mathcal{A}(L) $ and $\mathcal{A}(L')$, respectively, are forced to be \emph{different} in general. This is illustrated in Example \ref{Exa:differentgrstr} below.

\subsection{Torus action on a cycle}\label{subsec:torus-action-expansion}

For any $ b \geq 0 $, the torus $\mathbb G[b] \cong (\mathbb{G}_m)^b$ acts on each inserted component of $ \mathcal{Y}[b]_0 $ via projection to one of the rank $1$-factors. On the other hand, if $L$ denotes the diagonal line in $C[b]$, then $\mathcal{Y}[b]$ is uniformizable (combine Proposition \ref{prop:main-blow-up} and Lemma \ref{lem:unifblowup}). By Lemma \ref{lem:fundpropuniformization}, the rank $1$ subtorus $T \subset G_0^{\circ}$ acts by translation in the fibers. In this subsection, we clarify the relationship between these two actions.

\subsubsection{}\label{subsubsec:notationcomp}

We first introduce some notation. The fiber $ \mathcal{Y}[b]_0  $ has the structure of a cycle of components denoted $G_0^b, \ldots, G_{2N(b+1)-1}^b$, see \ref{subsubsection:explanationcycle}. Letting 
$ 0 \leq \tau \leq N-1 $, we write 
$$G^b_{\tau,i} = G^b_{2 \tau (b+1) + i}, $$ 
where $ 0 \leq i \leq b $, and 
$$ G^b_{\tau,-i} = G^b_{2 \tau (b+1) - i},$$
where $ 1 \leq i \leq b + 1 $ (when $\tau=0$, we identify $ G^b_{-i} = G^b_{2N(b+1)-i}$). When $b$ is clear from the context, we sometimes write $G'_{\tau,i}$ and $G'_{\tau,-i}$, for simplicity. 

Via the resolution $ \mathcal{Y}[b] \to \mathcal{Y} \times_C C[b]$, we can identify $G^b_{\tau,0}$ and $G^b_{\tau,-(b+1)}$ with the strict transforms of $G_{2 \tau}$ and $G_{2\tau-1}$, respectively. Moreover, $G^b_{\tau,0}$ and $G^b_{\tau+1, -(b+1)}$ are linked by the chain $G^b_{\tau,1}, \ldots, G^b_{\tau,b}$, whereas $G^b_{\tau, -(b+1)}$ and $G^b_{\tau,0}$ are linked by $G^b_{\tau,-b}, \ldots, G^b_{\tau,-1}$.

\subsubsection{}
For any $0 \leq \tau \leq N-1$, $1 \leq i \leq b$ and $d \in D$, the $\mathbb{P}^1$-fiber $(G^b_{\tau, \pm i})_d$ has homogeneous coordinates $(u_i \colon v_i)$ being a fiber of an $i$-th inserted component. In these coordinates, $\sigma = (\sigma_1, \ldots, \sigma_b) \in \mathbb G[b]$ acts by
$$ (u_i \colon v_i) \mapsto (\sigma_i u_i \colon v_i) = (u_i \colon \sigma_i^{-1} v_i). $$
Let $\mathbb{P}$ denote the fiber of the identity component $G^b_{0,0} = G_0$ over the identity $e_D \in D$. Since $\mathcal{Y}[b]_0$ is the special fiber of $\mathcal{Y}[b]_L$, which is obtained through uniformization, it follows from the explicit construction of $\mathcal{Y}[b]$ (see \cite[Prop.~1.5]{GHH}) that we have a canonical isomorphism $\mathbb{P} \cong (G^b_{\tau, i})_d$ sending $(0,\infty)$ to $(0,\infty)$. We moreover have a canonical isomorphism $\mathbb{P} \cong (G^b_{\tau, -i})_d$ sending $(0,\infty)$ to $(\infty,0)$.

\subsubsection{}
For any $g \in T$, we denote by $t_g$ translation by $g$ on $ (G^b_{\tau, \pm i})_d^{\circ}$, using the group structure of $\mathcal{Y}[b]_L$.

\begin{lemma}\label{lemma:translationinterpret}
Let $ \sigma = (\sigma_1, \ldots, \sigma_b) \in \mathbb G[b]$. Then
\begin{enumerate}
    \item $ \sigma \vert_{(G_{\tau,i})_d^{\circ}} = t_{\sigma_i}$, and 
    \item $ \sigma \vert_{(G_{\tau,-i})_d^{\circ}} = t_{- \sigma_i}$.
\end{enumerate}
\end{lemma}
\begin{proof}
Let $(U \colon V) $ be homogeneous coordinates on $\mathbb{P}$. Then the isomorphism $T \cong (G_{\tau,i})_d^{\circ}$ sends the coordinate $\frac{U}{V}$ to $a \frac{u_i}{v_i}$ for some $a \in k$, whereas $T \cong (G_{\tau,-i})_d^{\circ}$ sends $\frac{U}{V}$ to $a' \frac{v_i}{u_i}$, for some $a' \in k$. The statement is immediate from this.
\end{proof}

\subsubsection{} We discuss an example to illustrate various aspects of our discussion so far.

\begin{example}\label{Exa:differentgrstr}  
We assume that $\mathcal{Y}_s$ consists of two components $G_0$ and $G_1$ isomorphic to $\mathbb{P}^1$ (so the generic fiber is an elliptic curve). Let $L$ be the diagonal line in $C[1]$; then $(\mathcal{Y}[1]_L)_0$ is a cycle consisting of four components, denoted by 
$$ G'_{-2}, G'_{-1}, G'_{0} ~\mathrm{and}~ G'_{1}. $$
Here $G'_{-2}$ and $G'_{0}$ are the strict transforms of $G_1$ and $G_0$ respectively.

We pick a group element $\sigma \in \mathbb G[1](k)$ and write $\lambda$ for the point on $ (G'_0)^{\circ}$ with coordinate $\sigma$. The action of $\sigma$ is trivial on $(G'_{-2})^{\circ}$ and $(G'_{0})^{\circ}$, translation by $\lambda$ on $(G'_{1})^{\circ}$ and translation by $-\lambda$ on $(G'_{-1})^{\circ}$. For $P \in (G'_{1})^{\circ}$, we also write this as $\sigma(P) = P +_L \lambda $, where $+_L$ denotes the group law on $ \mathcal{A}(L)_0 $. Recall that we have an isomorphism of groups 
$$ \mathcal{A}(L)_0 \cong (G'_0)^{\circ} \times \mathbb{Z}/4 \mathbb{Z}. $$ 

Let $L'$ be the image of $L$ under $\sigma$; it is again a generic line. We let $+_{L'}$ denote the group law on $ \mathcal{A}(L')_0 $. We claim that this is different from $+_L$  in general. To see this, let $Q$ be another point on $(G'_{1})^{\circ}$. Then we compute
$$ P +_L Q = \sigma(P +_L Q) = \sigma((P -_L \lambda) +_L (Q -_L \lambda) +_L 2\lambda) $$
where the first equality is due to the fact that $P +_L Q \in G'_{-2} $ where $\sigma$ is trivial. The morpism $ \mathcal{A}(L)_0 \to \mathcal{A}(L')_0 $ on special fibers equals $ (\sigma_{L,L'})_0$, which by Proposition \ref{prop:genline-hom} is a homomorphism. Hence $ P +_L Q$ equals also
$$ \sigma(P -_L \lambda) +_{L'} \sigma(Q -_L \lambda) +_{L'} 2 \lambda = P +_{L'} Q +_{L'} 2 \lambda. $$
This means that $P +_L Q$ only coincides with $P +_{L'} Q$ when $\lambda$ is a $2$-torsion point (in which case the equations in Subsection \ref{subsec:genlinesbasics} show that $ L = L'$, and the map is an automorphism of order $2$. 
But from Lemma \ref{lemma:translationinterpret}, we see that this can only be the case for finitely many $\sigma$. 

\end{example}

\subsection{Extending the involution}\label{subsec:invoext}

We next show that the involution $ j \colon \mathcal{Y} \to \mathcal{Y} $ behaves well under the formation of the expansions $ \mathcal{Y}[n] $. We will use the results in \cite[Section 1.4.3]{GHH}, and follow the notation there. In particular, we recall that for each $n$, there is a natural map $p_n \colon \mathcal{Y}[n] \to \mathcal{Y}$.

\begin{lemma}\label{lem:involutiongeneral}
For each $n \geq 0$, the pullback $ j \times id $ lifts along the resolution $ \mathcal{Y}[n] \to \mathcal{Y} \times_C C[n] $. The resulting involution $j[n] \colon \mathcal{Y}[n] \to \mathcal{Y}[n] $ commutes with $j[n'] $, for any $ n' \leq n$,  via the maps $ \mathcal{Y}[n] \to \mathcal{Y}[n'] $.
\end{lemma}
\begin{proof}
We recall from \cite[Section 1.4.3]{GHH} that $ \mathcal{Y}[n] $ is the blowup of $ \mathcal{Y}[n-1] \times_{\mathbb{A}^1} \mathbb{A}^2 $ with center $ p_{n-1}^{-1}(G_{even}) \times V(t_{n+1})$. 

For $n=1$, we observed in Lemma \ref{lemma:bipar-pres} that $G_{even}$ 
is preserved under $j$. Thus the same holds for $G_{even} \times V(t_2) $ 
and we see that $j \times id$ lifts uniquely to the involution $j[1]$. In particular, 
$$ p_1 \colon \mathcal{Y}[1] \to \mathcal{Y} $$
is equivariant for the action $ j[1] $ on the source and $j$ on the target, and therefore $ p_1^{-1}(G_{even})$
is preserved under $j[1]$.

Consider now the general case $ n \geq 1$. When we blow up $ \mathcal{Y}[n-1] \times_{\mathbb{A}^1} \mathbb{A}^2 $, we may therefore suppose by induction that the center $ p_{n-1}^{-1}(G_{even}) \times V(t_{n+1})$
is preserved by $ j[n-1] \times id$, and the result follows.
\end{proof}

\subsubsection{}
Let $ L \to C[n]$ be a generic line. As we observed earlier, $ \mathcal{A}(L)  = \mathcal{Y}[n]^{sm}_L $ is a N\'eron model of its generic fiber. It therefore has an inverse
$$ \mathrm{inv}(L) \colon \mathcal{A}(L) \to \mathcal{A}(L). $$

\begin{lemma}\label{lemma:invoL}
The involution $ \mathrm{inv}(L)$ equals the restriction of $ j[n] $ to $\mathcal{Y}[n]^{sm}_L$.
\end{lemma}
\begin{proof}
By construction, the two maps coincide over $L \setminus 0$, in which case they are both the pullback of $j$ restricted to $ \mathcal{Y} \setminus \mathcal{Y}_s $ over $C \setminus 0$. So the statement again follows from \cite[Prop.~3.3.11]{liu}. 
\end{proof}

\begin{remark}
The involution sends a component $ G^n_{(t,i)} $ 
to $ G^n_{(N-t,-i)} $.
\end{remark}

\section{The Kummer locus}\label{sec:Kummerlocus}
In this section, we introduce and study in detail the \emph{Kummer locus} inside the stable locus of the $n$-th Hilbert scheme. It is defined as the closure of the Kummer construction for the group scheme $ \mathcal{Y}^{sm} \times_C C[n]$ over $C[n]$. In the main result of this section, Theorem \ref{theorem:mainstrat}, and a key result of the paper, we give a partition of the Kummer locus in terms of smooth irreducible and $\mathbb G[n]$-stable varieties, compatible with and relative to the natural stratification of the base $C[n]$ induced by the coordinate hyperplanes.

\subsection{Preliminaries} \label{subsec:Kumconstr}

\subsubsection{}

Let $ \mathcal{G}$ be a commutative group scheme, smooth and of finite type over a connected regular Noetherian scheme $B$. In our applications $B$ will be either a smooth $k$-variety or the spectrum of a discrete valuation ring with residue field $k$, so we assume that every closed point in $B$ has an algebraically closed residue field. We shall assume (the only case we will need) for each $b \in B$ that the fiber $ \mathcal{G}_b $ is equidimensional of dimension $\delta$, with $\delta \in \{1,2\}$, and that the identity component $ \mathcal{G}_b^{\circ} $ is a semi-abelian variety.

The summation map
$$ \mathcal{G}^n \to \mathcal{G};~~~~~(g_1, \ldots, g_n) \mapsto g_1 + \ldots + g_n $$
descends to the symmetric product, and hence gives rise to 
$$ \Sigma \colon \mathrm{Hilb}^n(\mathcal{G}/B) \to \mathcal{G}. $$
We shall often, somewhat informally, also refer to $\Sigma$ as the summation map. 

\begin{definition}
The $(n-1)$-st Kummer scheme of $ \mathcal{G}$ is the kernel
$$ \mathrm{Kum}^{n-1}(\mathcal{G}) = \Sigma^{-1}(e_{\mathcal{G}}), $$
where $e_{\mathcal{G}}$ is the unit section.
\end{definition}

\subsubsection{}

\begin{prop}\label{prop:Kummernsmooth}
 The scheme $ \mathrm{Kum}^{n-1}(\mathcal{G}) $ is smooth over $B$, of relative dimension $(n-1) \cdot \delta$.    
\end{prop}
\begin{proof}
By our assumptions, the $n$-th Hilbert scheme is smooth over $B$, so, using the \emph{Fibral Flatness Theorem} \cite[039A]{Sta}, it suffices to prove  smoothness with $ \mathcal{G} $ replaced by a (closed) fiber $ \mathcal{G}_b $. We will explain how the classical proof for abelian surfaces can be modified to our situation. 
First of all, note that $\mathrm{Hilb}^n(\mathcal{G}_b)$ is a disjoint union of smooth irreducible components. Let $\mathscr{C}$ be a component that maps under the summation map to $\mathcal{G}_b^{\circ}$. An elementary argument shows that 
$$ \mathscr{C} \to \mathcal{G}_b^{\circ} $$
is dominant, and in fact surjective. Since both, source and target, are smooth, and $\mathrm{char}(k) = 0$, the map is 
smooth over a dense open subset of $\mathcal{G}_b^{\circ} $, 
by \cite[Cor.~III.10.7]{Hartshorne}. Using the group structure of $\mathcal{G}_b^{\circ} $, 
this implies that all fibers are smooth, by translation.
\end{proof}

\subsection{Definition of the Kummer locus} 
We let $ \mathcal{Y} \to C $ be a projective type II Kulikov model with general fiber an abelian variety of dimension $ \delta = 1$ or $2$ (see Subsection \ref{subsec:typeII}). 
\subsubsection{}

The open subscheme $\mathcal{Y}^{sm}_{C[n]}$ of $ \mathcal{Y}[n]$ is a group scheme over $C[n]$, and we can form the Kummer scheme $ \mathrm{Kum}^{n-1}(\mathcal{Y}^{sm}_{C[n]}) $ as in Subsection \ref{subsec:Kumconstr}. 

We recall some notation which is used in the sequel. We write $\mathcal{H}^n = \mathrm{Hilb}^n(\mathcal{Y}[n]/C[n])^{st}$ for the stable locus and         $C[n]^*$ for the pre-image of $C \setminus 0$ under $C[n] \to C$. Let us also point out that any $C[n]$-scheme is viewed as a $C$-scheme via this map.

\begin{definition}
The \emph{generic Kummer locus}
is defined by 
$$ \mathcal{K}^{n-1}_{\circ} := \mathrm{Kum}^{n-1}(\mathcal{Y}^{sm}_{C[n]}) \cap \mathcal{H}^{n}. $$
We call the closure $\mathcal{K}^{n-1} $ of $ \mathcal{K}^{n-1}_{\circ} $ in $ \mathcal{H}^{n} $ the \emph{Kummer locus}.
\end{definition}

We list some immediate properties.

\begin{lemma}\label{lem:Kummerlocproperties}
The Kummer locus $\mathcal{K}^{n-1}$ is 
\begin{enumerate}
    \item integral,
    \item flat over $C$, and
    \item stable under the $\mathbb G[n]$-action.
\end{enumerate}
\end{lemma}
\begin{proof}
The restriction $ \mathcal{K}^{n-1}_{\circ} \vert_{C[n]^*}$ coincides with the pullback to $C[n]^*$ of the restriction $\mathrm{Kum}^{n-1}(\mathcal{Y}^{sm}) \vert_{C \setminus \{0\}}$. Since the latter is integral, the same holds for its closure $\mathcal{K}^{n-1} $. As it moreover dominates the Dedekind scheme $C$, flatness of $\mathcal{K}^{n-1} $ over $C$ is immediate \cite[Prop.~4.3.9]{liu}.

For part (3), observe that $ \mathcal{K}^{n-1}_{\circ} $ is $\mathbb G[n]$-stable. Thus every $\sigma \in \mathbb G[n]$ preserves $\mathcal{K}^{n-1} $, which implies that $ \mathbb G[n] \times \mathcal{K}^{n-1} \to \mathcal{H}^{n} $ has set-theoretic image $ \mathcal{K}^{n-1} $. This is also the scheme theoretic image, as the source of the action map is reduced, so in conclusion $ \mathcal{K}^{n-1} $ is stable under the torus action. 
\end{proof}

\begin{remark}
It is in fact \emph{not} true that $\mathcal{K}^{n-1}$ is flat over $C[n]$. In later sections, we shall study in detail the case $n=3$, where one can observe that the fiber dimension of $\mathcal{K}^2 \to C[3]$ is not constant.
\end{remark}

\subsubsection{}\label{subsubsec:Kummerlocquot} We write $K^{n-1}_{\mathcal{Y}/C}$ for the induced GIT quotient $\mathcal{K}^{n-1}/\mathbb G[n]$. The above lemma implies the following result.

\begin{prop}
The quotient $K^{n-1}_{\mathcal{Y}/C}$ is integral and flat over $C$. Moreover, the natural map $K^{n-1}_{\mathcal{Y}/C} \to I^{n}_{\mathcal{Y}/C}$ is a closed immersion.
\end{prop}
\begin{proof}
Since all schemes involved are of finite type over $k$, which is of characteristic zero, the statement follows from Lemma \ref{lem:Kummerlocproperties} and standard properties of quotients (see for instance \cite[Chapter 0, \S 2  and Chapter 1, \S 4]{GIT}).
\end{proof}

\subsection{The boundary of the Kummer locus}\label{subsubsec:Kummerboundary}
In the remainder of this section we describe the boundary $ \mathcal{K}^{n-1} \setminus \mathcal{K}^{n-1}_{\circ} $ of the Kummer locus. 

\subsubsection{} 
We first recall the following useful fact. For any closed point $ x $ 
in the boundary, we can find a complete discrete valuation ring $R$ with residue field $k$ and a morphism $ S = \mathrm{Spec}(R) \to \mathcal{K}^{n-1} $ mapping the closed point to $x$ and the generic point $\eta$ to the (dense) open subset $ \mathcal{K}^{n-1}_{\circ} \vert_{C[n]^*} $. Indeed, since $ \mathcal{K}^{n-1} $ is quasi-projective over the algebraically closed field $k$, and $char(k) = 0 $, one can show this by combining resolution of singularities with Bertini's theorem.

\subsubsection{}\label{subsubsec:ZinKum} Let $S = \mathrm{Spec}(R)$ be the spectrum of a complete discrete valuation ring with residue field $k$, and let $ Z $ be an $S$-valued point of $ \mathcal{H}^n $. We denote by $Z_{\eta}$ and $ Z_s $ the fibers over the generic and closed point respectively. We assume that $\eta$ has image in $C[n]^*$ under the projection $ \mathcal{H}^n \to C[n] $, and, moreover, that $Z_{\eta} $ maps to $\mathcal{K}^{n-1} \vert_{C[n]^*} $.

After, if necessary, acting with a point in $\mathbb G[n](R)$, we can by Lemma \ref{lemma:standardformmaptobase} assume that the map $S \to C[n]$ corresponding to $Z$ factors through a standard embedding 
$$ C[b] \subset C[n], $$
where $s$ maps to $ 0 \in C[b]$. Thus we can view $Z$ as a closed subscheme of $\mathcal{Y}[b]_S$, finite and flat over $S$. Since $ Z $ is stable, we have $Z \subset \mathcal{Y}[b]_S^{sm}$, and $Z_s$ is subject to the stability criterion in Theorem \ref{theorem:stabilitycondition} (which only depends on the underlying cycle).

\subsubsection{}\label{subsubsec:randbspeci}
Let $ \mathcal{Y}'_r \to \mathcal{Y}[b]_S $ be the resolution from Proposition \ref{prop:main-blow-up} (we do not need to specify $r$ in this discussion). Then we can view $Z$ as a subscheme of $(\mathcal{Y}'_r)^{sm} $ via the open inclusion
$$ \mathcal{Y}[b]_S^{sm} \subset (\mathcal{Y}'_r)^{sm}. $$

\begin{lemma}\label{lemma:kernelZ}
The scheme $ Z  $ is an $S$-valued point of $ \mathrm{Kum}^{n-1}((\mathcal{Y}'_r)^{sm}/S) $.
\end{lemma}
\begin{proof}
The generic fiber $Z_{\eta}$ is in the kernel of the summation map, so the statement follows from the fact that $ \mathrm{Kum}^{n-1}((\mathcal{Y}'_r)^{sm}/S) $ is closed in $\mathrm{Hilb}^n((\mathcal{Y}'_r)^{sm}/S)$.
\end{proof}

\subsubsection{Orbit of the Kummer locus over a line}
We briefly explain how N\'eron models over generic lines can be used to obtain a partial understanding of the boundary of the Kummer locus. However, considering only generic lines turns out to be too restrictive. Later on in this section we modify this approach and explain the general procedure for computing the boundary.

Let $ C[b] \subset C[n] $ be a standard embedding, and
$ L \subset C[b] $ a generic line. Pulling back $ \mathcal{Y}[n]^{sm} $ along the (closed) composition 
$ L \to C[n] $
yields the N\'eron model $ \mathcal{A}(L) \to L $. Then
$$\mathrm{Kum}^{n-1}(\mathcal{A}(L)) \subset \mathrm{Hilb}^{n}(\mathcal{Y}[n]^{sm}/C[n])$$
is a closed immersion. Intersecting with the stable locus $\mathcal{H}^{n}$, we obtain 
\begin{equation}\label{inclusion:GenLine}
\mathrm{Kum}^{n-1}(\mathcal{A}(L))^{st} \subset \mathcal{H}^{n}.    \end{equation}
Since (\ref{inclusion:GenLine}) is a closed immersion, the special fiber $ \mathrm{Kum}^{n-1}(\mathcal{A}(L)_s)^{st} $ will be contained in the Kummer locus. 

Let $ \sigma \in \mathbb G[b](k) $ and set $L' = \sigma(L)$. By Proposition \ref{prop:genline-hom}, the induced map
$$ \sigma \colon \mathcal{A}(L)_s \to \mathcal{A}(L')_s $$
is an isomorphism of group schemes and yields 
a corresponding isomorphism
$$ \sigma \colon \mathrm{Kum}^{n-1}(\mathcal{A}(L)_s)^{st} \to \mathrm{Kum}^{n-1}(\mathcal{A}(L')_s)^{st}. $$

By the above discussion, the orbit 
$$ \bigcup_{\sigma \in \mathbb G[n]} \sigma(\mathrm{Kum}^{n-1}(\mathcal{A}(L)_s)^{st}) $$
is independent of $L$ and contained in the Kummer locus. 

\subsection{Numerical obstruction}\label{subsec:newlabeling}
Consider an $S$-valued point $Z$ of $\mathcal{H}^n$ as in \ref{subsubsec:ZinKum}. By Lemma \ref{lemma:kernelZ}, the fiber $Z_s$ must be in the kernel of the summation map 
$$ \mathrm{Hilb}^n((\mathcal{Y}'_r)^{sm}_s) \to (\mathcal{Y}'_r)^{sm}_s, $$ 
in particular, it must map to the identity component. We will now derive from this a \emph{numerical obstruction}, which allows us to rule out certain strata of $ \mathcal{H}^n $ over $C[n] \setminus C[n]^*$ from the boundary of $\mathcal{K}^{n-1}$.

\subsubsection{}
We first introduce notation for the strata of $ \mathcal{H}^n $.
Consider the standard embedding $ \tau_I \colon C[b] \to C[n]$. 
Recall from \ref{subsubsec:notationcomp} that over the origin $0 \in C[b]$, the fiber $ \mathcal{Y}[b]_0 = (\tau_I^*\mathcal{Y}[n])_0 $ contains a chain of $b$ inserted components mapping to each component of the double locus of $\mathcal{Y}_s$. Thus $ \mathcal{Y}[b]_0 $ is a cycle of 
components $G^b_i$, where $0 \leq i < 2N(b+1)$. 
The strict transforms of the components of $ \mathcal{Y}_s $ are indexed by $ k(b+1)$ for $ 0 \leq k \leq 2N-1 $.

Now let 
$$ \mathbf{m} = (m_0, \ldots, m_j, \ldots, m_{2N(b+1)-1})$$
be a tuple of integers $m_j \geq 0$ such that $ \sum_j m_j = n$. Then $ \mathrm{Hilb}^n(\mathcal{Y}[b]_0^{sm})$ is a disjoint union of strata denoted
$$ \mathscr{S}_{\mathbf{m}} = \prod_j \mathrm{Hilb}^{m_j}((G^b_j)^{\circ}).$$

\subsubsection{}

Assume that $Z_s$ belongs to $ \mathscr{S}_{\mathbf{m}} $. The map $S \to C[b]$ is
given by 
$$ t_i \mapsto \pi^{r_i}, $$ 
where $ r_i > 0$ for $i=1, \ldots, b+1$. By Proposition \ref{prop:main-blow-up}, the integers $r_i$ tell us precisely how many exceptional components are inserted over each component of the double locus by the blowup
\begin{equation}\label{eq:resolution-b-r}
\mathcal{Y}'_r \to \mathcal{Y}[b]_S. 
\end{equation}

We need to keep track of the strict transforms of the components of $ \mathcal{Y}[b]_0 $ in (\ref{eq:resolution-b-r}). For simplicity, we denote the components of $ (\mathcal{Y}'_r)_s = \mathcal{Y}[r-1]_0$ by $\mathscr{G}_i = G_i^{r-1}$.

If we formally set $ r_0 = r_{b+2} = 0$, it is immediate to check that:
\begin{enumerate} 
\item The strict transform of $G^b_{2 \tau (b+1) + k}$ is $\mathscr{G}_{2 \tau r + \sum_{l=0}^k r_l}$
\item The strict transform of $G^b_{2 \tau (b+1) - k}$ is $\mathscr{G}_{2 \tau r - \sum_{l=1}^k r_l}$.
\end{enumerate}

\subsubsection{}
We define integers
$$ \mathscr{W}^+(\mathbf{m}) = \sum_{\tau=0}^{N-1} \sum_{k=0}^b m_{2\tau(b+1) + k} \cdot \left( 2\tau r + \sum_{l=0}^k r_l \right)$$
and 
$$ \mathscr{W}^-(\mathbf{m}) = \sum_{\tau=0}^{N-1} \sum_{k=1}^{b+1} m_{2 \tau (b+1) - k} \cdot \left( 2 \tau r - \sum_{l=1}^k r_l \right). $$

\begin{remark}
We point out that when $\tau=0$, the indexing of the components $G_i^b$, resp.~$G_{i}^{r-1}$, should be interpreted modulo $2N(b+1)$, resp.~$2Nr$, in the obvious way. Thus, in particular, $m_{-k} = m_{2N(b+1) - k}$ for $ 1 \leq k \leq b+1$. Additionally, as we are only interested in the value of $\mathscr{W}^-(\mathbf{m})$ modulo $2Nr$, we do not replace $- \sum_{l=1}^k r_l$ by $ 2Nr - \sum_{l=1}^k r_l$.
\end{remark}

The numerical obstruction for a stratum to have non-empty intersection with the Kummer locus can be formulated as follows:
\begin{prop}\label{prop:numobstr}
If $Z_s$ belongs to the stratum $ \mathscr{S}_{\mathbf{m}} $, then 
$$ \mathscr{W}^+(\mathbf{m}) + \mathscr{W}^-(\mathbf{m}) \equiv 0 \pmod {2Nr}. $$
\end{prop}
\begin{proof}
By assumption
$$ \mathrm{length}(Z_s \cap \mathscr{G}_{2 \tau r + \sum_{l=0}^k r_l}) = m_{2 \tau (b+1) + k}, $$ 
$$ \mathrm{length}(Z_s \cap \mathscr{G}_{2 \tau r - \sum_{l=0}^k r_l}) = m_{2 \tau (b+1) - k} $$
and $ Z_s \cap \mathscr{G}_j = \emptyset $ for any other $j$. Since, by Lemma \ref{lemma:kernelZ}, $Z$ is in the kernel of the summation map for $(\mathcal{Y}'_r)^{sm}_0$,
this means that 
$$ \mathscr{W}^+(\mathbf{m}) + \mathscr{W}^-(\mathbf{m}) \equiv \sum_{j=0}^{2Nr-1} \mathrm{length}(Z_s \cap \mathscr{G}_j) \cdot j \equiv 0 \pmod {2Nr}. \qedhere$$ 
\end{proof}

\subsection{Examples}\label{subsec:Examples}
We provide examples both of strata that are numerically obstructed, as well as unobstructed ones. Let $ n=3 $ and $N=1$, in which case $\mathcal{Y}_0$ has two components. 

Each standard embedding $ C[2] \subset C[3]$ is determined by a condition $ t_i = 1$. Over $0 \in C[2]$, we insert two components over each component of the double locus of $\mathcal{Y}_0$. Hence the fiber of $ \mathcal{Y}[3] \vert_{C[2]} \cong \mathcal{Y}[2] $ over $ 0 \in C[2] $ is a cycle of $6$ components $G^2_i$, where $ 0 \leq i \leq 5 $.

\subsubsection{}
Over the origin of the standard embedding given by $t_4=1$, we have the stratum $ \mathscr{S}_{(0,1,1,1,0,0)}$. It parametrizes subschemes consisting of a point on each of the components $G^2_1, G^2_2$ and $G^2_3$. Since $ 1 + 2 + 3 = 6$, we can let $r_i = 1$ for all $i$. The stratum is therefore unobstructed due to the fact that $\mathcal{Y}[2]$ restricted to any generic line in $C[2]$ is already a Kulikov model. 

\subsubsection{}
In $\mathcal{H}^3$, we now consider two particular strata $ \mathscr{S}_{\mathbf{m}}$ where $ \mathbf{m} = (1,1,1,0,0,0) $, and $ \mathscr{S}_{\mathbf{m}'}$ where $ \mathbf{m}' = (0,2,1,0,0,0) $. They live over the origins in the standard embeddings corresponding to $t_1=1$ and $t_2=1$ respectively. In both cases, we assume there exists a map
$$ \alpha \colon S \to C[2] \subset C[3] $$
and a stable subscheme $ Z \subset \mathcal{Y}[2]_S $ such that $Z_{\eta}$ is in the Kummer locus and the specialization has numerical type either $ \mathbf{m}$ or $ \mathbf{m}'$. We will show that there is an obstruction to the first case occurring, but no obstruction to the second one.

Let $\mathscr{G}_i$ denote the components of $ (\mathcal{Y}'_r)_s $. Then the strict transforms of $G^2_0, G^2_1$ and $G^2_2$ are $ \mathscr{G}_0, \mathscr{G}_{r_1} $ and $\mathscr{G}_{r_1+r_2}$, respectively. The remaining components will play no part in our computation. As in \ref{subsec:restrictionblowup}, we set $ r = r_1 + r_2 + r_3 $. Then the total number of irreducible components in $ (\mathcal{Y}'_r)_s $ equals  $ 2 r$. 

\subsubsection{} We first show that $ \mathscr{S}_{\mathbf{m}}$ is obstructed. The condition is that
$$ 1 \cdot 0 + 1 \cdot r_1 + 1 \cdot (r_1 + r_2 ) = 2 r_1 + r_2 $$
is zero modulo $2 r$.
However, the equation 
$$ 2 r_1 + r_2 = 2 r_1 + 2 r_2 + 2 r_3 $$
can be rewritten 
$$ 0 = r_2 + 2 r_3 $$
which is impossible since $r_i > 0$ for each $i$. A similar argument shows that
$$ 2 r_1 + r_2 = k \cdot 2r $$
has no solutions, for any $k \geq 1$.

\subsubsection{}
There is, however, no obstruction to $ \mathscr{S}_{\mathbf{m}'}$. Indeed, similar arguments as above yield the equality 
$$ r_1 + r_1 + (r_1 + r_2) = 2 r_1 + 2 r_2 + 2 r_3, $$
or equivalently 
$$ r_1 = r_2 + 2 r_3. $$
Here we can choose, for instance, $ (r_1, r_2, r_3) = (3, 1, 1) $. For this choice, $ (\mathcal{Y}'_r)_0^{sm} $ has $10$ components, and $Z_s$ would have length two when restricted to the third component and length one when restricted to the fourth. Of course, this only shows that there is no "numerical" obstruction to the existence of $Z$. To show that such $Z$ do exist requires a finer analysis, and will be discussed in the remainder of this section.

\subsection{Description of the intersection}\label{subsec:stratification}

\subsubsection{} We keep the notation from \ref{subsec:newlabeling}. Let $ C[b] \subset C[n]$ be a standard embedding, and consider a stratum $ \mathscr{S}_{\mathbf{m}}$ of $ \mathcal{H}^n $ over $ 0 \in C[b]$. 
\begin{definition}\label{def:admisstuple}
A tuple $\mathbf{r} = (r_1, \ldots, r_{b+1})$ of strictly positive integers is called \emph{admissible} with respect to $\mathbf{m}$ if 
$$ \mathscr{W}^+(\mathbf{m}) + \mathscr{W}^-(\mathbf{m}) \equiv 0 \pmod {2Nr}. $$
If there exists at least one such admissible tuple, we call the stratum $ \mathscr{S}_{\mathbf{m}}$ \emph{unobstructed}. 
\end{definition}

\subsubsection{}
We fix an admissible tuple $\mathbf{r}$. Consider the map 
$$ C[r-1] \to C[b], $$
which, for each $ 1 \leq k \leq b + 1$, and writing $ \rho_{k} = \sum_{i=0}^{k-1} r_i $, sends
$$ (\ldots, t'_{\rho_{k} + 1}, \ldots, t'_{\rho_{k} + r_k}, \ldots ) \mapsto t_k = \prod_{j=1}^{r_k} t'_{\rho_{k} + j} $$
(recall the convention that $r_0=0$).
It is straightforward to verify that this map is compatible with the respective torus actions under the projection
\begin{equation}\label{eq:torusproj}
 \mathbb G[r-1] \to \mathbb G[b]   
\end{equation}
sending, for $1 \leq k \leq b$,
$$ (\ldots, \sigma'_{\rho_{k} + 1}, \ldots, \sigma'_{\rho_{k} + r_k}, \ldots ) \mapsto (\ldots, \sigma_k = \sigma'_{\rho_{k} + r_k}, \ldots ). $$

\subsubsection{} The pullback $\mathcal{Y}[b]_{C[r-1]}$ has an action by $\mathbb G[r-1]$ via (\ref{eq:torusproj}). In particular, this makes the open immersion (see Proposition \ref{prop:main-blow-up})
\begin{equation}\label{eq:stratuminclusion}
 \mathcal{Y}[b]^{sm}_0 \subset \mathcal{Y}[r-1]^{sm}_0    
\end{equation}
$\mathbb G[r-1]$-equivariant, and we can identify the $\mathbb G[r-1]$-action on the source with the "original" $\mathbb G[b]$-action. Under the inclusion (\ref{eq:stratuminclusion}), the stratum $\mathscr{S}_{\mathbf{m}}$ corresponds to the stratum $\mathscr{S}_{\mathbf{m}(\mathbf{r})}$, where $ \mathbf{m}(\mathbf{r}) $ is defined by:
$$ m(\mathbf{r})_{2 \tau r + \rho_{k+1}} = m_{2 \tau (b+1) + k},$$
$$ m(\mathbf{r})_{2 \tau r - \rho_{k+1}} = m_{2 \tau (b+1) - k} $$
and otherwise $m(\mathbf{r})_j=0$. (Note that $\rho_{k} + r_k = \rho_{k+1}$ by definition.)

\subsubsection{}
We let $L$ denote the diagonal line in $C[r-1]$, and denote 
the group law on $\mathcal{A}(L)_0$ by $+$. For simplicity, we write the summation map associated with $\mathrm{Hilb}^n(\mathcal{Y}[r-1]_0^{sm})$ additively, i.e., 
$$ \bigcup_{i=1}^{n} P_i \mapsto P_1 + \ldots + P_{n}. $$
For any stratum $\mathscr{S}_{\overline{\mathbf{m}}}$ (where $\overline{\mathbf{m}}$ is not necessarily of the form $\mathbf{m}(\mathbf{r})$ considered above)
let $\mu$ denote the restriction of the summation map to $\mathscr{S}_{\overline{\mathbf{m}}}$.
If $\mu$ has target $ \mathcal{G}_0^{\circ}$ we note that
$$\mathscr{S}_{\overline{\mathbf{m}}} \cap \mathrm{Kum}^{n-1}(\mathcal{A}(L)_0) = \mu^{-1}(e), $$ 
with $e \in \mathcal{G}_0^{\circ}$ the identity point.

\subsubsection{}
It is useful to note that $\mathcal{G}_0^{\circ} = G_0^{\circ}$; the identity component does not change under ramified extensions. Moreover, we recall from \ref{subsubsection:explanationcycle} that $T$ denotes the maximal torus in the Chevalley decomposition of the semi-abelian variety $G_0^{\circ}$. 

\begin{lemma}\label{lem:Tfibersirreducible}
Both $\mu^{-1}(e)$ and $\mu^{-1}(T)$ are smooth closed subvarieties of $\mathscr{S}_{\overline{\mathbf{m}}}$.
\end{lemma}
\begin{proof}
Consider first the summation map
$$ \mu_s \colon \prod_{j} \mathrm{Sym}^{\overline{m}_j}(\mathcal{G}^{\circ}_j) \to \mathcal{G}^{\circ}_0. $$
In this case, the summation map is descended from 
$$ \mu_p \colon \prod_j (\mathcal{G}^{\circ}_j)^{\overline{m}_j} \to \mathcal{G}^{\circ}_0 $$
modulo the action of the product of symmetric groups
$ \prod_j S_{m_j} $. One checks that the kernel $\mu_p^{-1}(e) $ is isomorphic to $G_0^{n-1}$, hence $\mu_s^{-1}(e)$ is irreducible as well. 

Next, consider the Hilbert-Chow morphism
$$ h \colon \prod_{j} \mathrm{Hilb}^{\overline{m}_j}(\mathcal{G}^{\circ}_j)  \to \prod_{j} \mathrm{Sym}^{\overline{m}_j}(\mathcal{G}^{\circ}_j), $$
and write $ \mu = \mu_s \circ h $. Then we claim that $ \mu^{-1}(e) $ is the strict transform of $\mu_s^{-1}(e)$ under $h$. Indeed,
$ \mu^{-1}(e)$ equals $ h^{-1}(\mu_s^{-1}(e))$, and it follows from Proposition \ref{prop:Kummernsmooth} that the former is smooth of dimension $\delta(n-1)$. Thus it must be a disjoint union of its irreducible components. By properness of the Hilbert-Chow morphism, the strict transform $\overline{\mu_s^{-1}(e)}$ (which is non-empty) maps surjectively onto $\mu_s^{-1}(e)$. Clearly $\overline{\mu_s^{-1}(e)}$ is a connected component, as it is irreducible of dimension $\delta(n-1)$. The existence of another connected component would contradict the fact that $h$ has connected fibers \cite{Fogarty}. 

To prove the statement for $\mu^{-1}(T)$, recall that we already proved that $\mu$ is smooth, hence the same is true for the pullback $ \mu^{-1}(T) \to T $. Flatness implies that each irreducible component of $ \mu^{-1}(T) $ dominates $T$. Assume that $ \mu^{-1}(T) $ is reducible. Then it is in fact disconnected, and the general fiber over $T$ must be disconnected. But by the previous case, and by translation, each fiber is connected. 
\end{proof}

\subsubsection{}
Let $ \mathbf{m} $ and $ \mathbf{r} $ be as above. We will next compute the restriction of the image of $\mathrm{Kum}^{n-1}(\mathcal{A}(L)_0)$ under the $\mathbb G[r-1]$-action to $\mathscr{S}_{\mathbf{m}(\mathbf{r})}$. In particular, we find it is a closed and $\mathbb G[r-1]$-stable variety.

In the statement and proof of Proposition-Definition  \ref{prop:stratumdfn} below, we will use, for each $ k \in \{1, \ldots, b\}$, the notation 
$$ \alpha^+(\mathbf{m},k) = \sum_{\tau =0}^{N-1} m_{2 \tau (b+1) + k}$$ 
and 
$$ \alpha^-(\mathbf{m},k) = \sum_{\tau =0}^{N-1} m_{2 \tau (b+1) - k}. $$ 
In the proof, we will also use the fact (see Lemma \ref{lemma:translationinterpret}) that the action of $\mathbb G[r-1]$ on each inserted component $\mathcal{G}_{j}^{\circ}$ can be expressed in terms of the action of the subtorus $T \subset G_0^{\circ} $.

\begin{definition-prop}\label{prop:stratumdfn}
The intersection
$$ O(\mathbf{m}, \mathbf{r}) = \mathscr{S}_{\mathbf{m}(\mathbf{r})} \cap (\mathbb G[r-1] \cdot \mathrm{Kum}^{n-1}(\mathcal{A}(L)_0)) $$
is a smooth and $\mathbb G[r-1]$-stable closed subvariety of $\mathscr{S}_{\mathbf{m}(\mathbf{r})}$. Moreover:
\begin{enumerate}
    \item If $ \alpha^+(\mathbf{m},k) =  \alpha^-(\mathbf{m},k) $ for all $k$, then $O(\mathbf{m}, \mathbf{r}) = \mu^{-1}(e)$.  
    \item If $ \alpha^+(\mathbf{m},k) \neq \alpha^-(\mathbf{m},k) $ for at least one $k$, then $O(\mathbf{m}, \mathbf{r}) = \mu^{-1}(T)$.
\end{enumerate}
We say that $O(\mathbf{m}, \mathbf{r})$ is \textbf{narrow} in case $(1)$ and \textbf{wide} in case $(2)$.
\end{definition-prop}
\begin{proof}
Assume that 
$ P_1 + \ldots + P_{n} = e $. If $\sigma = (\sigma_1, \ldots, \sigma_{r-1})$ in $ \mathbb G[r-1]$, then 
\begin{equation}\label{eqn:action}
    \sigma(P_1) + \ldots + \sigma(P_{n}) = \sum_i P_i + \sum_{k=1}^b (\alpha^+(\mathbf{m},k) - \alpha^-(\mathbf{m},k)) \sigma_k. 
\end{equation}
Over the generic line $L' = \sigma(L)$ we have a group law $+' = +_{L'}$ which in general is different from $+$ when restricted to the special fibers. However, Proposition \ref{prop:genline-hom} states that $\sigma \colon \mathcal{A}(L)_0 \to \mathcal{A}(L')_0$ is a group isomorphism, which implies that
$$ e = \sigma(P_1 + \ldots + P_{n}) = \sigma(P_1) +' \ldots +'\sigma(P_{n}). $$
    
In the first case, we consequently find that
$$ e = \sigma(P_1) + \ldots + \sigma(P_{n}) = \sigma(P_1) +' \ldots +' \sigma(P_{n}). $$
In other words, $ \mu^{-1}(e)$ is stable under $\mathbb G[r-1]$, and does not depend on the choice of generic line.
In the second case we put 
$$ \lambda = \sum_{k=1}^b (\alpha^+(\mathbf{m},k) - \alpha^-(\mathbf{m},k)) \sigma_k. $$
Using Equation (\ref{eqn:action}), one then concludes that $\sigma$ takes $\mu^{-1}(e)$ to  $\mu^{-1}(\lambda)$. Since $\lambda$ can be made arbitrary, this sweeps out $\mu^{-1}(T) $, as asserted. 

The schemes $\mu^{-1}(e)$ and $\mu^{-1}(T) $ are smooth, closed and irreducible by Lemma \ref{lem:Tfibersirreducible}. By the description and computations above, we see that they are also stable under the $\mathbb G[r-1]$-action, and, moreover, that, in the respective situations, they coincide with $ O(\mathbf{m}, \mathbf{r}) $. 
\end{proof}

\subsection{The stratification}\label{sec:stratfinal}

Using the equivariant identification of $ \mathscr{S}_{\mathbf{m}} $ with $\mathscr{S}_{\mathbf{m}(\mathbf{r})} $, we can see $\mathrm{Kum}^{n-1}(\mathcal{A}(L)_0)$ as a closed subscheme of $ \mathscr{S}_{\mathbf{m}} $ (and hence of the stable locus $\mathcal{H}^n$), and identify its $\mathbb G[b]$-orbit with $O(\mathbf{m},\mathbf{r})$.

\subsubsection{}

We first give a lemma showing that the orbit associated to each unobstructed stratum does not depend on the choice of admissible tuple. 

\begin{lemma}\label{lem:independencer}
The subscheme $O(\mathbf{m}, \mathbf{r})$ 
of $ \mathscr{S}_{\mathbf{m}} $ is independent of $\mathbf{r}$. 
\end{lemma}
\begin{proof}

For any $ a = a' \cdot r$, with $a'>0$, let
$$ \mathfrak{a} \colon C[a-1] \to C[r-1] $$
be the map sending, for each $ 0 \leq m < r$, 
$$ u_{ma' + 1}, \ldots, u_{ma' + a'} \mapsto t_{m+1} = \prod_{l=1}^{a'} u_{ma' + l}. $$

The diagonal line $\tilde{L} \subset C[a-1]$ maps to $L$ under $\mathfrak{a}$. Pulling back to $\tilde{L}$ yields maps (see Proposition \ref{prop:main-blow-up})
$$ \mathcal{Y}[a-1]_{\tilde{L}} \to \mathcal{Y}[r-1]_{\tilde{L}} \to \mathcal{Y}[b]_{\tilde{L}}. $$
Both maps are trivial when restricted to the smooth locus of the respective targets, so on special fibers we obtain inclusions
\begin{equation}\label{eqn:prooffiberinclusions}
\mathcal{Y}[b]^{sm}_0 \subset \mathcal{Y}[r-1]^{sm}_0 \subset \mathcal{Y}[a-1]^{sm}_0.
\end{equation}
The latter inclusion can also be seen as the (special fiber of the) base change map
$$ \mathcal{A}(L) \times_L \tilde{L} \to \mathcal{A}(\tilde{L}). $$
Since $ \mathcal{A}(L)_s \subset \mathcal{A}(\tilde{L})_s$ is a \emph{subgroup}, we can, for our purposes, replace $C[r-1]$ with $C[a-1]$ (with $\mathbf{r}$ suitably modified). 

Letting $a'$ be a multiple of $b+1$, we can construct a map $\mathfrak{a}_0 \colon C[a-1] \to C[b]$ similarly to $\mathfrak{a}$ above. It maps $\tilde{L}$ to the diagonal line $L_0$ in $ C[b]$ and after pullback to $\tilde{L}$ we obtain a subgroup inclusion 
\begin{equation}\label{eqn:prooffiberinclusions2}
\mathcal{A}(L_0)_s \subset \mathcal{A}(\tilde{L})_s.   
\end{equation}
Observe, however, that in general (\ref{eqn:prooffiberinclusions}) and (\ref{eqn:prooffiberinclusions2}) are \emph{different} inclusions of $\mathcal{Y}[b]^{sm}_0 $ in $ \mathcal{Y}[a-1]^{sm}_0$. We shall next explain how they are related.

We denote by $G_i^{\circ}$ the components of $ \mathcal{A}(\tilde{L})_s$. 
Recall that we can find an isomorphism $\Phi(\tilde{L}) \cong \mathbb{Z}/2Na$, where $G_i^{\circ}$ corresponds to the class of $i$, and, moreover, an isomorphism 
$$ \mathcal{A}(\tilde{L})_s = G_0^{\circ} \times \mathbb{Z}/2Na $$ 
identifying $G_i^{\circ} = G_0^{\circ} \times \{i\}$. Then the inclusion in  (\ref{eqn:prooffiberinclusions2}) corresponds to the subgroup $G_0^{\circ} \times \frac{a}{b+1}\mathbb{Z}/2Na$. Consequently, each component $G_0^{\circ} \times \{l \cdot \frac{a}{b+1} \}$ of $\mathcal{A}(L_0)_s$ gets shifted to $G_0^{\circ} \times \{j_l\}$ under \ref{eqn:prooffiberinclusions}, by the map $ (g,l \cdot \frac{a}{b+1}) \mapsto (g,j_l)$ for some appropriate $j_l$. 

The summation map $ \mathrm{Hilb}^n(\mathcal{A}(\tilde{L})_s) \to \mathcal{A}(\tilde{L})_s$ restricts to 
$$ \prod_{l=0}^{2N(b+1)-1} \mathrm{Hilb}^{m_l}((G_0^{\circ}, j_l)) \to (G_0^{\circ},0), $$
whose kernel we denote by $\mathbf{K}_r$.

If $\mathbf{r}' $ is another admissible tuple, we can take $a = r'r(b+1)$, and similarly have integers $j'_l$ and a kernel $\mathbf{K}_{r'}$. However, via the natural identifications $G_i^{\circ} = (G_0^{\circ},i) \to G_0^{\circ} $, $ (g,i) \mapsto g$, both $\mathbf{K}_r$ and $\mathbf{K}_{r'}$ equal the kernel of 
$$ \prod_{l=0}^{2N(b+1)-1} \mathrm{Hilb}^{m_l}(G_0^{\circ}) \to G_0^{\circ}. $$
Indeed, this is immediate from the fact that $\sum_l m_l \cdot j_l$ and $\sum_l m_l \cdot j'_l$ both are zero in $\Phi(\tilde{L})$.
\end{proof}

Because of the above lemma, we simply write $O(b, \mathbf{m})$.

We say that the pair $(b,\mathbf{m})$ is unobstructed if $ C[b] \to C[n] $ is a standard embedding, and if $ \mathscr{S}_{\mathbf{m}} $ is a stratum of $\mathcal{H}^n$ over $0 \in C[b]$, which is unobstructed. Then $O(b, \mathbf{m})$ is defined as in Proposition \ref{prop:stratumdfn}. 

Let $I \subset \{1, \ldots, n+1\}$ be the subset of cardinality $b+1$ corresponding to the given standard embedding. Then, as explained in \ref{subsubseq:standardemb}, restricting to the (locally closed) subvariety $\mathcal{U}(I)_0 \subset C[n]$ yields a product
$$ \mathcal{Y}[n] \vert_{\mathcal{U}(I)_0} = \mathcal{Y}[b]_0 \times \mathcal{U}(I)_0. $$
Moreover, we can write the torus accordingly as a product $\mathbb G[n] = \mathbb G[b]  \times \mathbb G[n-b]$ of subtori acting on the respective factors (in fact, $\mathcal{U}(I)_0$ is a torsor under $\mathbb G[n-b]$).

From this description, we immediately find that the restriction of $\mathcal{K}^{n-1}$ to $ \mathcal{U}(I)_0 $
equals the $\mathbb G[n]$-stable, smooth and irreducible
scheme
\begin{equation}\label{eq:stratafull}
\mathcal{K}(b, \mathbf{m}) = O(b, \mathbf{m}) \times \mathcal{U}(I)_0. 
\end{equation}

\subsubsection{}
We can now formulate the main theorem of this section. 
\begin{theorem}\label{theorem:mainstrat}
The boundary of the Kummer locus
admits the stratification
$$ \mathcal{K}^{n-1} \setminus \mathcal{K}_{\circ}^{n-1} = \bigcup_{b, \mathbf{m}} \mathcal{K}(b, \mathbf{m}), $$
where $(b, \mathbf{m})$ runs over all unobstructed pairs and where each stratum $\mathcal{K}(b, \mathbf{m})$ is smooth, irreducible and $\mathbb G[n]$-stable.
\end{theorem}
\begin{proof}
The strata $\mathcal{K}(b, \mathbf{m})$ are defined as in (\ref{eq:stratafull}) above. Then the result follows from Lemma \ref{lemma:kernelZ}, Propositions \ref{prop:numobstr} and \ref{prop:stratumdfn}, and Lemma \ref{lem:independencer}.  
\end{proof}

\section{Degeneration from combinatorial point of view}\label{Sec:linechart}

We follow the notations in the previous section, namely, $\delta$ is the relative dimension of the family of abelian varieties $\mathcal{Y}/C$, $2N$ is the number of irreducible components in the degenerate fiber, $n$ is the number of points in a configuration represented by a closed point in the relative Hilbert scheme or Kummer variety. Let $K^{n-1}_{\mathcal{Y}/C} $ denote the GIT quotient of the Kummer locus $\mathcal{K}^{n-1} $. In this section we will introduce the dual complex of the central fiber $(\mathcal{K}^{n-1}_{\mathcal{Y}/C})_0$, denoted $\Delta(\mathcal{K}^{n-1}_{\mathcal{Y}/C})$ for simplicity, and give an algorithm for computing it.

For this purpose, we associate what we call a \emph{line chart} to each stratum in the degenerate fiber of the family of Kummer varieties. This is a combinatorial object that contains much information about the stratum; for example, we can tell immediately from the line chart whether the stratum is admissible, whether it is wide or narrow, how to smooth or specialize it, etc. Moreover, it is also a helpful tool for us to study the dual complex of the degenerate fiber; for instance, it helps to count the number of cells in the dual complex in each dimension, the shape of the cells, and the relation between the dual complexes for the degenerations of Kummer varieties and Hilbert schemes.

\subsection{Preliminaries}\label{subsec:prelsstratadescr}
We shall continue to use, for any $b$, the notation introduced in Section \ref{subsubsec:notationcomp} (and further expanded in Section \ref{subsec:newlabeling}) for the components of $\mathcal{Y}[b]_0$. However, we will introduce a more convenient notation for the strata of the boundary of $\mathcal{K}^{n-1}$.

\subsubsection{Notation}
Let $ b \leq n $, and let $x_i$, for $ i \in \{1,2, \dots, n\} $, be integers such that $ - b - 1 \leq x_i \leq b $. We require that
$$ \vert x_1 \vert \leq \vert x_2 \vert \leq \dots \leq \vert x_n \vert. $$
We shall often write $ \mathfrak{X} = \{x_1,x_2, \dots, x_n\}$ and $ \vert \mathfrak{X} \vert = \{\vert x_1 \vert, \vert x_2 \vert, \dots, \vert x_n \vert \}$.

We denote by $\tau = (\tau_1, \tau_2, \dots, \tau_n)$ a tuple of not necessarily distinct integers satisfying $ 0 \leq \tau_i \leq N-1 $, and write $ \mathfrak{T} = \tau_1 + \tau_2 + \dots + \tau_3$.

Let $ C[b] \to C[n] $ be a standard embedding. In Section \ref{sec:stratfinal}, we used the notation $ \mathcal{K}(b, \mathbf{m})$ for strata of $\mathcal{K}^{n-1}$ over $ (0 \in C[b]) \times \mathbb G[n-b] $. We now introduce an alternative notation, which is better suited for the various computations needed for precisely describing the irreducible components in the special fiber of $\mathcal{K}^{n-1} \to C$, and their intersections.

\begin{definition}\label{def:stratumnew}
We denote by $ X_b(\tau, x_1,x_2, \dots ,x_n) $ the stratum of $\mathcal{K}^{n-1}$ whose generic point over $0 \in C[b]$ is a reduced subscheme $ \bigcup_i P_i $ with $$ P_i \in G'_{\tau_i, x_i} = G'_{2 \tau_i (b+1) + x_i} $$ for each $i = 1, 2, \dots, n$. The corresponding stratum of $\mathcal{H}^n$ will be denoted $ \mathcal{X}_b(\tau, x_1,x_2, \dots, x_n) $. 
\end{definition}

Whenever $\tau = (\tau_1,\tau_2, \dots, \tau_n)$ is clear from the context, it will be dropped from the notation. We remark that even though the $x_i$'s are ordered according to absolute values, we do not impose such a condition on the $\tau_i$'s. In fact, as we shall see later, when $n=3$, studying the effect of permuting the entries of $(\tau_1,\tau_2,\tau_3)$ is fundamental for understanding the dual complex. 

\begin{example}
We give a brief example to illustrate the new notation. We keep the notation and assumptions from Section \ref{subsec:Examples}. Recall that for $\mathbf{m}' = (0,2,1,0,0,0) $, we saw that the stratum $ \mathcal{K}(2,\mathbf{m}')$ is contained in $ \mathcal{K}^2 $. In the notation of Definition \ref{def:stratumnew} this is the stratum $ X_2(\tau, 1,1,2)$, where $ \tau = (0,0,0) $. 
\end{example}

\subsubsection{Dimensions}
We give a precise formula for the dimension of each stratum of $\mathcal{K}^{n-1}$.

\begin{lemma}\label{lemma:stratdimension}
Under the above notations we have
$$ \dim X_b(x_1, x_2, \dots, x_n) = \delta (n-1) + (n-b) + \varepsilon, $$
where $\varepsilon=0$ for a narrow stratum and $\varepsilon=1$ for a wide stratum.

In particular, assuming $n=3$, we have $\varepsilon = 0$ if and only if either 
\begin{enumerate}
	\item $b=0$; or
	\item $b=1$ and $\mathfrak{X} = \{-1,1, x\} $, where $x$ is even.
\end{enumerate}
\end{lemma}
\begin{proof}
The formula for the dimension follows from the observation that $\dim \mathbb G[n-b] = n - b$. The value of $\varepsilon$ is determined by Proposition \ref{prop:stratumdfn}.

In the following we assume $n=3$. If $b=0$, we necessarily have $\varepsilon = 0$. If $b=3$, stability implies that $ \vert \mathfrak{X} \vert = \{1,2,3\} $, hence $\varepsilon = 1$. If $b=2$, stability implies that $\{1,2 \} \subset \vert \mathfrak{X} \vert$, and again we necessarily have $\varepsilon = 1$.

It remains to treat $b=1$, which is slightly different. We start by recalling that, by stability (see Section \ref{subsec:expandeddegen}), $ 1 \in \vert \mathfrak{X} \vert$. Assume first that $ \{-1,1\} \subset \mathfrak{X} $ and that the unique element of $ \mathfrak{X} \setminus \{-1,1\} $ is $0$ modulo $2$. Then $\varepsilon = 0$. In all remaining cases, it is straightforward to verify that $\varepsilon = 1$.
\end{proof}

An immediate consequence of the above result is the following
\begin{corollary}
	The dimension of the image of $X_b(x_1, x_2, \dots, x_n)$ in the GIT quotient $K^{n-1}_{\mathcal{Y}/C}$ is 
\begin{equation*}
	\delta (n-1) -b + \varepsilon,
\end{equation*}
where $\varepsilon=0$ for a narrow stratum and $\varepsilon=1$ for a wide stratum. In particular, the maximal (resp. minimal) possible dimension is $\delta(n-1)$ (resp. $(\delta -1)(n-1)$).
\end{corollary}

\begin{proof}
	The dimension formula follows immediately from Lemma \ref{lemma:stratdimension}. Since $0 \leq b \leq n$, the extremal dimensions follows from the observation that a stratum with $b=0$ (resp. $b=n$) must be narrow (resp. wide).
\end{proof}

For later convenience, we introduce another simplification of our notation, which we find useful when performing various computations. Namely, we set $ X_0 = A $, $X_1 = B$, $X_2 = C$, $X_3 = D$, etc.

\subsection{Deepest strata}\label{subsec:deepeststrata}

In this section, we introduce the combinatorial object ``line chart'' for describing strata, and apply it to study the deepest strata. We start with a general observation.

\subsubsection{Condition for the existence of solutions}

The following elementary lemma will be used to determine the admissibility of strata.

\begin{lemma}\label{lem:existence-solutions-general}
	Let $m \in \Z_+$ and $b_1, \dots, b_m \in \Z$. Then the equation
	\begin{equation}\label{eqn:general-indefinite}
		r_1 b_1 + \dots + r_m b_m = 0
	\end{equation}
	has solutions $r_1, \dots, r_m \in \Z_+$ if and only if
	\begin{equation}\label{eqn:two-lines-general}
	\begin{aligned}
		\text{either} \qquad &\min\{b_1, \dots, b_m\} < 0 < \max\{b_1, \dots, b_m\} \\
		\text{or} \qquad &\min\{b_1, \dots, b_m\} = 0 = \max\{b_1, \dots, b_m\}.
	\end{aligned} 
	\end{equation}
\end{lemma}

\begin{proof}
	For the necessity of \eqref{eqn:two-lines-general}, we assume that \eqref{eqn:general-indefinite} has a solution in positive integers. Then either the left side of the equation contains both positive and negative terms, or all terms are zero, hence condition \eqref{eqn:two-lines-general} holds.
	
	We turn to the sufficiency of \eqref{eqn:two-lines-general}. In fact, in the first case of \eqref{eqn:two-lines-general}, we assume $b_{i_-}<0<b_{i_+}$. First we set
	$$ \bar{r}_i = 1 \qquad \text{for each} \qquad i \neq i_+ \text{ or } i_-. $$
	Then we choose $ \bar{r}_{i_-} $ to be sufficiently large, such that
	$$ C \coloneqq \sum_{\substack{i=1 \\ i \neq i_+}}^{n+1} \bar{r}_ib_i < 0. $$
	Finally we observe that a solution to \eqref{eqn:general-indefinite} is given by
	$$ r_i = b_{i_+} \bar{r}_i \quad \text{for} \quad i \neq i_+ \qquad \text{and} \qquad r_{i_+} = -C. $$
	In the second case of \eqref{eqn:two-lines-general}, it is clear that a solution to \eqref{eqn:general-indefinite} exists.
\end{proof}

\subsubsection{Line chart in the case of $b=n$}

When $b=n$, the stability condition requires that $\lvert x_i \rvert = i$ for $i = 1, \dots, n$, hence we assume that $x_i = \varepsilon_i i$ for $\varepsilon_i \in \{\pm 1\}$. For such a stratum to be admissible, two conditions have to be satisfied: one condition concerning the values of $x_i$'s and the other concerning $\tau_i$'s. We first consider the condition for $x_i$'s, which is given by the equation
\begin{equation}\label{eqn:admissible-first}
    \varepsilon_1 r_1 + \varepsilon_2 (r_1 + r_2) + \dots + \varepsilon_n (r_1 + \dots + r_n) = 2kr,
\end{equation}
where $r = r_1 + \dots + r_{n+1}$ and $k \in \Z$. We rearrange the equation as
	\[
	    r_{n+1} \cdot 0 + r_n \cdot \varepsilon_n + r_{n-1} \cdot (\varepsilon_{n-    1} + \varepsilon_n) + \dots + r_1 \cdot (\varepsilon_1 + \dots + \varepsilon_n) = 2kr
	\]
and consider the set
$$ S = \{0, \varepsilon_n, \varepsilon_{n-1} + \varepsilon_n, \dots, \varepsilon_1 + \dots + \varepsilon_n\} $$
of coefficients on the left side. Observe that each element in $S$ is either $1$ larger or $1$ smaller than its previous element. Since the elements in $S$ cannot be all equal, we conclude by Lemma \ref{lem:existence-solutions-general} that \eqref{eqn:admissible-first} has solutions $r_1, \dots, r_{n+1} \in \Z_+$ if and only if
\begin{equation}\label{eqn:lc-condition}
	\min S < 2k < \max S.
\end{equation}

To visualize this condition, we associate to each variable $r_i$ an integral point $$(n+1-i, \varepsilon_i + \dots + \varepsilon_n)$$ in the coordinate plane, and connect all these vertices
\begin{equation}\label{eqn:points-line-chart}
    V = \{ (0,0),\ (1, \varepsilon_n),\ (2, \varepsilon_{n-1} + \varepsilon_n),\ \dots,\ (n, \varepsilon_1 + \dots + \varepsilon_n) \}
\end{equation}
by line segments from left to right. The resulting picture is called a \emph{line chart} associated to the stratum. The solvability condition \eqref{eqn:lc-condition} can be geometrically interpreted as the condition that the line chart is properly intersected by the horizontal line $y = 2k$, which we call the \emph{neutral line} associated to the line chart. A line chart satisfying this condition is said to be \emph{admissible}; or equivalently, it is a line chart associated to an admissible stratum.

Moreover, we note that in the case of $b=n$, each arrow in the line chart points diagonally to the next integral point in the direction of either northeast or southeast, so the first coordinates of its vertices as listed in \eqref{eqn:points-line-chart} exhaust all integers between $0$ and $n$. We say such a line chart is \emph{complete}. Later in Section \ref{subsec:higherdimstrata} we will also define line charts for higher dimensional strata, which will no longer be complete.

For the stratum to be admissible, it remains to determine the $\tau_i$'s for each fixed set of values $(\varepsilon_1, \dots, \varepsilon_n, k)$. They have to satisfy the condition
\begin{equation}\label{eqn:admissible-second}
    \tau_1 + \dots + \tau_n + k \equiv 0 \pmod N,
\end{equation}
where each $\tau_i$ is a residue class modulo $N$. Solutions certainly exists for each value of $k$.

\subsubsection{Example of computation for $n=3$}

The following concrete example shows how these combinatorial objects lead to geometric results.

\begin{example}\label{eg:complete-admissible}
    For $n=3$, we exhibit all line charts (together with the neutral line $y=2k$) that satisfy the condition \eqref{eqn:lc-condition} in Figure \ref{fig:admissible3}. Note that all these line charts spread on both sides of the neutral line. 
    
\begin{figure}
	\begin{tikzpicture}
        \draw[gray!30, step=1] (0,0) grid (3,3);
        \node at (-0.3,0) {$O$};
        \draw[thick] (0,0) -- (1,1);
        \draw[thick] (1,1) -- (2,2);
        \draw[thick] (2,2) -- (3,3);
        \draw (0,2) -- (3,2);
        \node at (3.75,2) {$y=2$};
        \node at (1.5,-0.5) {$D(1,2,3)$};
        \filldraw (0,0) circle [radius=2pt];
        \filldraw (1,1) circle [radius=2pt];
        \filldraw (2,2) circle [radius=2pt];
        \filldraw (3,3) circle [radius=2pt];
    \end{tikzpicture}
    \begin{tikzpicture}
        \draw[gray!30, step=1] (0,-1) grid (3,1);
        \node at (-0.3,0) {$O$};
        \draw[thick] (0,0) -- (1,1);
        \draw[thick] (1,1) -- (2,0);
        \draw[thick] (2,0) -- (3,-1);
        \draw (0,0) -- (3,0);
        \node at (3.75,0) {$y=0$};
        \node at (1.5,-1.5) {$D(-1,-2,3)$};
        \filldraw (0,0) circle [radius=2pt];
        \filldraw (1,1) circle [radius=2pt];
        \filldraw (2,0) circle [radius=2pt];
        \filldraw (3,-1) circle [radius=2pt];
    \end{tikzpicture}
    \begin{tikzpicture}
        \draw[gray!30, step=1] (0,-1) grid (3,1);
        \node at (-0.3,0) {$O$};
        \draw[thick] (0,0) -- (1,-1);
        \draw[thick] (1,-1) -- (2,0);
        \draw[thick] (2,0) -- (3,1);
        \draw (0,0) -- (3,0);
        \node at (3.75,0) {$y=0$};
        \node at (1.5,-1.5) {$D(1,2,-3)$};
        \filldraw (0,0) circle [radius=2pt];
        \filldraw (1,-1) circle [radius=2pt];
        \filldraw (2,0) circle [radius=2pt];
        \filldraw (3,1) circle [radius=2pt];
    \end{tikzpicture}
    \begin{tikzpicture}
        \draw[gray!30, step=1] (0,0) grid (3,-3);
        \node at (-0.3,0) {$O$};
        \draw[thick] (0,0) -- (1,-1);
        \draw[thick] (1,-1) -- (2,-2);
        \draw[thick] (2,-2) -- (3,-3);
        \draw (0,-2) -- (3,-2);
        \node at (3.75,-2) {$y=-2$};
        \node at (1.5,-3.5) {$D(-1,-2,-3)$};
        \filldraw (0,0) circle [radius=2pt];
        \filldraw (1,-1) circle [radius=2pt];
        \filldraw (2,-2) circle [radius=2pt];
        \filldraw (3,-3) circle [radius=2pt];
    \end{tikzpicture}
    \caption{Admissible line charts for $n=b=3$ in Example \ref{eg:complete-admissible}}
    \label{fig:admissible3}
\end{figure}
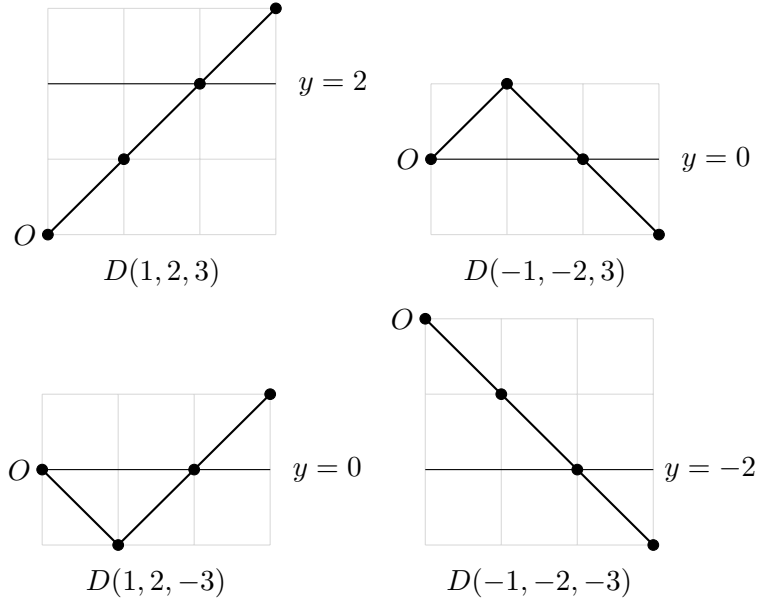
\end{example}

In geometric language, the above example leads immediately to the following result

\begin{prop}\label{prop:stratadeterm}
The inclusion $ D(\tau, x_1,x_2,x_3) \subset \mathcal{K}^2 $ holds if and only if one of the following conditions holds
\begin{enumerate}
    \item $ \mathfrak{X} = \pm \{1,2,-3\} $ and $ \mathfrak{T} \in \{ 0, N, 2N\}$, or
    \item $ \mathfrak{X} = \{1,2,3\} $ and $ \mathfrak{T} \in \{N-1, 2N-1\}$, or
    \item $\mathfrak{X} = \{-1,-2,-3\} $ and $ \mathfrak{T} \in \{1, N+1, 2N+1\}$.
\end{enumerate}

\end{prop}
\begin{proof}
	The choices for $\mathfrak{X}$ follows from \eqref{eqn:lc-condition} or Example \ref{eg:complete-admissible}, and the corresponding choices for $\mathfrak{T}$ follows from \eqref{eqn:admissible-second}.
\end{proof}

\begin{remark}
We point out that when $N \leq 3$, not all solutions $\mathfrak{T}$ will occur. This is due to the fact that $ 0 \leq \tau_i \leq N-1$, so that $\mathfrak{T} \leq 3N-3$. But the inequality $ 3 N - 3 \geq 2N+1$ implies that $N\geq 4$. 
\end{remark}

\subsubsection{Number of the deepest strata for $n \geq 3$}

As an application of the previous discussion, we count the number of deepest strata in the relative Kummer variety. 

A deepest stratum is a stratum of the minimal dimension. When $n \geq 3$, it is clear that a deepest stratum is precisely an admissible stratum with $b=n$ (which turns out to be a wide stratum). On the contrary, when $n=2$, the deepest strata are narrow and appear in the expansion with $b=1<n$.

In the following we will determine the number of strata in the case of $b=n$ for arbitrary $\delta$, $N$, and $n$, which is precisely the number of the deepest strata when $n \geq 3$. The case of $n=2$ will be discussed in Section \ref{subsec:Kummersurfaces} in detail.

\begin{prop}\label{prop:number-deepest}
    The number of admissible strata that require 
    an expansion of level $n$ is $N^{n-1} a_n$, where $a_n$ is recursively defined by $$a_{n+2} = 4a_n + 4 \binom{n}{\left[ \frac{n}{2} \right]}$$ with initial values $a_1=a_2=0$. For $n \geq 3$, it is also the number of the deepest strata. 
\end{prop}

\begin{proof}    
    First of all, we will show that the number of choices for $(\varepsilon_1, \dots, \varepsilon_n; k)$ satisfying \eqref{eqn:lc-condition} is indeed $a_n$, as defined in the statement of this result. 
    
    For $n=1$ (resp. $2$), there are only $2$ (resp. $4$) complete line charts of length $n$, none of which are properly intersected by a neutral line $y=2k$ for any $k \in \mathbb{Z}$. Therefore the numbers of admissible line charts, or equivalently the numbers of admissible strata, are indeed given by $a_1 = a_2 = 0$.

    It remains to verify the recursive formula. There are $4$ possibilities for the first two segments of the line chart of length $n+2$, which are given by all possible combinations of signs $(\varepsilon_{n+2}, \varepsilon_{n+1})$.
    
    For the case of $(\varepsilon_{n+2}, \varepsilon_{n+1}) = (1,1)$, we can view the rest of the line chart as a line chart of length $n$ with its starting point shifted to $(2,2)$. There are two different types of admissible line charts of this kind.
    
    For the first type, we can shift the starting point of an admissible line chart of length $n$ to $(2,2)$ to produce an admissible line chart of length $n+2$ with $(\varepsilon_{n+2}, \varepsilon_{n+1}) = (1,1)$. In other words, any combination $(\varepsilon_1, \dots, \varepsilon_n; k)$ satisfying \eqref{eqn:lc-condition} produces a combination $(\varepsilon_1, \dots, \varepsilon_n, 1, 1; k+1)$ satisfying the same condition with $n$ replaced by $n+2$. Figure \ref{fig:first-type-3-5} exhibits such an admissible line chart of length $5$ obtained from an admissible line chart of length $3$. The number of such line charts is $a_n$. 
   
\begin{figure}
    \begin{tikzpicture}
        \draw[gray!30, step=1] (0,-1) grid (3,1);
        \node at (-0.3,0) {$O$};
        \draw[thick] (0,0) -- (1,-1);
        \draw[thick] (1,-1) -- (2,0);
        \draw[thick] (2,0) -- (3,1);
        \draw (0,0) -- (3,0);
        \node at (3.75,0) {$y=0$};
        \node at (1.5,-1.5) {admissible with $n=3$};
        \filldraw (0,0) circle [radius=2pt];
        \filldraw (1,-1) circle [radius=2pt];
        \filldraw (2,0) circle [radius=2pt];
        \filldraw (3,1) circle [radius=2pt];
    \end{tikzpicture}
    \begin{tikzpicture}
        \draw[gray!30, step=1] (0,-1) grid (5,3);
        \node at (-0.3,0) {$O$};
        \draw[thick,dashed] (0,0) -- (2,2);
        \draw[thick] (2,2) -- (3,1);
        \draw[thick] (3,1) -- (5,3);
        \draw (0,2) -- (5,2);
        \node at (5.75,2) {$y=2$};
        \node at (2.5,-1.5) {admissible with $n=5$};
        \filldraw (0,0) circle [radius=2pt];
        \filldraw (1,1) circle [radius=2pt];
        \filldraw (2,2) circle [radius=2pt];
        \filldraw (3,1) circle [radius=2pt];
        \filldraw (4,2) circle [radius=2pt];
        \filldraw (5,3) circle [radius=2pt];
    \end{tikzpicture}
    \caption{Example of an admissible line chart of length $5$ obtained from an admissible line chart of length $3$ by shifting its starting point from $(0,0)$ to $(2,2)$}
    \label{fig:first-type-3-5}
\end{figure}
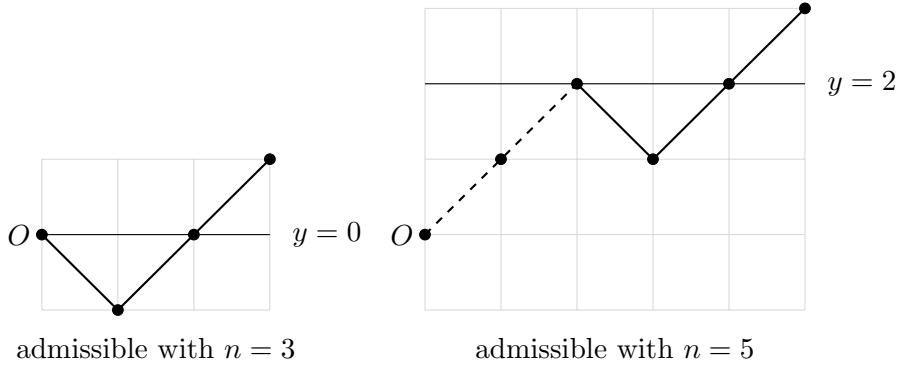
    
    For the second type, an inadmissible line chart of length $n$ could also become admissible by shifting it starting point from $(0,0)$ to $(2,2)$. Such an inadmissible line chart of length $n$ must lie in the region given by $y \geq 0$ with the corresponding $k=0$. In other words, any combination $(\varepsilon_1, \dots, \varepsilon_n; 0)$ satisfying the conditions $\varepsilon_i + \dots + \varepsilon_n \geq 0$ for each $i$ produces a combination $(\varepsilon_1, \dots, \varepsilon_n, 1, 1; 1)$ that satisfies \eqref{eqn:lc-condition}. Figure \ref{fig:first-type-not3-5} exhibits such an admissible line chart of length $5$ obtained from an inadmissible line chart of length $3$.
    
    The above discussion shows that the number of admissible line charts of length $n+2$ of the second type is the number of all line charts of length $n$ that lie in the region $y \geq 0$ (which are necessarily inadmissible with $k=0$), which turns out to be the classical ``ballot number'' $\binom{n}{\left[ \frac{n}{2} \right]}$. Indeed, we assume that such a line chart consists of $q$ line segments of slope $-1$ and $n-q$ line segments of slope $1$, then $0 \leq q \leq \left[ \frac{n}{2} \right]$. For each fixed integer value of $q$ in this interval, the number of line charts starting at $(0,0)$, ending at $(n-q,q)$ and lying in the region $y \geq 0$ is $\binom{n}{q} - \binom{n}{q-1}$, which can be easily obtained by the ``reflection principle''; see e.g. \cite[Lemma on page 72]{Feller-1968}. By summing over all possible values of $q$, we obtain the total number of such line charts, which equals $\binom{n}{\left[ \frac{n}{2} \right]}$.

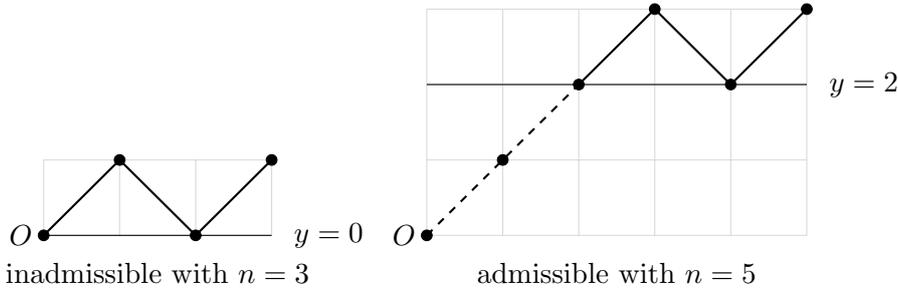
\begin{figure}
    \begin{tikzpicture}
        \draw[gray!30, step=1] (0,0) grid (3,1);
        \node at (-0.3,0) {$O$};
        \draw[thick] (0,0) -- (1,1);
        \draw[thick] (1,1) -- (2,0);
        \draw[thick] (2,0) -- (3,1);
        \draw (0,0) -- (3,0);
        \node at (3.75,0) {$y=0$};
        \node at (1.5,-0.5) {inadmissible with $n=3$};
        \filldraw (0,0) circle [radius=2pt];
        \filldraw (1,1) circle [radius=2pt];
        \filldraw (2,0) circle [radius=2pt];
        \filldraw (3,1) circle [radius=2pt];
    \end{tikzpicture}
    \begin{tikzpicture}
        \draw[gray!30, step=1] (0,0) grid (5,3);
        \node at (-0.3,0) {$O$};
        \draw[thick,dashed] (0,0) -- (2,2);
        \draw[thick] (2,2) -- (3,3);
        \draw[thick] (3,3) -- (4,2);
        \draw[thick] (4,2) -- (5,3);
        \draw (0,2) -- (5,2);
        \node at (5.75,2) {$y=2$};
        \node at (2.5,-0.5) {admissible with $n=5$};
        \filldraw (0,0) circle [radius=2pt];
        \filldraw (1,1) circle [radius=2pt];
        \filldraw (2,2) circle [radius=2pt];
        \filldraw (3,3) circle [radius=2pt];
        \filldraw (4,2) circle [radius=2pt];
        \filldraw (5,3) circle [radius=2pt];
    \end{tikzpicture}
    \caption{Example of an admissible line chart of length $5$ obtained from an inadmissible line chart of length $3$ by shifting its starting point from $(0,0)$ to $(2,2)$}
    \label{fig:first-type-not3-5}
\end{figure}

    Combining the numbers of admissible line charts of both types, there are precisely $a_n + \binom{n}{\left[ \frac{n}{2} \right]}$ values satisfying \eqref{eqn:lc-condition} with $(\varepsilon_{n+2}, \varepsilon_{n+1}) = (1,1)$.

    Similarly, for any of the other choices $$(\varepsilon_{n+2}, \varepsilon_{n+1}) = (1,-1), (-1,1), (-1,-1),$$ the number of admissible line charts is the same as the case of $(\varepsilon_{n+2}, \varepsilon_{n+1})=(1,1)$. The recursive formula for $a_n$ follows.

    In our situation the $x_i$'s are distinct, hence the $\tau_i$'s in \eqref{eqn:admissible-second} are ordered. An arbitrary choice of values for $n-1$ of them leads to a unique value for the last one. Therefore there are precisely $N^{n-1}$ choices.

    We multiply the number of choices for $\tau_i$'s and the number of choices for $\varepsilon_i$'s and $k$ to conclude the statement.
\end{proof}

\begin{remark}
    An interesting feature in the proof worth realizing is that the value of $k$ is determined by $\{ \varepsilon_1, \dots, \varepsilon_n \}$ when $n=3$ or $4$. It is not the case for $n \geq 5$. (For example, when $n=5$ and $\varepsilon_i = 1$ for $i = 1, \dots, 5$, we could have $k=1$ or $2$.) Moreover, the first few nontrivial values of $a_n$ are $a_3 = 4$, $a_4 = 8$, $a_5 = 28$.
\end{remark}

\subsection{Higher dimensional strata}\label{subsec:higherdimstrata}

The higher dimensional strata can be computed along the same lines as in the previous paragraph. We point out that we allow ourselves to deviate slightly from notation, in that we sometimes allow $ x_i = b+1$ or $ x_i = -b-1$. 

\subsubsection{Line chart associated to an arbitrary stratum} \label{subsubsec:line-chart-arbitrary}

For an arbitrary stratum, we have $b \leq n$. In such a case equation \eqref{eqn:admissible-first} has to be modified. For convenience we set $x_0 = 0$ and $x_{n+1} = b+1$, then we have the condition $\lvert x_0 \rvert \leq \lvert x_1 \rvert \leq \dots \leq \lvert x_n \rvert \leq \lvert x_{n+1} \rvert$. For each $i = 1, \dots, n+1$, if $\lvert x_i \rvert = \lvert x_{i-1} \rvert$, then we set $r_i = 0$ in \eqref{eqn:admissible-first} (in other words, we remove $r_i$ from the equation). The set $S'$ of coefficients of the remaining $r_i$'s is a subset of $S$, and a necessary and sufficient condition for the solvability of the new equation becomes
\begin{equation}\label{eqn:lc-condition-prime}
    \text{either } \quad \min S' < 2k < \max S' \quad \text{ or } \quad \min S' = 2k = \max S'.
\end{equation}
We can still visualize the condition using a line chart. Indeed, for each surviving $r_i$, we still associate to it an integral point as before. Then we connect the set of vertices
\begin{equation}\label{eqn:remaining-points-chart}
    V' = \{ (n+1-i, \varepsilon_i + \dots + \varepsilon_n) \mid \lvert x_{i-1} \rvert < \lvert x_i \rvert, i = 1, \dots, n \}
\end{equation}
from left to right by line segments. The resulting picture is again called a \emph{line chart} associated to the stratum. If a line chart satisfies condition \eqref{eqn:lc-condition-prime}, we say that the line chart is \emph{admissible}; or equivalently, it is a line chart associated to an admissible stratum. Moreover, we also observe that a line chart is never complete when $b<n$.

\subsubsection{The (non)uniqueness of line charts}\label{subsubsec:non-uniqueness}

From the above construction, we can tell that the line chart encodes the information of all $x_i$'s for each stratum, and ignores the $\tau_i$'s. 

In many cases, the values of $x_i$'s are uniquely determined by the stratum, therefore the associated line chart is also uniquely determined by the stratum. However, there is one exception: if $\lvert x_i \rvert = b+1$, then no matter whether $\varepsilon_i = \pm 1$, it always indicates that the $i$-th point is located on one of the original components (with an odd label).

Therefore, for a point located on such an original component, we have $x_i = \varepsilon_i (b+1)$ where $\varepsilon$ could be either $+1$ or $-1$. Assume that there are $p$ such points, then the first vertex of the corresponding line chart is $(p, \varepsilon_{n-p+1} + \dots + \varepsilon_n)$ by \eqref{eqn:remaining-points-chart}, whose second coordinate could take $p+1$ different possible values $\{ -p, -p+2, \dots, p-2, p\}$. In other words, the initial vertex of the line chart is always located in the diagonal quadrant $\{ (x,y) \in \Z^2 \mid \lvert y \rvert \leq x \}$, and can be moved up or down by multiples of $2$ units.

For each of the subsequent vertices in $V'$ as given in \eqref{eqn:remaining-points-chart}, since its second coordinate is given by a sum that also contains $\varepsilon_{n-p+1} + \dots + \varepsilon_n$, it is also moved up or down by multiples of $2$ units along with the first vertex of $V'$. As a result, the entire line chart could be shifted vertically in either direction, by multiples of $2$ units, and without changing its shape, subject to the restriction on the initial vertex mentioned above. To simplify the later discussion, we introduce the following terminology.

\begin{definition}\label{def:linechartequiv}
	Line charts that represent the same strata are said to be \emph{equivalent}.
\end{definition}

Indeed, as explained above, two line charts are equivalent if and only if one is obtained from the other by a vertical shift $2p$ units for any integer $p$.

\begin{example}
	We exhibit two examples of non-unique line charts associated to a stratum. 
	
	In Figure \ref{fig:same-stratum-1}, two distinct line charts are associated to the same stratum $C(1,2,3) = C(1,2,-3)$. In this example, the level of expansion is given by $b=2$, and we have precisely one point located on one of the original components with an odd label, namely $x_3 = \pm 3$. Therefore $p=1$, and the initial vertex of the line chart is $(1, \varepsilon_3)$, which could take the value of either $(1,1)$ or $(1,-1)$. At the same time, the entire line chart is also shifted vertically along with the initial vertex, without changing its shape. This phenomenon can be easily seen in Figure \ref{fig:same-stratum-1}.
	
	In Figure \ref{fig:same-stratum-2}, three distinct line charts are associated to the same stratum $A(0,1,1)=A(0,1,-1)=A(0,-1,-1)$. Indeed, in this example, the level of expansion is $b=0$ (i.e. no inserted components), and we have $x_1=0$, $x_2=\pm 1$, $x_3=\pm 1$. Since $p=2$, the initial vertex of the line chart (which is also the only vertex in $V'$) is given by $(2, \varepsilon_2+\varepsilon_3)$, which could take the value of $(2,2)$ or $(2,0)$ or $(2,-2)$. Therefore the entire line chart can also be shifted vertically to three different positions, as displayed in Figure \ref{fig:same-stratum-2}.
	
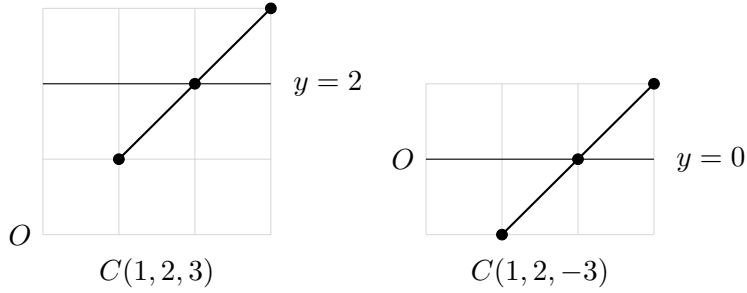
\begin{figure}
    \begin{tikzpicture}
        \draw[gray!30, step=1] (0,0) grid (3,3);
        \node at (-0.3,0) {$O$};
        \draw[thick] (1,1) -- (3,3);
        \draw (0,2) -- (3,2);
        \node at (3.75,2) {$y=2$};
        \node at (1.5,-0.5) {$C(1,2,3)$};
        \filldraw (1,1) circle [radius=2pt];
        \filldraw (2,2) circle [radius=2pt];
        \filldraw (3,3) circle [radius=2pt];
    \end{tikzpicture}
    \begin{tikzpicture}
        \draw[gray!30, step=1] (0,-1) grid (3,1);
        \node at (-0.3,0) {$O$};
        \draw[thick] (1,-1) -- (3,1);
        \draw (0,0) -- (3,0);
        \node at (3.75,0) {$y=0$};
        \node at (1.5,-1.5) {$C(1,2,-3)$};
        \filldraw (1,-1) circle [radius=2pt];
        \filldraw (2,0) circle [radius=2pt];
        \filldraw (3,1) circle [radius=2pt];
    \end{tikzpicture}
    \caption{An example of line charts for the same stratum}
    \label{fig:same-stratum-1}
\end{figure}

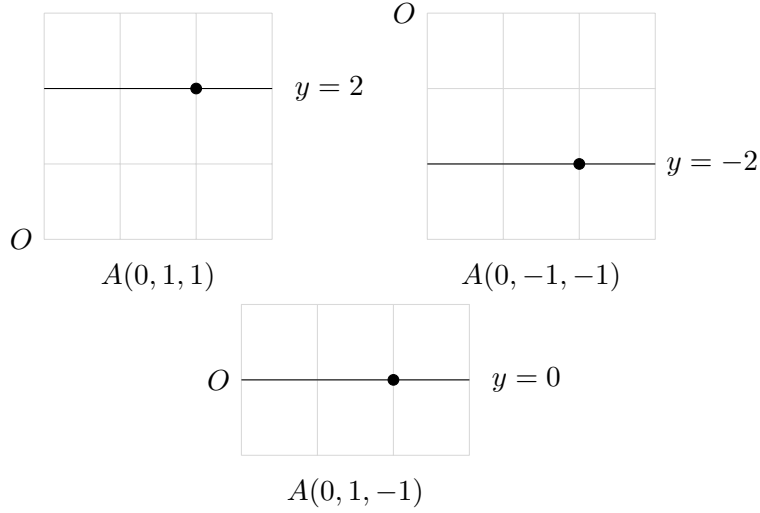
\begin{figure}
    \begin{tikzpicture}
        \draw[gray!30, step=1] (0,0) grid (3,3);
        \node at (-0.3,0) {$O$};
        \draw (0,2) -- (3,2);
        \node at (3.75,2) {$y=2$};
        \node at (1.5,-0.5) {$A(0,1,1)$};
        \filldraw (2,2) circle [radius=2pt];
    \end{tikzpicture}
    \begin{tikzpicture}
        \draw[gray!30, step=1] (0,0) grid (3,-3);
        \node at (-0.3,0) {$O$};
        \draw (0,-2) -- (3,-2);
        \node at (3.75,-2) {$y=-2$};
        \node at (1.5,-3.5) {$A(0,-1,-1)$};
        \filldraw (2,-2) circle [radius=2pt];
    \end{tikzpicture}
    \begin{tikzpicture}
        \draw[gray!30, step=1] (0,-1) grid (3,1);
        \node at (-0.3,0) {$O$};
        \draw (0,0) -- (3,0);
        \node at (3.75,0) {$y=0$};
        \node at (1.5,-1.5) {$A(0,1,-1)$};
        \filldraw (2,0) circle [radius=2pt];
    \end{tikzpicture}
    \caption{Another example of line charts for the same stratum}
    \label{fig:same-stratum-2}
\end{figure}

\end{example}

If we were to insist on avoiding the non-uniqueness of line charts associated to an arbitrary stratum, we could impose an extra requirement, for example, $\varepsilon_i=1$ if $\lvert x_i \rvert = b+1$. By the above discussion, this corresponds to requiring the initial vertex of any line chart to be located on the line $y=x$; or equivalently, the line chart is located at the highest possible position. However, we prefer not to impose such a restriction for two reasons. First of all, we will see that all these distinct line charts associated to the same stratum could show up as an output of the algorithm of numerical smoothing in Subsection \ref{subsec:numerical-smoothing}. Secondly, this flexibility allows us to glue various local pieces of the dual complex to a global picture, as we will see in Section \ref{sec:dualcplxIII}, where we will also mention the rules for how the triple $\tau=(\tau_1, \tau_2, \tau_3)$ changes when the line chart is moved upwards or downwards.

\subsubsection{Information contained in line charts}\label{subsubsec:sublinechart}

From the above construction, we have already seen that the information about the configuration of any stratum is contained in the corresponding line charts. Some information is straightforward; for example,
\begin{itemize}
	\item The level of expansion, denoted by $b$, is $1$ smaller than the number of vertices in the line chart;
	\item The difference in the $x$-coordinates of two neighboring vertices gives the count of points in the corresponding level of expansion;
	\item The difference in the $y$-coordinates gives the signed count of these points, where the sign is given by the sign of the $x_i$'s;
	\item The number of points on the original components of the degeneration with odd labels is the $x$-coordinate of the first vertex;
	\item The number of points on the original components of the degeneration with even labels is the difference between $n$ and the $x$-coordinate of the last vertex.
\end{itemize}

Also, as an immediate consequence of the above construction, it is important to observe that in conditions \eqref{eqn:lc-condition} and \eqref{eqn:lc-condition-prime}, both possibilities have their own geometric meaning.

\begin{lemma}\label{lemma:narrowwidedescr}
    Assume $\mathcal{L}$ is an admissible line chart, and $S$ is the set of $y$-coordinates of its vertices. Then $\min S < 2k < \max S$ iff the corresponding admissible stratum is wide, and $\min S = 2k = \max S$ iff the corresponding admissible stratum is narrow.
\end{lemma}

\begin{proof}
    By the construction of the set $S$, the condition $\min S = 2k = \max S$ is equivalent to
    \[
    \sum_{\lvert x_i \rvert = t} \varepsilon_i = \begin{cases}
        0 & \text{if } 1 \leq t \leq b, \\
        2k & \text{if } t = b+1,
    \end{cases}
    \]
    which is precisely the condition of a stratum being narrow.
\end{proof}

Given two line charts $\mathcal{L}_1$ and $\mathcal{L}_2$, assume their sets of vertices are $V_1$ and $V_2$ respectively. If $V_1 \subseteq V_2$, then we say $\mathcal{L}_1$ is a \emph{line subchart} of $\mathcal{L}_2$, denoted as $\mathcal{L}_1 \prec \mathcal{L}_2$. 
To determine all admissible line charts for any fixed value of $n \geq 3$, the following result is helpful.

\begin{lemma}\label{lem:correct-dim}
    Every line chart $\mathcal{L}'$ is a line subchart of a complete line chart $\mathcal{L}$. For $n \geq 3$, every admissible line chart $\mathcal{L}'$ is a line subchart of an admissible complete line chart $\mathcal{L}$.
\end{lemma}

\begin{proof}
    Since an arbitrary line chart is obtained by omitting some vertices in a complete line chart, the first statement follows immediately. 

    For the second statement, assume that $\mathcal{L}'$ is admissible, we show that we can choose $\mathcal{L}$ also to be admissible. There are two possibilities. If the condition $\min S' < 2k < \max S'$ is satisfied, since $S' \subseteq S$, we also have $\min S < 2k < \max S$, hence $\mathcal{L}$ is automatically admissible. 
    
    Now we assume that the condition $\min S' = 2k = \max S'$ is satisfied. Note that $S'$ is the set of $y$-coordinates of vertices in $V'$. Let $R'$ be the set of $x$-coordinates of vertices in $V'$, then the values in $R'$ divide the interval $[0,n]$ into several subintervals. We observe that the restriction of the line chart $\mathcal{L}$ to each of the subintervals has nontrivial intersection with either 
    the half space $y>2k$ or the half space $y<2k$. If the number of subintervals is at least $2$, then by flipping the line chart over some of the subintervals across the neutral line $y=2k$, we can make the line chart $\mathcal{L}$ to satisfy the condition $\min S < 2k < \max S$, therefore $\mathcal{L}$ is admissible. If the number of subintervals is $1$, then $R' \subseteq \{0, n\}$. Then there are several possibilities to choose an admissible line chart $\mathcal{L}$ that contains $\mathcal{L}'$ as a line subchart. More precisely, if $0 \in R'$, then $k=0$. Since $n \geq 3$, we can choose $\mathcal{L}$ such that its set of vertices $V$ contains $\{(0,0), (1,1), (2,0), (3,-1)\}$ to make sure it is admissible. Similarly, if $n \in R'$, we can also choose $\mathcal{L}$ such that $V$ contains $\{ (n,2k), (n-1,2k-1), (n-2,2k), (n-3,2k+1) \}$. 
\end{proof}

The above result shows that, to determine all admissible line charts for any fixed value of $n \geq 3$, it suffices to find all admissible line subcharts of the complete admissible line charts. 

\subsubsection{Example of computation for $n=3$}

We illustrate the above method by the following examples. For $n=3$, all complete admissible line charts have been described in Example \ref{eg:complete-admissible}. In the following, we describe all admissible line subcharts for two of them, namely $D(1,2,3)$ and $D(1,2,-3)$. The admissible line subcharts for the other two cases $D(-1,-2,-3)$ and $D(-1,-2,3)$ can be analyzed in a similar manner.

\begin{example}\label{eg:all-admissible-123}
    Let $\mathcal{L}$ be the line chart corresponding to $D(1,2,3)$. It has one vertex above the neutral line $y=2$, one vertex on this line, and two vertices below this line. To find all line subcharts $\mathcal{L}'$ of $\mathcal{L}$ that correspond to wide strata, the set of its vertices $V'$ must contain $(3,3)$, and at least one between $(0,0)$ and $(1,1)$. However it is unimportant whether $(2,2)$ is contained in $V'$. Hence there are a total of $6$ admissible line subcharts in this case. To find all line subcharts $\mathcal{L}'$ of $\mathcal{L}$ which correspond to narrow strata, the only possibility is that $V' = \{ (2,2) \}$. Therefore, there are in total $7$ admissible line subcharts of $\mathcal{L}$. These line subcharts, as well as their associated strata, are displayed in Figure \ref{fig:subcharts-123} (including the case $\mathcal{L}'=\mathcal{L}$). 
\begin{figure}
    \begin{tikzpicture}
        \draw[gray!30, step=1] (0,0) grid (3,3);
        \node at (-0.3,0) {$O$};
        \draw[thick] (0,0) -- (1,1);
        \draw[thick] (1,1) -- (2,2);
        \draw[thick] (2,2) -- (3,3);
        \draw (0,2) -- (3,2);
        \node at (3.75,2) {$y=2$};
        \node at (1.5,-0.5) {$D(1,2,3)$};
        \filldraw (0,0) circle [radius=2pt];
        \filldraw (1,1) circle [radius=2pt];
        \filldraw (2,2) circle [radius=2pt];
        \filldraw (3,3) circle [radius=2pt];
    \end{tikzpicture}
    \begin{tikzpicture}
        \draw[gray!30, step=1] (0,0) grid (3,3);
        \node at (-0.3,0) {$O$};
        \draw[thick] (0,0) -- (1,1);
        \draw[thick] (1,1) -- (3,3);
        \draw (0,2) -- (3,2);
        \node at (3.75,2) {$y=2$};
        \node at (1.5,-0.5) {$C(1,1,2)$};
        \filldraw (0,0) circle [radius=2pt];
        \filldraw (1,1) circle [radius=2pt];
        \filldraw (3,3) circle [radius=2pt];
    \end{tikzpicture}
    \begin{tikzpicture}
        \draw[gray!30, step=1] (0,0) grid (3,3);
        \node at (-0.3,0) {$O$};
        \draw[thick] (0,0) -- (2,2);
        \draw[thick] (2,2) -- (3,3);
        \draw (0,2) -- (3,2);
        \node at (3.75,2) {$y=2$};
        \node at (1.5,-0.5) {$C(1,2,2)$};
        \filldraw (0,0) circle [radius=2pt];
        \filldraw (2,2) circle [radius=2pt];
        \filldraw (3,3) circle [radius=2pt];
    \end{tikzpicture}
    \begin{tikzpicture}
        \draw[gray!30, step=1] (0,0) grid (3,3);
        \node at (-0.3,0) {$O$};
        \draw[thick] (0,0) -- (3,3);
        \draw (0,2) -- (3,2);
        \node at (3.75,2) {$y=2$};
        \node at (1.5,-0.5) {$B(1,1,1)$};
        \filldraw (0,0) circle [radius=2pt];
        \filldraw (3,3) circle [radius=2pt];
    \end{tikzpicture}
    \begin{tikzpicture}
        \draw[gray!30, step=1] (0,0) grid (3,3);
        \node at (-0.3,0) {$O$};
        \draw[thick] (1,1) -- (2,2);
        \draw[thick] (2,2) -- (3,3);
        \draw (0,2) -- (3,2);
        \node at (3.75,2) {$y=2$};
        \node at (1.5,-0.5) {$C(1,2,3)$};
        \filldraw (1,1) circle [radius=2pt];
        \filldraw (2,2) circle [radius=2pt];
        \filldraw (3,3) circle [radius=2pt];
    \end{tikzpicture}
    \begin{tikzpicture}
        \draw[gray!30, step=1] (0,0) grid (3,3);
        \node at (-0.3,0) {$O$};
        \draw[thick] (1,1) -- (3,3);
        \draw (0,2) -- (3,2);
        \node at (3.75,2) {$y=2$};
        \node at (1.5,-0.5) {$B(1,1,2)$};
        \filldraw (1,1) circle [radius=2pt];
        \filldraw (3,3) circle [radius=2pt];
    \end{tikzpicture}
    \begin{tikzpicture}
        \draw[gray!30, step=1] (0,0) grid (3,3);
        \node at (-0.3,0) {$O$};
        \filldraw (2,2) circle [radius=2pt];
        \draw (0,2) -- (3,2);
        \node at (3.75,2) {$y=2$};
        \node at (1.5,-0.5) {$A(0,1,1)$};
    \end{tikzpicture}
    \caption{Admissible subcharts for $D(1,2,3)$ in Example \ref{eg:all-admissible-123}}
    \label{fig:subcharts-123}
\end{figure}
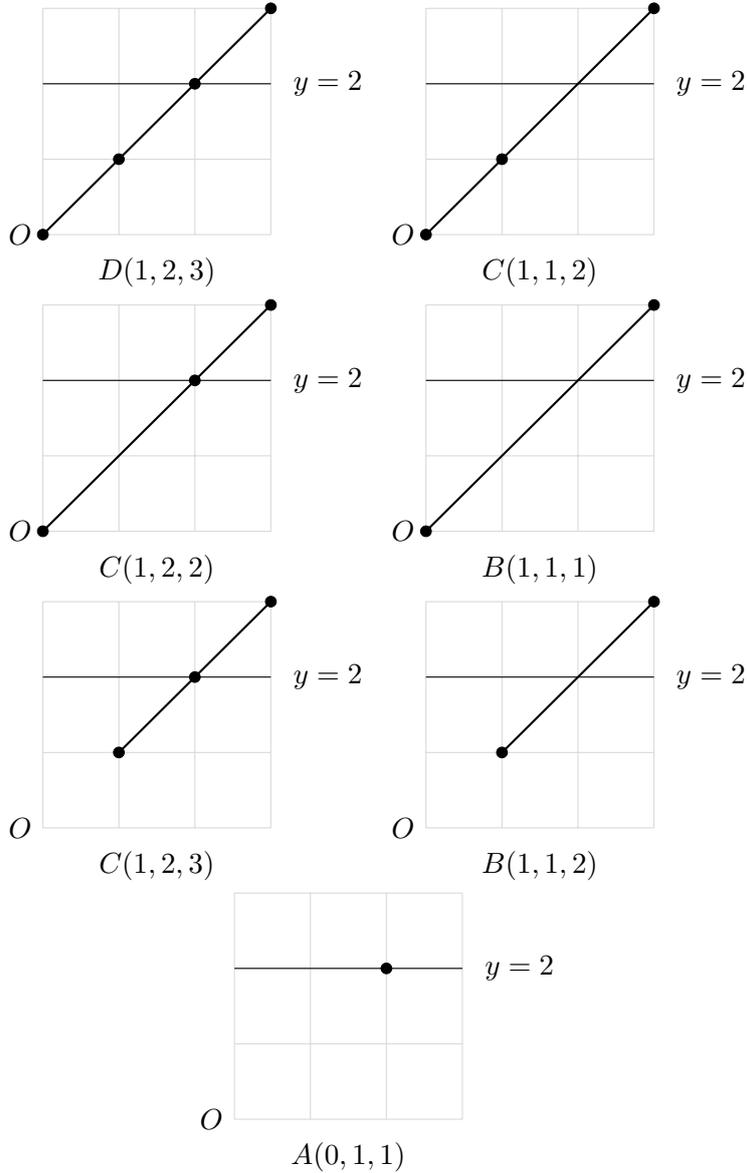
\end{example}

\begin{example}\label{eg:all-admissible-12minus3}
As a second example, let $\mathcal{L}$ be the line chart corresponding to $D(1,2,-3)$. In this case there is one vertex above the neutral line $y=0$, two vertices on this line, and one vertex below this line. To find all line subcharts $\mathcal{L}'$ that correspond to wide strata, both vertices $(1,1)$ and $(3,-1)$ have to survive, but $(0,0)$ and $(2,0)$ could be removed, hence there are in total $4$ admissible line subcharts in this case. To find all line subcharts that correspond to narrow strata, we need to remove both $(1,1)$ and $(3,-1)$, and keep at least one of the vertices on the line $y=0$,  
hence there are in total $3$ admissible line subcharts in this case. All these admissible line subcharts, as well as their associated strata, are displayed in Figure \ref{fig:subcharts-12minus3} (including the case $\mathcal{L}'=\mathcal{L}$).

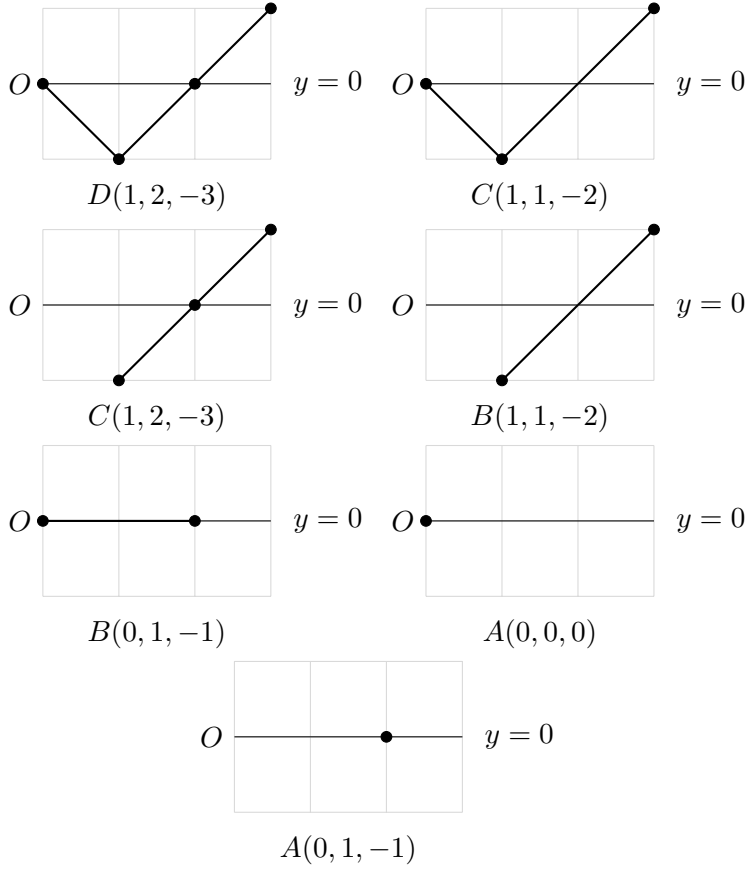
\begin{figure}
    \begin{tikzpicture}
        \draw[gray!30, step=1] (0,-1) grid (3,1);
        \node at (-0.3,0) {$O$};
        \draw[thick] (0,0) -- (1,-1);
        \draw[thick] (1,-1) -- (2,0);
        \draw[thick] (2,0) -- (3,1);
        \draw (0,0) -- (3,0);
        \node at (3.75,0) {$y=0$};
        \node at (1.5,-1.5) {$D(1,2,-3)$};
        \filldraw (0,0) circle [radius=2pt];
        \filldraw (1,-1) circle [radius=2pt];
        \filldraw (2,0) circle [radius=2pt];
        \filldraw (3,1) circle [radius=2pt];
    \end{tikzpicture}
    \begin{tikzpicture}
        \draw[gray!30, step=1] (0,-1) grid (3,1);
        \node at (-0.3,0) {$O$};
        \draw[thick] (0,0) -- (1,-1);
        \draw[thick] (1,-1) -- (2,0);
        \draw[thick] (2,0) -- (3,1);
        \draw (0,0) -- (3,0);
        \node at (3.75,0) {$y=0$};
        \node at (1.5,-1.5) {$C(1,1,-2)$};
        \filldraw (0,0) circle [radius=2pt];
        \filldraw (1,-1) circle [radius=2pt];
        \filldraw (3,1) circle [radius=2pt];
    \end{tikzpicture}
    \begin{tikzpicture}
        \draw[gray!30, step=1] (0,-1) grid (3,1);
        \node at (-0.3,0) {$O$};
        \draw[thick] (1,-1) -- (2,0);
        \draw[thick] (2,0) -- (3,1);
        \draw (0,0) -- (3,0);
        \node at (3.75,0) {$y=0$};
        \node at (1.5,-1.5) {$C(1,2,-3)$};
        \filldraw (1,-1) circle [radius=2pt];
        \filldraw (2,0) circle [radius=2pt];
        \filldraw (3,1) circle [radius=2pt];
    \end{tikzpicture}
    \begin{tikzpicture}
        \draw[gray!30, step=1] (0,-1) grid (3,1);
        \node at (-0.3,0) {$O$};
        \draw[thick] (1,-1) -- (2,0);
        \draw[thick] (2,0) -- (3,1);
        \draw (0,0) -- (3,0);
        \node at (3.75,0) {$y=0$};
        \node at (1.5,-1.5) {$B(1,1,-2)$};
        \filldraw (1,-1) circle [radius=2pt];
        \filldraw (3,1) circle [radius=2pt];
    \end{tikzpicture}
    \begin{tikzpicture}
        \draw[gray!30, step=1] (0,-1) grid (3,1);
        \node at (-0.3,0) {$O$};
        \draw[thick] (0,0) -- (2,0);
        \draw (0,0) -- (3,0);
        \node at (3.75,0) {$y=0$};
        \node at (1.5,-1.5) {$B(0,1,-1)$};
        \filldraw (0,0) circle [radius=2pt];
        \filldraw (2,0) circle [radius=2pt];
    \end{tikzpicture}
    \begin{tikzpicture}
        \draw[gray!30, step=1] (0,-1) grid (3,1);
        \node at (-0.3,0) {$O$};
        \draw (0,0) -- (3,0);
        \node at (3.75,0) {$y=0$};
        \node at (1.5,-1.5) {$A(0,0,0)$};
        \filldraw (0,0) circle [radius=2pt];
    \end{tikzpicture}
    \begin{tikzpicture}
        \draw[gray!30, step=1] (0,-1) grid (3,1);
        \node at (-0.3,0) {$O$};
        \draw (0,0) -- (3,0);
        \node at (3.75,0) {$y=0$};
        \node at (1.5,-1.5) {$A(0,1,-1)$};
        \filldraw (2,0) circle [radius=2pt];
    \end{tikzpicture}
    \caption{Admissible subcharts for $D(1,2,-3)$ in Example \ref{eg:all-admissible-12minus3}}
    \label{fig:subcharts-12minus3}
\end{figure}
\end{example}

We summarize the results from both examples in geometric language.

\begin{prop}\label{prop:stratadeterm2}
Besides the strata determined in Proposition \ref{prop:stratadeterm}, the remaining ones included in $\mathcal{K}^2$ are:
\begin{enumerate}
    \item $A(0,1,1)$, $B(1,1,1)$, $B(1,1,2)$, $C(1,1,2)$, $C(1,2,2)$ and $C(1,2,3)$ when $\mathfrak{T} \in \{N-1, 2N-1\}$,
    \item The negatives of the strata in $(1)$ when $\mathfrak{T} \in \{1, N+1, 2N+1\}$,
    \item $A(0,0,0)$, $A(0,1,-1)$, $B(1,1,-2)$, $B(0,1,-1)$, $C(1,2,-3)$ and  $C(1,1,-2)$ when $\mathfrak{T} \in \{0, N, 2N\}$,
    \item The negatives of the strata in $(3)$ when $\mathfrak{T} \in \{0, N, 2N\}$.
\end{enumerate}
\end{prop}

\begin{proof}
	The cases (1) and (3) are geometric interpretations of Examples \ref{eg:all-admissible-123} and \ref{eg:all-admissible-12minus3} respectively; the cases (2) and (4) can be analyzed similarly. The possible values for $\mathfrak{T}$ in all cases still follow from \eqref{eqn:admissible-second}.
\end{proof}

\begin{remark}
It is crucial to point out a subtle difference between Propositions \ref{prop:stratadeterm} and  \ref{prop:stratadeterm2}. Since we allow $x_i = b+1$ and $-b-1$, the strata appearing in Proposition \ref{prop:stratadeterm2} are not always disjoint cases. For example, it is straightforward to verify that the stratum $ A(\tau_3-1, \tau_1,\tau_2; 0,1,1)$ is the same as $A(\tau_3-1, \tau_1,\tau_2+1; 0,1,-1)$. In fact, these types of identifications will be used extensively later on, when gluing $2$-simplices in Section \ref{sec:dualcplxIII}.
\end{remark}

\subsection{Smoothing}\label{subsec:numerical-smoothing}

For any $n$, the map $ \mathcal{H}^n \to C$ is a semi-stable degeneration, since it is a composition of the smooth morphism $ \mathcal{H}^n \to C[n]$ and the semi-stable map $C[n] \to C$. Using the description of $ \mathcal{H}^n $ from \cite{GHH} and \cite{GHHZ21}, we can systematically keep track of the intersections of the irreducible components of the special fiber $ \mathcal{H}^n_0 $ over $0 \in C$, and \emph{e.g.} compute the dual complex. 

In the notation introduced above, the irreducible components are the closures of the maximal dimensional strata $\mathcal{X}_0$, 
which intersect along the strata $\mathcal{X}_{b}$, for $b \leq n $. 
Here we note that the family $ \mathcal{H}^n \to C$ is flat and thus all irreducible components have the same dimension. One can also formulate this in terms of \emph{smoothings}, where we say that $\mathcal{X}_{b'}(\tau, x'_1, \ldots,x'_n)$ is a smoothing of $\mathcal{X}_{b}(\tau, x_1, \ldots,x_n)$ if the latter is contained in the closure of the former.

\subsubsection{The smoothing mechanism}\label{subsubsec:numrules}
Let $C[b] \to C[n]$ be the standard embedding with coordinates $ \{t_{i_1}, \ldots, t_{i_l} \} \subset \{t_1, \ldots,t_{n+1}\} $, where $l=b+1$. For $ 1 \leq j \leq l$, the lemma below describes the smoothing of a stratum $ \mathcal{X}_b(x_1, \ldots, x_n)$ \emph{along} the coordinate $t_{i_j}$. By this we mean the unique stratum over $ (0 \in C[b-1]) \times \mathbb G[n-(b-1)] $ containing $ \mathcal{X}_b(x_1, \ldots, x_{n})$ in its closure, where $C[b-1] \to C[n]$ is the standard embedding with coordinates $ \{t_{i_1}, \ldots, \hat{t}_{i_j}, \ldots, t_{i_l} \} $.

\begin{lemma}\label{lemma:numsmoothing}
	We set \[
x_i' = \left\{
        \begin{array}{ll}
          x_i, & \mathrm{if} \quad \vert x_i \vert < j  \\[1em]
          
          x_i - 1,  & \mathrm{if} \quad \vert x_i \vert \geq j \quad \mathrm{and}  \quad x_j > 0 \\
          [1em]

          x_i + 1, & \mathrm{if} \quad \vert x_i \vert \geq j \quad \mathrm{and}  \quad x_j < 0.
				\end{array}
    \right.
\]
Then we have
\begin{itemize}
	\item For the strata in the relative Hilbert scheme, the smoothing of $ \mathcal{X}_b(x_1, \ldots, x_n)$ along $t_{i_j}$ equals the stratum $ \mathcal{X}_{b-1}(x_1', \ldots,x_n') $;
	\item For the strata in the relative Kummer variety, the smoothing of $ X_b(x_1, \ldots, x_n)$ along $t_{i_j}$ equals the stratum $ X_{b-1}(x_1', \ldots,x_n') $, provided that the latter is admissible.
\end{itemize}
\end{lemma}

\begin{proof}
	For the case of the relative Hilbert scheme, see \cite[Theorem 4.8]{GHHZ21}.
	
	Moreover, we will prove in Proposition \ref{prop:C-goal} a relation between the Hilbert strata and the corresponding Kummer strata, namely:
\begin{equation}\label{eqn:intersection-of-strata}
	\mathrm{cl}(X_{b'}(x'_1, \ldots, x'_n)) \cap \mathcal{X}_{b}(x_1, \ldots, x_n) = X_{b}(x_1, \ldots, x_n)
\end{equation}
for any choice of $b'<b$, when the stratum $X_{b}(x_1, \ldots, x_n)$ is admissible (otherwise the intersection is empty). Then the case of the relative Kummer variety follows immediately from the case of the relative Hilbert scheme and \eqref{eqn:intersection-of-strata}.
\end{proof}

\subsubsection{} We explain how the smoothing of $D(1,2,3)$ works. The result for $D(-1,-2,-3)$ is entirely similar, except that all integers should be replaced by their negatives.

The stratum $D(1,2,3)$ has dimension $2 \delta + 1$ by Lemma \ref{lemma:stratdimension}. If we smooth along $t_1, t_2, t_3$ and $t_4$ we get $C(0,1,2)$, $C(1,1,2)$, $C(1,2,2)$ and $C(1,2,3)$, respectively. All these strata have dimension $2 \delta + 2$. Note that $C(0,1,2)$ is not in $\mathcal{K}^2$, by Proposition \ref{prop:stratadeterm2}.

Smoothing $C(1,1,2)$ along $t_{i_1} = t_1$, $t_{i_2} = t_3$ and $t_{i_3} = t_4$ yields $B(0,0,1)$, $B(1,1,1)$ and $B(1,1,2)$, respectively. Note that $B(0,0,1)$ is not in $\mathcal{K}^2$. Both strata $B(1,1,1)$ and $B(1,1,2)$ have dimension $2 \delta + 3$. Smoothing $C(1,2,2)$ along $t_1$, $t_2$ and $t_4$ yields $B(0,1,1)$, $B(1,1,1)$ and $B(1,2,2)$, respectively. Note that $B(0,1,1)$ and $B(1,2,2)$ are not in $\mathcal{K}^2$. Similarly, the only smoothing of $C(1,2,3)$ resulting in a stratum of $\mathcal{K}^2$ is along $t_2$, which gives $B(1,1,2)$.

There is only one further stratum appearing. Namely, smoothing $C(1,2,2)$ along $t_1$ and $t_4$, or smoothing $C(1,2,3)$ along $t_1$ and $t_4$, in both cases gives $A(0,1,1)$. This stratum has dimension $2 \delta + 3$.

\subsubsection{Interpretation in terms of line charts}\label{subsubsec:smoothing-line-charts}

In both cases, namely of Hilbert schemes and generalized Kummer varieties, the geometry of smoothings can be conveniently described in the language of line charts.

Assume $\mathcal{L}$ is a line chart associated to a certain stratum, whose set of vertices is $V$. Once a certain coordinate $t_j$ in the base $C[b]$ deforms from $0$ to a nonzero value, we have to set $r_j=0$ in equation \eqref{eqn:admissible-first}, which corresponds to removing the point $(n+1-j, \varepsilon_j + \dots + \varepsilon_n)$ from $V$. We therefore conclude that smoothing corresponds to taking line subcharts. 

In the case of Hilbert schemes, every possible smoothing remains a stratum. However, in the case of generalized Kummer varieties, only the admissible smoothings survive as strata.

Opposite to smoothing, specialization corresponds to adding vertices to $ \mathcal{L}$ with the missing values of $x$-coordinates.

\subsubsection{}\label{eg:admissible-123-again}

    The calculation in Example \ref{eg:all-admissible-123} clearly shows the limiting relation among the $7$ strata appearing in the smoothing of $D(1,2,3)$. More precisely, the line charts corresponding to $C(1,1,2)$, $C(1,2,2)$ and $C(1,2,3)$ are obtained by removing one vertex from that of $D(1,2,3)$, therefore their corresponding strata are smoothings of $D(1,2,3)$. Similarly, by removing vertices, the strata $B(1,1,1)$ and $B(1,1,2)$ are smoothings of $C(1,1,2)$; the strata $B(1,1,1)$ and $A(0,1,1)$ are smoothings of $C(1,2,2)$; the strata $B(1,1,2)$ and $A(0,1,1)$ are smoothings of $C(1,2,3)$.
    
    If we assign a $0$-cell to $B(1,1,1)$, $B(1,1,2)$ and $A(0,1,1)$, a $1$-cell to $C(1,1,2)$, $C(1,2,2)$ and $C(1,2,3)$ and a $2$-cell to $D(1,2,3)$, we see that the smoothing relations also are encoded combinatorially by a $2$-simplex (with specialization of a stratum corresponding to being part of the boundary of a cell), denoted $\mathcal{D}^+$ and visualized in \ref{subsubsec:dualcplx123} below. A similar analysis shows how the smoothings of $D(1,2,-3)$ correspond to the $2$-simplex denoted $\mathscr{D}^+$ in \ref{subsubsec:dualcplx12minus3}.
    
    In the next subsection, we shall initiate a systematic study of dual complexes in our context. 

\subsubsection{}\label{subsubsec:dualcplx123}
The dual complex $ \mathcal{D}^+ = \mathcal{D}(1,2,3)$ encoding the numerical smoothings of $D(1,2,3)$ can be described as follows: 

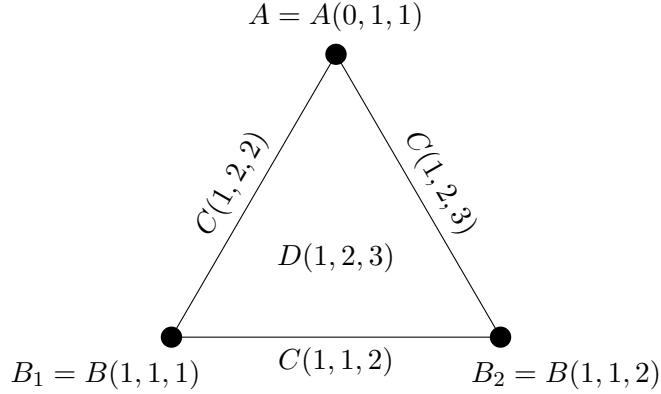
\begin{center}
		\begin{tikzpicture}
			\draw (90:2.5cm) -- (210:2.5cm) node[midway,sloped,above] {$C(1,2,2)$} -- (330:2.5cm) node[midway,sloped,below] {$C(1,1,2)$} -- (90:2.5cm) node[midway,sloped,above] {$C(1,2,3)$};
			\fill (90:2.5cm) circle (4pt);
			\fill (210:2.5cm) circle (4pt);
			\fill (330:2.5cm) circle (4pt);
			\draw (90:3cm) node {$A=A(0,1,1)$};
			\draw (210:3.5cm) node {$B_1=B(1,1,1)$};
			\draw (330:3.5cm) node {$B_2=B(1,1,2)$};
			\draw (0,-0.2) node{$D(1,2,3)$};
		\end{tikzpicture}
	\end{center}

Similarly, we can define a dual complex $ \mathcal{D}^- = \mathcal{D}(-1,-2,-3)$ encoding the smoothings of $D(-1,-2,-3)$; it is obtained from $ \mathcal{D}^+$ by replacing every $X_b(x_1,x_2,x_3)$ by its negative $X_b(-x_1,-x_2,-x_3)$.

\subsubsection{}\label{subsubsec:dualcplx12minus3}
The dual complex $ \mathscr{D}^+ = \mathscr{D}(1,2,-3)$ encoding the
smoothings of $D(1, 2, -3)$ is given as follows:

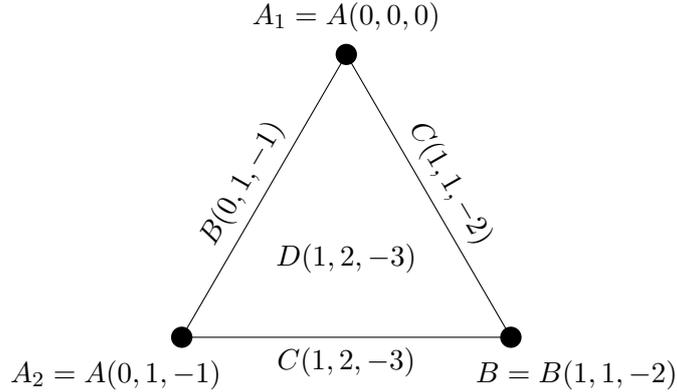
\begin{center}
		\begin{tikzpicture}
			\draw (90:2.5cm) -- (210:2.5cm) node[midway,sloped,above] {$B(0,1,-1)$} -- (330:2.5cm) node[midway,sloped,below] {$C(1,2,-3)$} -- (90:2.5cm) node[midway,sloped,above] {$C(1,1,-2)$} ;
			\fill (90:2.5cm) circle (4pt);
			\fill (210:2.5cm) circle (4pt);
			\fill (330:2.5cm) circle (4pt);
			\draw (90:3cm) node {$A_1=A(0,0,0)$};
			\draw (210:3.5cm) node {$A_2=A(0,1,-1)$};
			\draw (330:3.5cm) node {$B=B(1,1,-2)$};
			\draw (0,-0.2) node {$D(1,2,-3)$};
		\end{tikzpicture}
	\end{center}

Similarly, we can define a dual complex $ \mathscr{D}^- = \mathscr{D}(-1,-2,3)$ encoding the smoothings of $D(-1,-2,3)$; it is obtained from $ \mathscr{D}^+$ by replacing every $X_b(x_1,x_2,x_3)$ by its negative $X_b(-x_1,-x_2,-x_3)$.

\begin{remark}
    We point out the similarity and difference between the two cases above. Smoothing $D(1,2,3)$ and $D(1,2,-3)$ both gives a $2$-simplex as the dual complex. Later we will realize these $2$-simplices as subsets in $3$-simplices that are building blocks of the "full" dual complex of the relative Hilbert scheme, and discover that the way they embed in the $3$-simplices as subsets will be different. We will see that the reason for the difference is the different distribution of vertices above, below, and on the neutral line. 
   
\end{remark}

\subsubsection{}
For arbitrary $n$ the dual complex associated to a deepest stratum is not always an $(n-1)$-simplex, as demonstrated by the following example.

\begin{example}\label{eg:4d-example}
    When $n=4$, we consider a deepest admissible stratum $E(1,2,3,4)$ whose corresponding line chart $ \mathcal{L} $ has vertices 
    \[
    S = \{ (0,0), (1,1), (2,2), (3,3), (4,4) \}
    \]
    with $2k = 2$. By listing all its admissible line subcharts, we discover that the dual complex that corresponds to the above deepest stratum is a pyramid over a square (instead of a tetrahedron). Figure \ref{fig:4d-example} explicitly indicates the stratum that each vertex or edge represents. Moreover, the bottom face represents $D(1,2,2,3)$; the left face represents $D(1,1,2,3)$; the back face represents $D(1,2,3,4)$; the right face represents $D(0,1,2,3)$; the front face represents $D(1,2,3,3)$; the interior of the pyramid represents $E(1,2,3,4)$.

Geometrically, the occurrence of such a pyramid indicates that five $3$-dimensional irreducible components of the central fiber intersect at a common point, hence the degeneration is not semistable. 

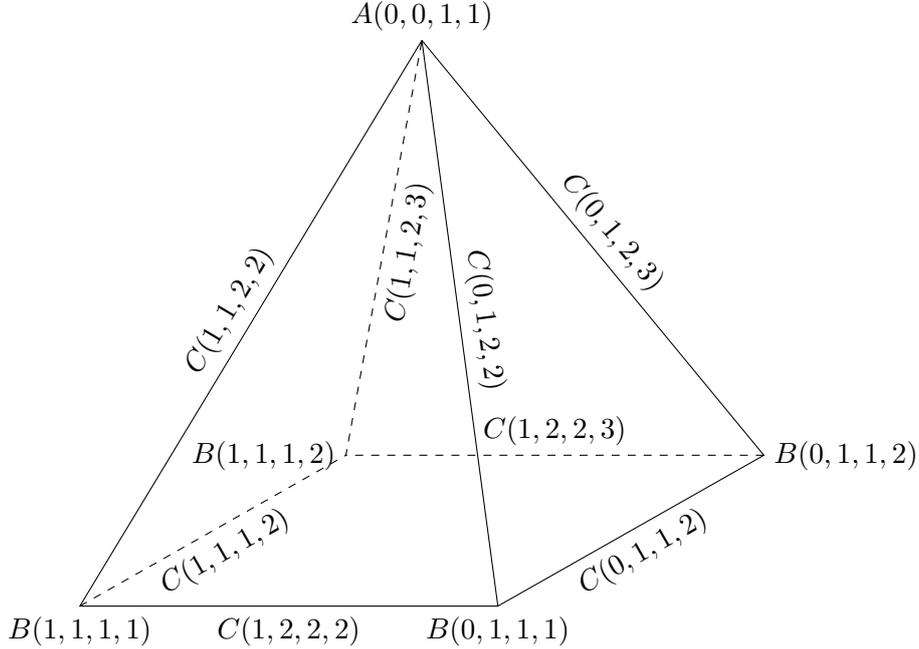
\begin{figure}
    \begin{tikzpicture}
        \coordinate (T) at (0,6.5);
        \coordinate (A) at (-4.5,-1);
        \coordinate (B) at (-1,1);
        \coordinate (D) at (1,-1);
        \coordinate (C) at (4.5,1);
        \draw (T) node[anchor=south] {$A(0,0,1,1)$};
        \draw (A) node[anchor=north] {$B(1,1,1,1)$};
        \draw (B) node[anchor=east] {$B(1,1,1,2)$};
        \draw (C) node[anchor=west] {$B(0,1,1,2)$};
        \draw (D) node[anchor=north] {$B(0,1,1,1)$};
        \draw (T) -- node[sloped,above]{$C(1,1,2,2)$} (A) -- node[sloped,below]{$C(1,2,2,2)$} (D) -- node[sloped,below]{$C(0,1,1,2)$} (C) -- node[sloped,above]{$C(0,1,2,3)$} (T);
        \draw (D) -- node[sloped,above]{$C(0,1,2,2)$} (T);
        \draw[dashed] (A) -- node[sloped,below]{$C(1,1,1,2)$} (B) -- node[sloped,below]{$C(1,1,2,3)$} (T);
        \draw[dashed] (C) -- node[sloped,above]{$C(1,2,2,3)$} (B);
        \end{tikzpicture}
    \caption{Smoothings of $E(1,2,3,4)$ in Example \ref{eg:4d-example}}
    \label{fig:4d-example}
\end{figure}

\end{example}

\subsection{Dual complexes}\label{subsec:dualcpldef}
We shall now introduce dual complexes more systematically, without relying on low dimension like above. 
In particular, we let $ n \geq 3 $ be arbitrary. In the sequel, some standard facts and terminology related to 
convex polytopes will be used, we refer to \cite{Zie} for a comprehensive reference.

For the Hilbert scheme degeneration $I^n = I^n_{\mathcal{Y}/C}$, we have already introduced and studied the dual complex $\Delta(I^n)$ in \cite{GHHZ21}. Our task is thus to develop a notion of dual complex also for the generalized Kummer degeneration $K^{n-1} = K^{n-1}_{\mathcal{Y}/C}$. Since $K^{n-1} \to C$ will not in general be \emph{dlt}, we do not get the existence of a well behaved dual complex from \cite{dFKX}, as in the case of $I^n$.  

Our approach will be to first define a dual complex encoding the smoothings of a deepest stratum, i.e. corresponding to a complete admissible line chart, and then to glue this objects to obtain the "global" dual complex, which shall be denoted $\Delta(K^{n-1}_{\mathcal{Y}/C})$.

\subsubsection{} A complete admissible line chart $\mathcal{L}$ defines two graded posets. The first is $P_H$ (here we do not use admissibility), formed by all line subcharts $\mathcal{L}' \prec \mathcal{L}$, with rank function $r(\mathcal{L}') = \delta n - \mathrm{dim}(\mathcal{X}(\mathcal{L}'))$. The second is $P_K$, consisting of all admissible line subcharts $\mathcal{L}'$, with rank function $r(\mathcal{L}') = \delta (n-1) - \mathrm{dim}(X(\mathcal{L}'))$. Here $\mathcal{X}(\mathcal{L}')$, resp.~$X(\mathcal{L}')$, denotes the unique stratum in the Hilbert scheme, resp.~the Kummer scheme, corresponding to $\mathcal{L}'$. (That the latter forms a graded poset is a straightforward verification using properties of line charts together with Lemma \ref{lemma:numsmoothing}.)

Since $P_H$ is just the power set of the set of vertices $V$, it can also be seen as the face lattice of the standard $n$-simplex. For a given $\mathcal{L}$, this defines the (H-)dual complex of $\mathcal{L}$, denoted $\Delta_H(\mathcal{L})$.

\subsubsection{}

Also $P_K$ can in fact be identified with the face lattice of a polytope, which we will now explain. First of all, we recall that there is a one-to-one correspondence between the vertex set $V$ of $\mathcal{L}$ and the vertex set of $\Delta_H$, hence we use also $V$ for the latter. We moreover identify $\Delta_H$ with the convex hull of the standard basis vectors of $\mathbb{R}^{n+1}$ (see e.g. \cite[p.~7]{Zie}).

We decompose $V$ into a disjoint union
$$ V_+ \cup V_0 \cup V_- $$
according to the subsets of vertices that lie above, on and below the neutral line, respectively. These sets have cardinality $n_+$, $ n_0$ and $n_-$, respectively. Note that since $\mathcal{L}$ is complete and admissible, both $n_+$ and $n_-$ are strictly positive. Like in the setting of Lemma \ref{lemma:descriteratedconeprod}, we consider the hyperplane 
$$ \Lambda = \{z_1 + \ldots + z_{n_+} - z_{n_+ + 1} - \ldots - z_{n_+ + n_-} = 0 \}$$
in $\mathbb{R}^{n+1}$. By construction, it contains all vertices of $V_0$, the vertices of $V_+$ are all on one side of $\Lambda$, and the vertices of $V_-$ are all on the other side of $\Lambda$.

In the proposition below, we use the notation $\delta_m$ for the standard $m$-simplex, and $C^m(Q)$ for the $m$-fold cone over a convex polytope $Q$. 

\begin{prop}\label{prop:from-H-to-K}
Let $\Lambda$ be as above. Then the intersection $\Lambda \cap \Delta_H $ is a convex polytope affinely equivalent to $C^{n_0}(\delta_{n_+ - 1} \times \delta_{n_- - 1})$. Moreover, its face lattice equals $P_K$.
\end{prop}
\begin{proof}
We have seen in Subsection \ref{subsubsec:line-chart-arbitrary} that only admissible line subcharts of $\mathcal{L}$ correspond to strata that are smoothings of $X_n(\mathcal{L})$. In other words, only the following types of faces of $\Delta_H$ correspond to strata in the relative Kummer variety:
    \begin{itemize}
    	\item the faces of $\Delta_H$ that are completely contained in $\Lambda$, which correspond to narrow strata,
    	\item the faces of $\Delta_H$ that have nonempty intersection with both sides of $\Lambda$, which correspond to wide strata.
    \end{itemize}

Then $\Lambda \cap \Delta_H$ is a complex, whose faces are given by the intersections of $\Lambda$ with both above types of faces, which correspond to strata in the relative Kummer variety. A purely combinatorial argument shows that it is equivalent to $C^{n_0}(\delta_{n_+ - 1} \times \delta_{n_- - 1})$, see Lemma \ref{lemma:descriteratedconeprod}. Moreover, the boundary relations of the faces of $\Lambda \cap \Delta_H$ are precisely the same as the boundary relations of the corresponding faces of $\Delta_H$, hence also agree with the boundary relations of the strata in the relative Kummer variety. Therefore we conclude that $\Lambda \cap \Delta_H$ precisely encodes all smoothings of $X_n(\mathcal{L})$, as desired.
\end{proof}

In light of the description in Proposition \ref{prop:from-H-to-K}, we therefore define 
$$\Delta_K(\mathcal{L}) = \Lambda \cap \Delta_K(\mathcal{L}). $$

\subsubsection{}
We next record an elementary feature of smoothings, or, equivalently, line charts, that will be useful in defining the global dual complex.

\begin{prop}\label{prop:dual-const-dim}
    For $n \geq 3$, the closure of every admissible stratum in the relative Kummer variety contains a deepest stratum (of dimension $(\delta-1)(n-1)$). 
\end{prop}
\begin{proof}
    It follows immediately from the smoothing mechanism in Subsection \ref{subsubsec:numrules} and the second statement of Lemma \ref{lem:correct-dim}. Indeed, to every admissible line chart $\mathcal{L}'$, we can add more vertices to obtain a complete admissible line chart $\mathcal{L}$, which corresponds to an $(n-1)$-cell. By Proposition \ref{prop:C-goal} we conclude that the stratum associated to $\mathcal{L}$ is included in the closure of the stratum associated to $\mathcal{L}'$.
\end{proof}

For use later in the paper, we also record the following phenomenon.

\begin{cor}\label{cor-boundarynarrowwide}
    The boundary of a narrow stratum can contain wide and narrow strata, while the boundary of a wide stratum can only contain wide strata.
\end{cor}

\begin{proof}
    Taking the boundary corresponds to adding more vertices to the line chart. If a line chart satisfies the condition $\min S < 2k < \max S$, this condition still holds after adding more vertices. Therefore a wide stratum can only specialize to wide strata.

    All the other possibilities can be achieved. For example, in the dual complex $\mathscr{D}^+$ in Subsection \ref{subsubsec:dualcplx12minus3}, we have seen that the narrow stratum $A(0,0,0)$ specializes to both a narrow stratum $B(0,1,-1)$ and a wide stratum $C(1,1,-2)$, and that the wide stratum $C(1,1,-2)$ specializes to a wide stratum $D(1,2,-3)$.
\end{proof}

\subsubsection{Definition of the global dual complex}\label{subsec:dualcplxdef}
If $\mathcal{L}'$ is an admissible line chart, Proposition \ref{prop:dual-const-dim} asserts that it can be embedded in a complete admissible line chart $\mathcal{L}$. By Proposition \ref{prop:from-H-to-K}, this embedding corresponds to a face $F(\mathcal{L}')$ in $\Delta_K(\mathcal{L})$.

Let $ m_0 $, $m_+$ and $m_-$ denote the number of vertives of $\mathcal{L}'$ on, above and below the neutral line, respectively. We point out that, since $\mathcal{L}'$ is not assumed complete, the integers $m_+$ and $m_-$ could be both zero. Moreover, admissibility implies that $ m_+ = 0 $ if and only if $m_- = 0$. With this notation, the convex polytope $F(\mathcal{L}')$ is affinely equivalent to
$$ C^{m_0}(\delta_{m_+ - 1} \times \delta_{m_- - 1}).$$
Indeed, this is immediate from Proposition \ref{prop:from-H-to-K}, using the explicit description of faces in a product and in the cone construction, together with the fact that a convex polytope is the convex hull of its vertices (see \cite{Zie} for these well known properties of polytopes). To be precise, if $m_+ = m_- = 0$, $F(\mathcal{L}')$ equals $\delta_{m_0 - 1}$, interpreted as the $m_0$-fold cone over the empty face.

In particular, we see that the convex polytope $F(\mathcal{L}')$ is intrinsic to $\mathcal{L}'$ up to affine equivalence. We can therefore define the dual complex of $K^{n-1}$ as
\begin{equation}\label{equation:dualcplxdef}
\Delta(K^{n-1}) = \bigsqcup_{\mathcal{L}} \Delta_K(\mathcal{L})/\sim.
\end{equation}
Here $\mathcal{L}$ runs over the complete admissible line charts with $n+1$ vertices, and $\sim$ is the equivalence relation that, if $\mathcal{L}'$ is a line subchart of both $\mathcal{L}_1$ and $\mathcal{L}_2$, identifies $F(\mathcal{L}')$ in $\Delta_K(\mathcal{L}_1)$ and $\Delta_K(\mathcal{L}_2)$.

We note that as an immediate consequence of 
Proposition \ref{prop:dual-const-dim}, the dual complex $\Delta(K^{n-1})$ is of pure dimension $n - 1$. In Section \ref{sec:dualcplxIII}, we explicitly compute $\Delta(K^{2})$. We plan to return to the general case in future work.

\subsection{Dual complexes for $n=3$ and $n=4$}
We illustrate some of the results in Section \ref{subsec:dualcpldef} for small values of $n$. In particular, we clarify the relationship between the dual complexes $\Delta_K$ and $\Delta_H$ for any complete admissible line chart with $n+1$ vertices, for $n=3$ of $n=4$. 

\begin{example}\label{eg:example-dual-complex}
    We continue the construction from \ref{eg:admissible-123-again}. For the dual complex $\Delta_K$ that represents smoothings of the stratum $D(1,2,3)$ (which is a $2$-simplex), we represent it as a subset of $\Delta_H$ (which is a $3$-simplex or a tetrahedron). The line chart associated to the stratum $D(1,2,3)$ has $2$ vertices above the line $y=2$, $1$ vertex below this line, and $1$ vertex on this line. By intersecting $\Delta_H$ with the plane $\Lambda = \{z_1 + z_2 - z_3 =0\}$, we obtain $\Delta_K$ as a subset of $\Delta_H$.
  
    Similarly, we consider smoothings of the stratum $D(1,2,-3)$. The line chart associated to it has $1$ vertex above $y=0$, $1$ vertex below this line, and $2$ vertices on this line. By intersecting $\Delta_H$ with the plane $\Lambda = \{z_1 - z_2 = 0\}$, we obtain $\Delta_K$ as a subset of $\Delta_H$.

    Figure \ref{fig:example-dual-complex} exhibits how $\Delta_K$ embeds in $\Delta_H$ as a subset in both cases. The tetrahedron and the shaded triangle represent $\Delta_H$ and $\Delta_K$ in both pictures.

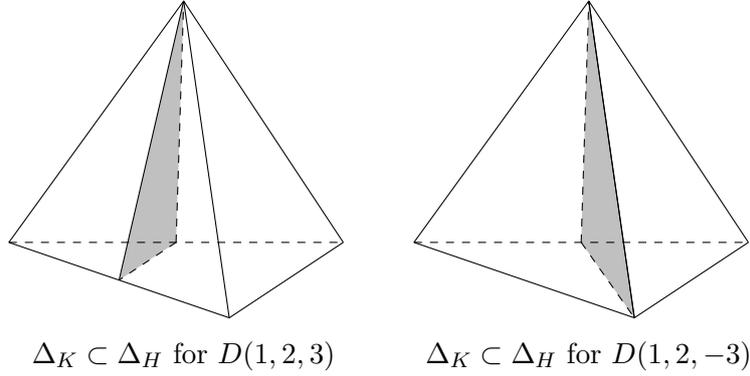
\begin{figure}
    \begin{tikzpicture}
        \coordinate (A) at (-2.3,0);
        \coordinate (B) at (2.1,0);
        \coordinate (C) at (0.6,-1);
        \coordinate (AB) at ($1/2*(A) + 1/2*(B)$);
        \coordinate (AC) at ($1/2*(A) + 1/2*(C)$);
        \coordinate (D) at (0,3.2);
        \draw (A) -- (C) -- (B);
        \draw[dashed] (A) -- (B);
        \draw (A) -- (D);
        \draw (B) -- (D);
        \draw (C) -- (D);
        \draw[dashed] (AB) -- (D);
        \draw (AC) -- (D);
        \draw[dashed] (AB) -- (AC);
        \fill[nearly transparent] (AB) -- (AC) -- (D) -- (AB);
        \node at (0,-1.5) {$\Delta_K \subset \Delta_H$ for $D(1,2,3)$};
    \end{tikzpicture}
    \qquad
    \begin{tikzpicture}
        \coordinate (A) at (-2.3,0);
        \coordinate (B) at (2.1,0);
        \coordinate (C) at (0.6,-1);
        \coordinate (AB) at ($1/2*(A) + 1/2*(B)$);
        \coordinate (AC) at ($1/2*(A) + 1/2*(C)$);
        \coordinate (D) at (0,3.2);
        \draw (A) -- (C) -- (B);
        \draw[dashed] (A) -- (B);
        \draw (A) -- (D);
        \draw (B) -- (D);
        \draw (C) -- (D);
        \draw[dashed] (AB) -- (D);
        \draw (C) -- (D);
        \draw[dashed] (AB) -- (C);
        \fill[nearly transparent] (AB) -- (C) -- (D) -- (AB);
        \node at (0,-1.5) {$\Delta_K \subset \Delta_H$ for $D(1,2,-3)$};
    \end{tikzpicture}
    \caption{Dual complexes in Example \ref{eg:example-dual-complex}}
    \label{fig:example-dual-complex}
\end{figure}
\end{example}

In the above example, we have seen two different relative positions of $\Delta_K$ in $\Delta_H$ (up to isomorphism of the tetrahedron). We can show that they have exhausted all possibilities when $n=3$, and arrive at the following conclusion.

\begin{cor}
    
    Assume $n=3$. Then there are only two 
possible relative positions of $\Delta_K$ in $\Delta_H$. These are described in Example \ref{eg:example-dual-complex}, see also Figure 
\ref{fig:example-dual-complex}. Moreover, the entire dual complex of the relative Kummer variety is a union of $2$-simplices. 
\end{cor}

\begin{proof}
    When $n=3$, an admissible complete line chart contains $4$ vertices. Since $n_+$, $n_-$ and $n_0$ are positive integers, one of them equals $2$ while both others are equal to $1$. If $n_+=2$ or $n_-=2$, we obtain the relative position of $\Delta_K$ in $\Delta_H$ that is similar to the case of $D(1,2,3)$ in Example \ref{eg:example-dual-complex}; if $n_0=2$, we obtain the relative position that is similar to the case of $D(1,2,-3)$ in the same example.  
     
    Therefore these are the only possible relative positions when $n=3$. Since in either case $\Delta_K$ is a $2$-simplex, we conclude that the entire dual complex is a union of these $2$-simplices.
\end{proof}

We now increase the dimension and analyze the building blocks of the entire dual complex of the relative Kummer variety when $n=4$.

\begin{example}\label{eg:4d-example-cont}
    We continue with Example \ref{eg:4d-example}. Here $\Delta_H$, for the stratum $E(1,2,3,4)$, is a $4$-simplex. We place it in $\R^5$ as required by Proposition \ref{prop:from-H-to-K}, using the parameters $n_+=n_-=2$ and $n_0=1$. By Lemma \ref{lemma:descriteratedconeprod}, the intersection of such a $4$-simplex and the hyperplane $z_1 + z_2 - z_3 - z_4 = 0$ is a pyramid over a square. 
\end{example}

The above example shows that the pyramid over a square is one type of building block for the entire dual complex of the relative Kummer variety. The following result gives a complete list of building blocks for the case of $n=4$.

\begin{cor}\label{cor:possible-cells-in-dual}
    For $n=4$, there are three possible relative positions of $\Delta_K$ in $\Delta_H$. In two of the three cases, $\Delta_K$ is a $3$-simplex (tetrahedron); while in the other case, $\Delta_K$ is a pyramid over a square. Hence the entire dual complex for the relative Kummer variety is a union of tetrahedra and pyramids over squares.
\end{cor}

\begin{proof}
    A complete admissible line chart consists of $5$ vertices. We observe that only the first, third and fifth vertex could be on the neutral line, hence $v_0 \leq 3$. 
    
    If $n_0 = 1$, then the vertex that lies on the neutral line must be the third one, hence $n_+=n_-=2$. It follows that $\Delta_K$ is the pyramid over a square, as described in Example \ref{eg:4d-example-cont}. 
    
    If $n_0=2$, then $\{ n_+, n_- \} = \{ 1, 2 \}$ and if $n_0=3$, then $n_+=n_-=1$. In both these cases, Lemma \ref{lemma:descriteratedconeprod} asserts that $\Delta_K$ is a tetrahedron.
    
    Summarizing all the three possibilities, we conclude that tetrahedra and pyramids over squares are building blocks of the entire dual complex for the relative Kummer variety when $n=4$.
\end{proof}

The occurrence of pyramids over squares indicates that irreducible components do not always intersect transversely. However, we will mainly focus on the cases with $n \leq 3$ in the following sections, and leave this phenomenon for future discussion.

For each fixed value of $n$, we can always follow the proof of Corollary \ref{cor:possible-cells-in-dual} to find out all possible building blocks of the dual complex. Since there are possibly more $n$-cells that are not $n$-simplices, we expect that there are more non-transversal intersections of irreducible components in the degenerate fiber of the relative Kummer variety.

\subsection{Number of strata in every dimension}

It is difficult to give a closed formula for the number of strata in every dimension in general, so in this subsection we focus on the case $n=3$ (while $\delta$ and $N$ remain arbitrary). The corresponding minimal and maximal dimension of an admissible stratum is $2\delta -2$ and $2\delta$. Each stratum of dimension $2\delta-2$, $2\delta-1$ and $2\delta$ is represented by a face, edge and vertex respectively in the dual complex.

\begin{prop}\label{prop:number}
    For $n=3$, the number of strata of dimension $2\delta-2$, $2\delta-1$ and $2\delta$ is $4N^2$, $6N^2 + 3N$ and $2N^2 + 3N + 1$ respectively.
\end{prop}

\begin{proof}
    The number of strata of dimension $2\delta -2$ is given in Proposition \ref{prop:number-deepest}. To motivate the following calculation, we recall that there are four complete admissible line charts, which are given in Figure \ref{fig:admissible3}. Each of them determines a set of values of $\{x_1, x_2, x_3\}$. For the values of $\{\tau_1, \tau_2, \tau_3\}$, two of them can be chosen arbitrarily among all residue classes modulo $N$, 
    while the third is then determined by the admissibility. Therefore the total number of such strata is $4N^2$.
    
    The number of strata in other dimensions can be computed along a similar line. We first find all admissible line charts representing strata in a particular dimension, which determine the values of $\{x_1, x_2, x_3\}$, then count the number of possible combinations of $\{\tau_1, \tau_2, \tau_3\}$. Among different line charts that represent the same strata, which could show up due to the reason explained in Subsection \ref{subsubsec:non-uniqueness}, we keep only one of them.
    
    To find all admissible line charts, we recall that it suffices to find all admissible line subcharts of the complete admissible line charts. Such a calculation has been carried out in Examples \ref{eg:all-admissible-123} and \ref{eg:all-admissible-12minus3} for $D(1,2,3)$ and $D(1,2,-3)$. The remaining cases of $D(-1,-2,-3)$ and $D(-1,-2,3)$ are mirror images of both above cases across the $x$-axis. We collect all these admissible line charts and classify them into the following four types.

    \emph{Type 1.} The dimension of strata is $2\delta-1$, and the values of $\{x_1, x_2, x_3\}$ are all distinct; see Figure \ref{fig:typeA}.
    
\begin{figure}
    \begin{tikzpicture}
        \draw[gray!30, step=1] (0,-1) grid (3,1);
        \node at (-0.3,0) {$O$};
        \draw[thick] (1,-1) -- (3,1);
        \draw (0,0) -- (3,0);
        \node at (3.75,0) {$y=0$};
        \node at (1.5,-1.5) {$C(1,2,-3)$};
        \filldraw (1,-1) circle [radius=2pt];
        \filldraw (2,0) circle [radius=2pt];
        \filldraw (3,1) circle [radius=2pt];
    \end{tikzpicture}
    \begin{tikzpicture}
        \draw[gray!30, step=1] (0,-1) grid (3,1);
        \node at (-0.3,0) {$O$};
        \draw[thick] (1,1) -- (3,-1);
        \draw (0,0) -- (3,0);
        \node at (3.75,0) {$y=0$};
        \node at (1.5,-1.5) {$C(-1,-2,3)$};
        \filldraw (1,1) circle [radius=2pt];
        \filldraw (2,0) circle [radius=2pt];
        \filldraw (3,-1) circle [radius=2pt];
    \end{tikzpicture}
    \begin{tikzpicture}
        \draw[gray!30, step=1] (0,-1) grid (3,1);
        \node at (-0.3,0) {$O$};
        \draw[thick] (0,0) -- (2,0);
        \draw (0,0) -- (3,0);
        \node at (3.75,0) {$y=0$};
        \node at (1.5,-1.5) {$B(0,1,-1)$};
        \filldraw (0,0) circle [radius=2pt];
        \filldraw (2,0) circle [radius=2pt];
    \end{tikzpicture}
    \caption{Admissible line charts of type 1 in Proposition \ref{prop:number}}
    \label{fig:typeA}
\end{figure}
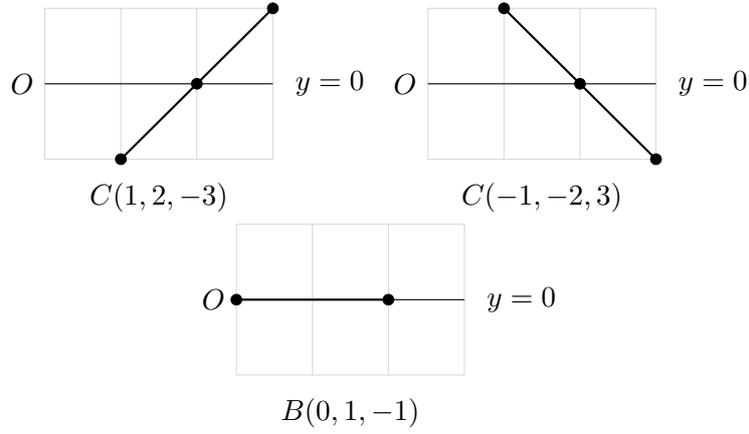

	For each of these line charts, since $\{x_1, x_2, x_3\}$ have distinct values, $(\tau_1, \tau_2, \tau_3)$ is an ordered triple of residue classes modulo $N$, which is subject to \eqref{eqn:admissible-second}. It follows that we can require two of them to take arbitrary values, while the third one is automatically determined. Therefore each line chart contributes $N^2$ strata.
    
    \emph{Type 2.} The dimension of the strata is still $2\delta-1$, two among the values of $\{x_1, x_2, x_3\}$ agree and the third is different; see Figure \ref{fig:typeB}.
    
\begin{figure}
    \begin{tikzpicture}
        \draw[gray!30, step=1] (0,0) grid (3,3);
        \node at (-0.3,0) {$O$};
        \draw[thick] (0,0) -- (2,2);
        \draw[thick] (2,2) -- (3,3);
        \draw (0,2) -- (3,2);
        \node at (3.75,2) {$y=2$};
        \node at (1.5,-0.5) {$C(1,2,2)$};
        \filldraw (0,0) circle [radius=2pt];
        \filldraw (2,2) circle [radius=2pt];
        \filldraw (3,3) circle [radius=2pt];
    \end{tikzpicture}
    \begin{tikzpicture}
        \draw[gray!30, step=1] (0,0) grid (3,-3);
        \node at (-0.3,0) {$O$};
        \draw[thick] (0,0) -- (2,-2);
        \draw[thick] (2,-2) -- (3,-3);
        \draw (0,-2) -- (3,-2);
        \node at (3.75,-2) {$y=-2$};
        \node at (1.5,-3.5) {$C(-1,-2,-2)$};
        \filldraw (0,0) circle [radius=2pt];
        \filldraw (2,-2) circle [radius=2pt];
        \filldraw (3,-3) circle [radius=2pt];
    \end{tikzpicture}
    \begin{tikzpicture}
        \draw[gray!30, step=1] (0,0) grid (3,3);
        \node at (-0.3,0) {$O$};
        \draw[thick] (0,0) -- (1,1);
        \draw[thick] (1,1) -- (3,3);
        \draw (0,2) -- (3,2);
        \node at (3.75,2) {$y=2$};
        \node at (1.5,-0.5) {$C(1,1,2)$};
        \filldraw (0,0) circle [radius=2pt];
        \filldraw (1,1) circle [radius=2pt];
        \filldraw (3,3) circle [radius=2pt];
    \end{tikzpicture}
    \begin{tikzpicture}
        \draw[gray!30, step=1] (0,0) grid (3,-3);
        \node at (-0.3,0) {$O$};
        \draw[thick] (0,0) -- (1,-1);
        \draw[thick] (1,-1) -- (3,-3);
        \draw (0,-2) -- (3,-2);
        \node at (3.75,-2) {$y=-2$};
        \node at (1.5,-3.5) {$C(-1,-1,-2)$};
        \filldraw (0,0) circle [radius=2pt];
        \filldraw (1,-1) circle [radius=2pt];
        \filldraw (3,-3) circle [radius=2pt];
    \end{tikzpicture}
    \begin{tikzpicture}
        \draw[gray!30, step=1] (0,-1) grid (3,1);
        \node at (-0.3,0) {$O$};
        \draw[thick] (0,0) -- (1,1);
        \draw[thick] (1,1) -- (3,-1);
        \draw (0,0) -- (3,0);
        \node at (3.75,0) {$y=0$};
        \node at (1.5,-1.5) {$C(-1,-1,2)$};
        \filldraw (0,0) circle [radius=2pt];
        \filldraw (1,1) circle [radius=2pt];
        \filldraw (3,-1) circle [radius=2pt];
    \end{tikzpicture}
    \begin{tikzpicture}
        \draw[gray!30, step=1] (0,-1) grid (3,1);
        \node at (-0.3,0) {$O$};
        \draw[thick] (0,0) -- (1,-1);
        \draw[thick] (1,-1) -- (3,1);
        \draw (0,0) -- (3,0);
        \node at (3.75,0) {$y=0$};
        \node at (1.5,-1.5) {$C(1,1,-2)$};
        \filldraw (0,0) circle [radius=2pt];
        \filldraw (1,-1) circle [radius=2pt];
        \filldraw (3,1) circle [radius=2pt];
    \end{tikzpicture}
    \caption{Admissible line charts of type 2 in Proposition \ref{prop:number}}
    \label{fig:typeB}
\end{figure}
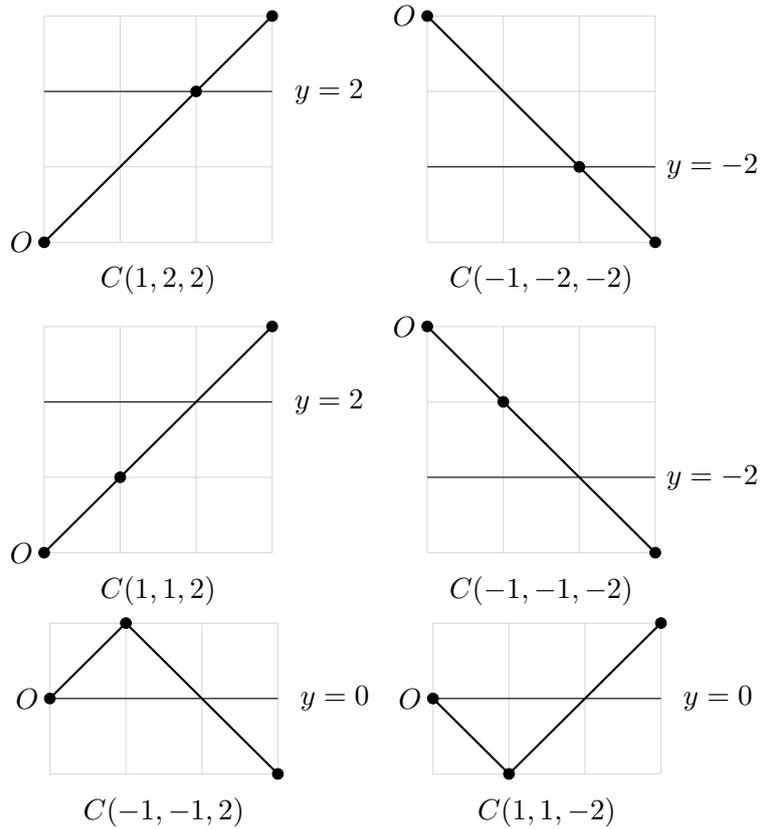

	For each of these line charts, two of the $x_i$'s have equal value, hence the values of the corresponding $\tau_i$'s are exchangeable. The number of choices of both these $\tau_i$'s is $\binom{N+1}{2}$, and the third $\tau_i$ is determined as a consequence of \eqref{eqn:admissible-second}.
	
	\emph{Type 3.} The dimension of the strata is $2\delta$; moreover, possibly after replacing the line chart by an equivalent one, among the values of $\{x_1, x_2, x_3\}$, still two of them agree and the third is different; see Figure \ref{fig:typeC}. For example, here we use that  $A(0,1,-1)$, $A(0,1,1)$ and $A(0,-1,-1)$ are equivalent (see Definition \ref{def:linechartequiv}).

\begin{figure}
    \begin{tikzpicture}
        \draw[gray!30, step=1] (0,-1) grid (3,1);
        \node at (-0.3,0) {$O$};
        \draw[thick] (1,-1) -- (3,1);
        \draw (0,0) -- (3,0);
        \node at (3.75,0) {$y=0$};
        \node at (1.5,-1.5) {$B(1,1,-2)$};
        \filldraw (1,-1) circle [radius=2pt];
        \filldraw (3,1) circle [radius=2pt];
    \end{tikzpicture}
    \begin{tikzpicture}
        \draw[gray!30, step=1] (0,-1) grid (3,1);
        \node at (-0.3,0) {$O$};
        \draw[thick] (1,1) -- (3,-1);
        \draw (0,0) -- (3,0);
        \node at (3.75,0) {$y=0$};
        \node at (1.5,-1.5) {$B(-1,-1,2)$};
        \filldraw (1,1) circle [radius=2pt];
        \filldraw (3,-1) circle [radius=2pt];
    \end{tikzpicture}
    \begin{tikzpicture}
        \draw[gray!30, step=1] (0,-1) grid (3,1);
        \node at (-0.3,0) {$O$};
        \draw (0,0) -- (3,0);
        \node at (3.75,0) {$y=0$};
        \node at (1.5,-1.5) {$A(0,1,-1)$};
        \filldraw (2,0) circle [radius=2pt];
    \end{tikzpicture}
    \caption{Admissible line charts of type 3 in Proposition \ref{prop:number}}
    \label{fig:typeC}
\end{figure}
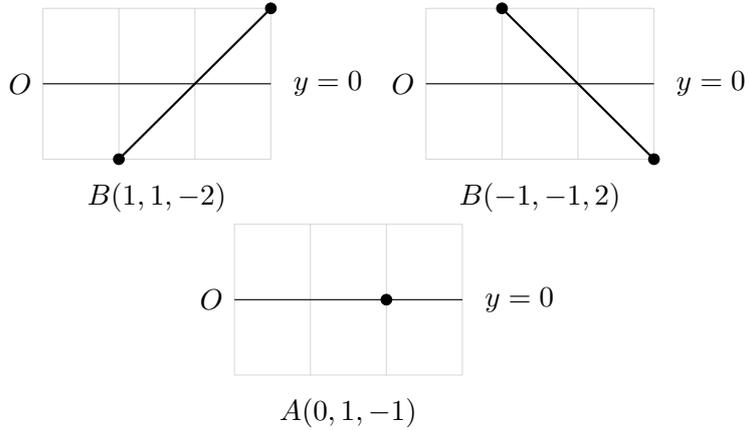

	The contribution in this case can be computed similarly as for the previous type. For each of these line charts, two of the $x_i$'s have equal value, hence the values of the corresponding $\tau_i$'s are exchangeable. The number of choices of both of these $\tau_i$'s is $\binom{N+1}{2}$, and the third $\tau_i$ is determined as a consequence of \eqref{eqn:admissible-second}.
	
	\emph{Type 4.} The dimension of strata is still $2\delta$, the values of $\{x_1, x_2, x_3\}$ are all equal; see Figure \ref{fig:typeD}.
	
\begin{figure}
	\begin{tikzpicture}
        \draw[gray!30, step=1] (0,0) grid (3,3);
        \node at (-0.3,0) {$O$};
        \draw[thick] (0,0) -- (3,3);
        \draw (0,2) -- (3,2);
        \node at (3.75,2) {$y=2$};
        \node at (1.5,-0.5) {$B(1,1,1)$};
        \filldraw (0,0) circle [radius=2pt];
        \filldraw (3,3) circle [radius=2pt];
    \end{tikzpicture}
    \begin{tikzpicture}
        \draw[gray!30, step=1] (0,0) grid (3,-3);
        \node at (-0.3,0) {$O$};
        \draw[thick] (0,0) -- (3,-3);
        \draw (0,-2) -- (3,-2);
        \node at (3.75,-2) {$y=-2$};
        \node at (1.5,-3.5) {$B(-1,-1,-1)$};
        \filldraw (0,0) circle [radius=2pt];
        \filldraw (3,-3) circle [radius=2pt];
    \end{tikzpicture}
    \begin{tikzpicture}
        \draw[gray!30, step=1] (0,-1) grid (3,1);
        \node at (-0.3,0) {$O$};
        \draw (0,0) -- (3,0);
        \node at (3.75,0) {$y=0$};
        \node at (1.5,-1.5) {$A(0,0,0)$};
        \filldraw (0,0) circle [radius=2pt];
    \end{tikzpicture}
	\caption{Admissible line charts of type 4 in Proposition \ref{prop:number}}
    \label{fig:typeD}
\end{figure}
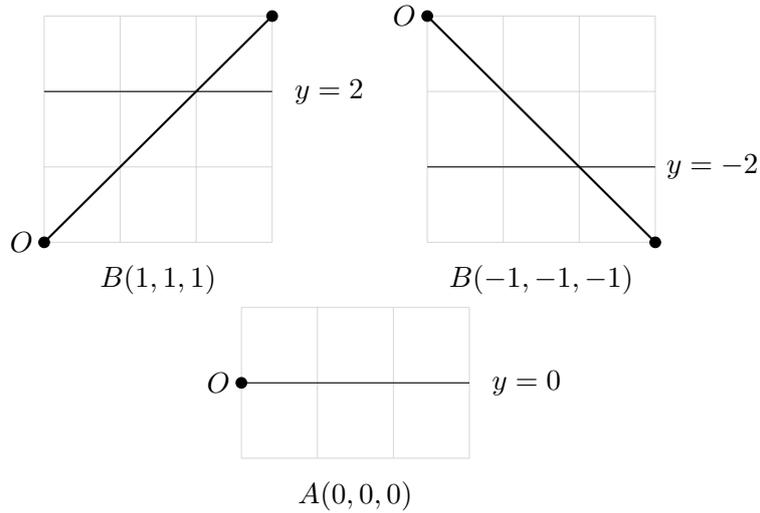

	We compute the contribution of each line chart to the number of strata. Since $\{x_1, x_2, x_3\}$ take the same values, $(\tau_1, \tau_2, \tau_3)$ becomes an unordered triple. The problem reduces to the computation of the number of such unordered triples satisfying \eqref{eqn:admissible-second} with $n=3$ and various values of $k$. For simplicity we denote the number of unordered solutions $(\tau_1, \tau_2, \tau_3)$ to such equation
	$$ \tau_1 + \tau_2 + \tau_3 + k \equiv 0 \pmod N $$
	by $m(k)$, then the number of strata given by the above line charts are $m(1), m(-1), m(0)$ respectively, and we need to compute their sum.
	
	We can observe the following simple properties of the $m(k)$'s:
	\begin{itemize}
		\item The value of $m(k)$ only depends on the residue class of $k \pmod N$.
		\item For every $k$, we have \begin{equation}\label{eqn:mk-property1} m(k) = m(k-3), \end{equation} since by adding $1$ to each component of a solution $(\tau_1, \tau_2, \tau_3)$, we obtain a solution to the above equation with $k$ replaced by $k-3$.
		\item The sum \begin{equation}\label{eqn:mk-property2} m(0) + m(1) + \dots + m(N-1) = \binom{N+2}{3}, \end{equation} because, when $k$ runs through all residue classes, $(\tau_1, \tau_2, \tau_3)$ runs through all unordered triples of residue classes.
	\end{itemize}
	
	We claim that these properties already imply
	\begin{equation}\label{eqn:sum-of-three}
		m(-1) + m(0) + m(1) = \frac{3}{N} \binom{N+2}{3}.
	\end{equation}
	
	Indeed, this can be seen in two cases: if $N$ is divisible by $3$, then \eqref{eqn:mk-property1} implies that among all summands in the left side of \eqref{eqn:mk-property2}, there are exactly $\frac{N}{3}$ of them equal to $m(-1)$, or $m(0)$, or $m(1)$, 
    hence \eqref{eqn:sum-of-three} holds; if $N$ is not divisible by $3$, then \eqref{eqn:mk-property1} implies that all summands in the left side of \eqref{eqn:mk-property2} are equal, which still gives \eqref{eqn:sum-of-three}.

	Finally we summarize the above calculation. We have already seen that the total number of strata of dimension $2\delta -2$ is $4N^2$. 
	
	The line charts associated to strata of dimension $2\delta - 1$ is either of type 1 (as in Figure \ref{fig:typeA}) or type 2 (as in Figure \ref{fig:typeB}). By the calculation above, the number of these strata is given by
	$$ 3N^2 + 6\binom{N+1}{2} = 6N^2 + 3N. $$
	
	Similarly, the line charts associated to strata of dimension $2\delta$ is either of type 3 (as in Figure \ref{fig:typeC}) or type 4 (as in Figure \ref{fig:typeD}). By the calculation above, the number of these strata is given by
	$$ 3\binom{N+1}{2} + \frac{3}{N}\binom{N+2}{3} = 2N^2 + 3N + 1. $$
	
	This finishes the calculation of the number of strata in each possible dimension.
\end{proof}

\begin{remark} 
    Later in Section \ref{sec:dualcplxIII}, we will prove a stronger statement when $n=3$, namely that the dual complex for the degeneration of generalized Kummer varieties is PL-homeomorphic to a standard $2$-simplex, see Theorem \ref{thm:dual-complex-n3}. In particular, the Euler characteristic is $1$, which agrees with Proposition \ref{prop:number}.
\end{remark}

\section{Kummer surfaces}
\label{subsec:Kummersurfaces}
In this section we assume that $\mathcal{Y} \to C $ has relative dimension $2$. We will describe the Kummer locus $ \mathcal{K}^1$, both in terms of the stratification introduced in \ref{subsec:stratification}, and as a scheme. This is subsequently used to describe the GIT quotient $ \mathcal{K}^1/\mathbb G[2] $, which forms a compactification of the family of Kummer surfaces $$ 
 \mathrm{Kum}(\mathcal{Y}) := \mathrm{Kum}^1(\mathcal{Y}^{sm}/C) \to C. $$

\subsection{The stratification}
\label{subsection:stratificationKumersurfaces}
We  start out by computing the strata of the (boundary of the) Kummer locus $\mathcal{K}^1$. Since $n=2$ and $ 0 \leq b \leq n$, there are three cases to consider.

\subsubsection{$(b=0)$}
Each component of $\mathrm{Kum}(\mathcal{Y})_s$ arises as the kernel of a summation map
$$ \mathcal{S}_{\mathbf{m}(i)} \to G_0^{\circ}, $$
 where the only nonzero indices of $\mathbf{m}(i)$ are $m_i = m_{2N-i} = 1$ if $ 0 < i < N $ and $m_0=2$, resp. $m_N=2$, if $i=0$, resp. $i=N$. The resulting component will be denoted $F^{\circ}_{i} $. 

Note that this also describes the degenerate fibers of the generic Kummer locus $ \mathcal{K}^1_{\circ} $ over $C[2]$, since the pullback $\mathcal{Y}^{sm} \times_C C[b]$ remains a group scheme for any $b \geq 0$, and the Kummer construction commutes with base change. 

\subsubsection{$(b=2)$}
We next show that there is no deepest stratum when $n=2$. Note that this is not the case when $n \geq 3$, by Proposition \ref{prop:number-deepest}.

\begin{lemma}
The fiber of $ \mathcal{K}^1$ over $ 0 \in C[2] $ is empty.
\end{lemma}
\begin{proof}
We use the approach and notation from Subsection \ref{subsec:deepeststrata}. When $n=2$, we have
$$ S = \{ 0, \varepsilon_2, \varepsilon_1 + \varepsilon_2 \}. $$
Since $\varepsilon \neq 0$, we can never have $\mathrm{max}~S = \mathrm{min}~S$. Moreover, the inequalities
$$ \mathrm{min}~S < 2k < \mathrm{max}~S $$
cannot both hold. Indeed, since $ \vert \varepsilon_1 + \varepsilon_2 \vert \leq 2$, we must have $k=0$. But $\varepsilon_2$ and $ \varepsilon_1 + \varepsilon_2 $ cannot simultaneously both be nonzero, and of opposite sign. 
\end{proof}

\subsubsection{$(b=1)$}
It remains to consider the case $b=1$.
 \begin{lemma}\label{lemma:KumSurb1}
The strata of $\mathcal{K}^1$ corresponding to $b=1$ are
$$ F_0^1, \ldots, F_{N-1}^1, $$
where $F_i^1$ is narrow and isomorphic to $(G_{4i+1})^{\circ}$ for each $i$.
\end{lemma}
\begin{proof}
We proceed as in Subsection \ref{subsec:higherdimstrata}. It is easy to check that the only possibility is when $\vert x_1 \vert = \vert x_2 \vert $ and $k=0$, in which case
$$ S' = \{ 0, \varepsilon_1 + \varepsilon_2 \}. $$
Then we must require that $\varepsilon_1 + \varepsilon_2 = 0$, thus, the stratum is narrow by Lemma \ref{lemma:narrowwidedescr}. Since $k=0$, we have $\tau_2 = N - \tau_1$. Any point of this stratum is a subscheme supported on components $G_{4 i + 1}$ and $G_{4 (N-i) - 1}$, in particular it is reduced, since $ 4i+1 $ is not congruent to $4(N-i) - 1$ modulo $4N$. The rest of the statement now follows directly from Proposition \ref{prop:stratumdfn}.
\end{proof}

We shall return to this discussion in Section \ref{sec:stratageom}, in particular Subsection \ref{subsection:geomstrataKum} where we will encounter the cases $b=0,2$ from a slightly different point of view.

\subsection{Scheme structure of the Kummer locus}

We now explain an alternative description of the Kummer locus $\mathcal{K}^1$, using the involution $j[2] \colon \mathcal{Y}[2] \to \mathcal{Y}[2] $ from Lemma \ref{lem:involutiongeneral}. 

\subsubsection{}
We recall from \cite[Section 3]{GHHZ21} that we have maps
\begin{equation}\label{ABCD}
\xymatrix{
     & \left( \mathcal{Y}[2] \times_{C[2]} \mathcal{Y}[2] \right)  \ar[d]^q \\
    \mathrm{Hilb}^2(\mathcal{Y}[2]/C[2])  \ar[r]^{hc} & \mathrm{Sym}^2(\mathcal{Y}[2]/C[2])
}    
\end{equation}
where $hc$ is the Hilbert-Chow morphism. Both maps are compatible with the formation of (semi)stable loci:
$$ \left( \mathcal{Y}[2] \times_{C[2]} \mathcal{Y}[2] \right)^{st} = q^{-1} \left(\mathrm{Sym}^2(\mathcal{Y}[2]/C[2])^{st}\right)$$
and 
$$ \mathrm{Hilb}^2(\mathcal{Y}[2]/C[2])^{st} = hc^{-1} \left(\mathrm{Sym}^2(\mathcal{Y}[2]/C[2])^{st}\right). $$

\begin{remark}
 In all three cases, the stable and semi-stable loci are the same.    
\end{remark}
\subsubsection{}

The graph of $j[2]$ is a closed immersion
$$ (1,j[2]) \colon \mathcal{Y}[2] \to \mathcal{Y}[2] \times_{C[2]} \mathcal{Y}[2] $$
whose image we shall denote $V$. Since $j[2]$ is an involution, $V$ is stable under the $\mathcal{S}_2$-action permuting the two factors of the product. Therefore, its quotient is a closed subscheme
$$ W := V/\mathcal{S}_2 \subset \mathrm{Sym}^2(\mathcal{Y}[2]/C[2]). $$
This construction is compatible with restriction to the stable locus;  writing $ V^{st} = V \cap \left( \mathcal{Y}[2] \times_{C[2]} \mathcal{Y}[2] \right)^{st}$, we get 
$$ V^{st} \to W^{st} = V^{st}/\mathcal{S}_2 \subset  \mathrm{Sym}^2(\mathcal{Y})^{st}. $$

\subsubsection{} 
We form the strict transform
$$ \widetilde{W^{st}} \to W^{st} $$
under the (restricted) Hilbert-Chow morphism
$$ hc \colon \mathrm{Hilb}^2(\mathcal{Y}[2]/C[2])^{st} \to \mathrm{Sym}^2(\mathcal{Y}[2]/C[2])^{st}. $$

\begin{prop}\label{prop:Kummerprops}
 The strict transform $ \widetilde{W^{st}}  $ is smooth over $C[2]$, stable under the action of $\mathbb G[2]$, and coincides with $ \mathcal{K}^1$ as a closed subscheme of $ \mathcal{H}^2$.
 
 \end{prop}
\begin{proof}
We observe that $ \widetilde{W^{st}} $ is closed in $ \mathcal{H}^2$, integral (since $ W^{st} $ is integral) and coincides with $ \mathcal{K}^1_{\circ} $ over $C[2]^*$. So it suffices to prove that it is smooth over $C[2]$. We shall use the fact that $ \mathrm{Hilb}^2(\mathcal{Y}[2]^{sm}/C[2]) $ can also be obtained by blowing up $\mathcal{Y}[2]^{sm} \times_{C[2]} \mathcal{Y}[2]^{sm}$ in the diagonal $\mathcal{D}$ and dividing out by the $\mathcal{S}_2$ action. 

We claim that $\mathcal{D}_V = \mathcal{D} \cap V^{st}$ is smooth over $C[2]$. If $ c \in C[2]$, $(\mathcal{D}_V)_c $ is the intersection of the graph of the involution $-1 \colon \mathcal{Y}[2]^{sm}_c \to \mathcal{Y}[2]^{sm}_c $ with the diagonal. This intersection is supported on the $2$-torsion points of $\mathcal{Y}[2]_c$, and the multiplicities are all equal to $1$ (since, at any fixed point, the induced map on the tangent space is multiplication by $-1$). Hence each closed fiber $(\mathcal{D}_V)_c $ is smooth. Since $\mathcal{D}_V $ is locally cut out by $2$ equations, $\mathcal{D}_V \subset V^{st}$ is a relative regular immersion by \cite[Lemma 31.22.7]{Sta} and in particular flat, which proves the claim. Consequently, the strict transform $\widetilde{V}^{st}$ of $ V^{st} $ is smooth. 

The quotient $\widetilde{V}^{st}/\mathcal{S}_2$ is a closed subscheme of the Hilbert scheme coinciding generically with $ \widetilde{W^{st}} $ so they must be equal, since they are integral schemes. It remains to prove that $\widetilde{V}^{st}/\mathcal{S}_2$ is smooth over $C[2]$. Since this quotient is Cohen-Macaulay by the Hochster-Roberts theorem \cite{Hochster-Roberts} and $C[2]$ is regular, flatness follows from \cite[Lemma 00R4]{Sta}. It therefore suffices to note that each closed fiber is smooth, by \cite[Lemma 3.9]{Overkamp}.
\end{proof}

\subsection{The GIT quotient} 
The GIT quotient $ I^2_{\mathcal{Y}/C} = \mathcal{H}^2/\mathbb G[2] $ is, as recalled in \ref{subsubsec:GITstability}, a projective dlt-model over $C$. In \ref{subsubsec:Kummerlocquot}, we moreover saw that the restriction of the quotient  
$$  \mathcal{H}^2 \to I^2_{\mathcal{Y}/C} $$
to $\mathcal{K}^1$ yields a geometric quotient $ K^1_{\mathcal{Y}/C} $ which is closed in $I^2_{\mathcal{Y}/C} $.

\begin{theorem}\label{theorem:Kummerdeg}
The GIT quotient $ K^1_{\mathcal{Y}/C} $ forms a projective Kulikov model over $C$. 

\end{theorem}
\begin{proof}
The definition of a Kulikov model was given in Definition \ref{def:KulikovMod}. By Proposition \ref{prop:Kummerprops}, the Kummer locus $\mathcal{K}^1$ is smooth over $C[2]$ and from the description in \ref{lemma:KumSurb1} we see that there are no non-trivial stabilizers. Using Luna's \'etale slice theorem, it is straightforward to verify that $K^1_{\mathcal{Y}/C}$ is regular, and that the fiber over $ 0 \in C $ is a reduced strict normal crossings divisor. The scheme $K^1_{\mathcal{Y}/C}$ contains $\mathrm{Kum}^1(\mathcal{Y}^{sm}/C)$ as an open subset, and the complement has codimension $2$; this follows again from Lemma \ref{lemma:KumSurb1}. Since the latter family has trivial relative canonical sheaf, the same holds for $K^1_{\mathcal{Y}/C} \to C$. 
\end{proof}

Recall that Kulikov models of K3 surfaces have been classified by Kulikov, Persson and Pinkham, see \cite[(6.1.2)]{HN} for an overview. They are subdivided into three types, depending on the unipotency of the monodromy of the special fiber. The models appearing in Theorem \ref{theorem:Kummerdeg} are of type II. This implies that the special fiber of $ K^1_{\mathcal{Y}/C} $ is a chain in which the two exterior surfaces are rational, the interior surfaces are elliptic ruled, and all intersection curves are elliptic sections. We will come back to the geometry of the strata in more detail 
in Section \ref{sec:stratageom}, in particular Subsection \ref{subsection:geomstrataKum}, where we discuss the geometry of the strata in $ (K^{n-1}_{\mathcal{Y}/C})_0 $ in more detail.

Degenerations of Kummer surfaces have also been studied by Overkamp in \cite{Overkamp}. His approach is to start with a relatively projective model $\mathcal{X}$ of an abelian surface $A$ constructed via uniformization techniques, making sure that the involution $ -1_A $ extends. Then the desired model of $\mathrm{Kum}(A)$ is obtained as a quotient of a suitable modification of $\mathcal{X}$ by the action of the involution. It is noteworthy that he does not require $\mathcal{X}$ to be a simple degeneration. His methods do not extend to the case where $n > 2$, which is the main focus of our paper.

\section{The dual complex for $n=3$}\label{sec:dualcplxIII}

In this section, we will explicitly construct the dual complex for $n=3$ and arbitrary value of $N$, which we denote as $\Delta(K^2_{\mathcal{Y}/C})$. We have seen from Proposition \ref{prop:dual-const-dim} that $\Delta(K^2_{\mathcal{Y}/C})$ is a complex of dimension $2$, whose cells of dimension $0$, $1$ and $2$ will be called vertices, edges and faces respectively.

\subsection{The notion of a local chart}

Recall that the faces of $\Delta(K^2_{\mathcal{Y}/C})$ have been completely classified in Proposition \ref{prop:stratadeterm}. The closure of each face is a $2$-simplex, denoted as $\mathcal{D}^\pm$ or $\mathscr{D}^\pm$, which has been explicitly described in Subsection \ref{subsec:numerical-smoothing}. The dual complex $\Delta(K^2_{\mathcal{Y}/C})$ will be obtained by glueing them along edges. To obtain the full dual complex, we first glue the $2$-simplices that share a common vertex. For this purpose we use

\begin{definition}
	In the dual complex, the union of all $2$-simplices which contain any particular vertex will be called the \emph{local chart} around this vertex. This vertex is called the \emph{center} of the local chart.
\end{definition}

Recall that in the proof of Proposition \ref{prop:number}, we also have a classification of edges into Type 1 and Type 2, as well as a classification of vertices into Type 3 and Type 4. The following observation helps understand the role of some particular local charts.

\begin{lemma}\label{lem:glue-type-4}
	Each $2$-simplex in $\Delta(K^2_{\mathcal{Y}/C})$ contains $2$ vertices of Type 3 and $1$ vertex of Type 4. Therefore $\Delta(K^2_{\mathcal{Y}/C})$ can be obtained by glueing local charts around vertices of Type 4 along edges.
\end{lemma}

\begin{proof}
	The vertices of $\mathcal{D}^\pm$ and $\mathscr{D}^\pm$ have also been explicitly computed in Subsections \ref{subsubsec:dualcplx123} and \ref{subsubsec:dualcplx12minus3}. We can see that in either case, there is precisely one vertex of Type 4 (given by $B_1$ in Subsection \ref{subsubsec:dualcplx123} and $A_1$ in Subsection \ref{subsubsec:dualcplx12minus3}), as well as both other vertices of Type 3. It follows that each $2$-simplex is contained in exactly one local chart around a vertex of Type 4, which proves the claim.
\end{proof}

We will have to describe all local charts around vertices of Type 4, and the algorithm for glueing them. Recall from Proposition \ref{prop:number} that the line chart associated to a vertex of Type 4 has three possibilities: $A(0,0,0)$ or $B(1,1,1)$ or $B(-1,-1,-1)$. To be more precise, we will also take the values of $\tau_i$'s into consideration, then such a vertex of Type 4 will be labelled as $A(\tau;0,0,0)$ or $B(\tau;1,1,1)$ or $B(\tau;-1,-1,-1)$, where $\tau = (\tau_1, \tau_2, \tau_3)$ with each component viewed as a residue class modulo $N$. For simplicity, we will sometimes omit the values of $x_i$'s, and label the vertex as $A_1(\tau)$ or $B_1^+(\tau)$ or $B_1^-(\tau)$. 
We also remind the readers that since the values of $x_i$'s are equal, $\tau$ should be viewed as an unordered triple for each vertex of Type 4. 
Next we consider the local chart around each of these vertices.

\subsection{The local chart $\mathscr{C}(\tau)$ around $A_1(\tau)$}\label{subsec:3-charts-A}

For the vertex $ A_1(\tau) \coloneqq A(\tau;0,0,0)$, the admissibility condition \eqref{eqn:admissible-second} requires 
\begin{equation} \label{eqn:adm-A}
	\tau_1 + \tau_2 + \tau_3  \equiv 0 \pmod N.
\end{equation}
We denote the local chart around the vertex $A_1(\tau)$ as $\mathscr{C}(\tau)$. The following is an easy observation

\begin{lemma} \label{lem:local-around-A}
	The local chart around the vertex $A(\tau;0,0,0)$ can be given by
	\begin{equation} \label{eq:config1-C}
		\mathscr{C}(\tau) = \bigcup_{\zeta \in \mathcal{S}_3} \left( \mathscr{D}^+(\zeta(\tau)) \cup \mathscr{D}^-(\zeta(\tau)) \right)
	\end{equation}
	as a union of $2$-simplices, where $\mathcal{S}_3$ is the symmetric group on $3$ elements.
\end{lemma}

\begin{proof}
	First of all, the only complete admissible line charts that contain $A(0,0,0)$ as a line subchart are $D(1,2,-3)$ and $D(-1,-2,3)$. Moreover, specialization does not alter the values of $\tau_i$'s. Therefore the $2$-simplices which contain $A_1(\tau)$ are exactly given by $\mathscr{D}^\pm$ parametrized by any permutation of $\tau$. The claim follows.
\end{proof}

It is clear from Lemma \ref{lem:local-around-A} that the shape of $\mathscr{C}(\tau)$ depends on the number of distinct values among the components of $\tau$. In the following we will distinguish three cases.

\subsubsection{The set $\{ \tau_1, \tau_2, \tau_3 \}$ contains $3$ distinct values} \label{subsubsec:3-values-C}

In such a case, the $2$-simplices on the right side of \eqref{eq:config1-C} are distinct, hence $\mathscr{C}(\tau)$ consists of $12$ distinct $2$-simplices. Let $(x,y,z)$ be any permutation of $(\tau_1, \tau_2, \tau_3)$. By manipulating the corresponding line charts, we observe the following rules concerning the relative positions of the $2$-simplicies
\begin{itemize}
	\item[(a)] the $2$-simplices $\mathscr{D}^+(x,y,z)$ and $\mathscr{D}^-(x,z,y)$ share a common edge labelled by $B(x,y,z;0,1,-1)$;
	\item[(b1)] the $2$-simplices $\mathscr{D}^\pm(x,y,z)$ and $\mathscr{D}^\pm(y,x,z)$ share a common edge labelled by $C(x,y,z; \pm 1, \pm 1, \mp 2)$;
	\item[(c)] the edge of $\mathscr{D}^\pm(x,y,z)$ opposite to the vertex $A(\tau;0,0,0)$ is labelled by $C(x,y,z; \pm 1, \pm 2, \mp 3) = C(x,y,z \mp 1;\pm 1, \pm 2, \pm 3)$, which is also an edge of $\mathcal{D}^\pm(x,y,z \mp 1)$.
\end{itemize}

The first two rules determine that the local chart around $A_1(\tau)$ can be viewed as a full disk, which splits into $12$ sectors; see Figure \ref{fig:fulldisk12-C}. The last rule determines all $2$-simplicies that are attached to this local chart by edges. In particular, none of the boundary edges of this local chart is on the boundary of the dual complex $\Delta(K^2_{\mathcal{Y}/C})$.

\begin{figure}[h]
	\begin{tikzpicture}
		\draw (0,0) circle (5cm);
		\foreach \x in {0,...,5}
			\draw (30*\x:5cm) -- (30*\x+180:5cm);
		\foreach \x in {0,...,11}
			\fill (30*\x:5cm) circle (4pt);
		\fill (0,0) circle (4pt);
		\draw (15:3.5cm) node {$\mathscr{D}^-(z,y,x)$};
		\draw (45:3.5cm) node {$\mathscr{D}^+(z,x,y)$};
		\draw (75:3.8cm) node {$\mathscr{D}^+(x,z,y)$};
		\draw (105:3.8cm) node {$\mathscr{D}^-(x,y,z)$};
		\draw (135:3.5cm) node {$\mathscr{D}^-(y,x,z)$};
		\draw (165:3.5cm) node {$\mathscr{D}^+(y,z,x)$};
		\draw (195:3.5cm) node {$\mathscr{D}^+(z,y,x)$};
		\draw (225:3.5cm) node {$\mathscr{D}^-(z,x,y)$};
		\draw (255:3.8cm) node {$\mathscr{D}^-(x,z,y)$};
		\draw (285:3.8cm) node {$\mathscr{D}^+(x,y,z)$};
		\draw (315:3.5cm) node {$\mathscr{D}^+(y,x,z)$};
		\draw (345:3.5cm) node {$\mathscr{D}^-(y,z,x)$};
		\draw (15:6.4cm) node {$\mathcal{D}^-(z,y,x+1)$};
		\draw (45:6.1cm) node {$\mathcal{D}^+(z,x,y-1)$};
		\draw (75:5.5cm) node {$\mathcal{D}^+(x,z,y-1)$};
		\draw (105:5.5cm) node {$\mathcal{D}^-(x,y,z+1)$};
		\draw (135:6.1cm) node {$\mathcal{D}^-(y,x,z+1)$};
		\draw (165:6.4cm) node {$\mathcal{D}^+(y,z,x-1)$};
		\draw (195:6.4cm) node {$\mathcal{D}^+(z,y,x-1)$};
		\draw (225:6.1cm) node {$\mathcal{D}^-(z,x,y+1)$};
		\draw (255:5.5cm) node {$\mathcal{D}^-(x,z,y+1)$};
		\draw (285:5.5cm) node {$\mathcal{D}^+(x,y,z-1)$};
		\draw (315:6.1cm) node {$\mathcal{D}^+(y,x,z-1)$};
		\draw (345:6.4cm) node {$\mathcal{D}^-(y,z,x+1)$};
		\draw (0,-0.6cm) node [fill=white] {$A_1(x,y,z)$};
	\end{tikzpicture}
	\caption{Local chart $\mathscr{C}(\tau)$ for pairwise distinct $\tau_i$'s}
    \label{fig:fulldisk12-C}
\end{figure}

\subsubsection{The set $\{ \tau_1, \tau_2, \tau_3 \}$ contains $2$ distinct values} \label{subsubsec:2-values-C}

In such a case, the right side of \eqref{eq:config1-C} consists of $6$ distinct $2$-simplices forming $\mathscr{C}(\tau)$. The rules concerning the relative positions of $2$-simplices are similar to those in the above case, namely
\begin{itemize}
	\item[(a)] the $2$-simplices $\mathscr{D}^+(x,y,z)$ and $\mathscr{D}^-(x,z,y)$ share a common edge labelled by $B(x,y,z;0,1,-1)$;
	\item[(b1)] when $x \neq y$, the $2$-simplices $\mathscr{D}^\pm(x,y,z)$ and $\mathscr{D}^\pm(y,x,z)$ share a common edge labelled by $C(x,y,z; \pm 1, \pm 1, \mp 2)$;
	\item[(b2)] when $x = y$, the $2$-simplex $\mathscr{D}^\pm(x,y,z)$ is the only one that contains the edge labelled by $C(x,y,z; \pm 1, \pm 1, \mp 2)$;
	\item[(c)] the edge of $\mathscr{D}^\pm(x,y,z)$ opposite to the vertex $A_1(\tau)$ is labelled by $C(x,y,z; \pm 1, \pm 2, \mp 3) = C(x,y,z \mp 1;\pm 1, \pm 2, \pm 3)$, which is also an edge of $\mathcal{D}^\pm(x,y,z \mp 1)$.
\end{itemize}

The first two rules determine that the local chart around $A_1(\tau)$ can be viewed as a semi-disk, which splits into $6$ sectors; see Figure \ref{fig:semidisk6-C}. The last two rules determines all $2$-simplicies that are attached to this local chart by edges, as well as both boundary edges of this local chart that also lie on the boundary of the dual complex $\Delta(K^2_{\mathcal{Y}/C})$; see the thick edges in Figure \ref{fig:semidisk6-C}.

\begin{figure}[h]
	\begin{tikzpicture}
		\draw (5,0) arc (0:180:5);
		\foreach \x in {0,...,6}
			\draw (30*\x:5cm) -- (0,0);
		\foreach \x in {0,...,6}
			\fill (30*\x:5cm) circle (4pt);
		\draw [ultra thick] (0:5cm) -- (0,0);
		\draw [ultra thick] (180:5cm) -- (0,0);
		\fill (0,0) circle (4pt);
		\draw (15:3.5cm) node {$\mathscr{D}^-(y,y,x)$};
		\draw (45:3.5cm) node {$\mathscr{D}^+(y,x,y)$};
		\draw (75:3.8cm) node {$\mathscr{D}^+(x,y,y)$};
		\draw (105:3.8cm) node {$\mathscr{D}^-(x,y,y)$};
		\draw (135:3.5cm) node {$\mathscr{D}^-(y,x,y)$};
		\draw (165:3.5cm) node {$\mathscr{D}^+(y,y,x)$};
		\draw (15:6.4cm) node {$\mathcal{D}^-(y,y,x+1)$};
		\draw (45:6.1cm) node {$\mathcal{D}^+(y,x,y-1)$};
		\draw (75:5.5cm) node {$\mathcal{D}^+(x,y,y-1)$};
		\draw (105:5.5cm) node {$\mathcal{D}^-(x,y,y+1)$};
		\draw (135:6.1cm) node {$\mathcal{D}^-(y,x,y+1)$};
		\draw (165:6.4cm) node {$\mathcal{D}^+(y,y,x-1)$};
		\draw (0,-0.5cm) node {$A_1(x,y,y)$};
	\end{tikzpicture}
    \caption{Local chart $\mathscr{C}(\tau)$ for $\tau_i$'s taking two distinct values; thick edges are on the boundary of $\Delta(K^2_{\mathcal{Y}/C})$}
    \label{fig:semidisk6-C}
\end{figure}

\subsubsection{The set $\{ \tau_1, \tau_2, \tau_3 \}$ contains $1$ distinct value} \label{subsubsec:1-values-C}

In such a case, the right side of \eqref{eq:config1-C} consists of $2$ distinct $2$-simplices forming $\mathscr{C}(\tau)$. The rules concerning the relative positions of $2$-simplices are similar to those in both above cases, namely
\begin{itemize}
	\item[(a)] the $2$-simplices $\mathscr{D}^+(x,x,x)$ and $\mathscr{D}^-(x,x,x)$ share a common edge labelled by $B(x,x,x;0,1,-1)$;
	\item[(b2)] the $2$-simplex $\mathscr{D}^\pm(x,x,x)$ is the only one that contains the edge labelled by $C(x,x,x; \pm 1, \pm 1, \mp 2)$;
	\item[(c)] the edge of $\mathscr{D}^\pm(x,x,x)$ opposite to the vertex $A(\tau;0,0,0)$ is labelled by $C(x,x,x; \pm 1, \pm 2, \mp 3) = C(x,x,x \mp 1;\pm 1, \pm 2, \pm 3)$, which is also an edge of $\mathcal{D}^\pm(x,x,x \mp 1)$.
\end{itemize}

The first two rules determine that the local chart around $A_1(\tau)$ can be viewed as a sixth of a disk, which splits into $2$ sectors; see Figure \ref{fig:sector2-C}. The last two rules determines all $2$-simplicies that are attached to this local chart by edges, as well as both boundary edges of this local chart that also lie on the boundary of the dual complex $\Delta(K^2_{\mathcal{Y}/C})$; see the thick edges in Figure \ref{fig:sector2-C}.

\begin{figure}[h]
	\begin{tikzpicture}
		\draw (5,0) arc (0:60:5);
		\foreach \x in {0,...,2}
			\draw (30*\x:5cm) -- (0,0);
		\foreach \x in {0,...,2}
			\fill (30*\x:5cm) circle (4pt);
		\draw [ultra thick] (0:5cm) -- (0,0);
		\draw [ultra thick] (60:5cm) -- (0,0);
		\fill (0,0) circle (4pt);
		\draw (13:3.5cm) node {$\mathscr{D}^-(x,x,x)$};
		\draw (43:3.5cm) node {$\mathscr{D}^+(x,x,x)$};
		\draw (13:6.4cm) node {$\mathcal{D}^-(x,x,x+1)$};
		\draw (43:6.1cm) node {$\mathcal{D}^+(x,x,x-1)$};
		\draw (0,-0.5cm) node {$A_1(x,x,x)$};
	\end{tikzpicture}
    \caption{Local chart $\mathscr{C}(\tau)$ for $\tau_i$'s taking identical values; thick edges are on the boundary of $\Delta(K^2_{\mathcal{Y}/C})$}
    \label{fig:sector2-C}
\end{figure}

We summarize the above discussion by the following result

\begin{prop} \label{prop:local-charts-A}
	The local chart $\mathscr{C}(\tau)$ around the vertex $A_1(\tau)$ can be one of the three forms, as shown in Figures \ref{fig:fulldisk12-C}, \ref{fig:semidisk6-C} or \ref{fig:sector2-C}. Moreover, the thick edges in these figures exhibit the edges in $\mathscr{C}(\tau)$ that occur on the boundary of the dual complex $\Delta(K^2_{\mathcal{Y}/C})$. \qed
\end{prop}

\subsection{The local chart $\mathcal{C}^+(\tau)$ around $B_1^+(\tau)$} \label{subsec:local-chart-B111-C}

For the vertex $B_1^+(\tau) \coloneqq B(\tau;1,1,1)$, the admissibility condition \eqref{eqn:admissible-second} requires
\begin{equation} \label{eqn:adm-B+}
	\tau_1 + \tau_2 + \tau_3  \equiv -1 \pmod N.
\end{equation}
We denote the local chart around the vertex $B_1^+(\tau)$ as $\mathcal{C}^+(\tau)$. The following is an easy observation

\begin{lemma} \label{lem:local-around-B}
	The local chart around the vertex $B_1^+(\tau)$ can be given by
	\begin{equation} \label{eq:config2-C}
		\mathcal{C}^+(\tau) = \bigcup_{\zeta \in \mathcal{S}_3} \mathcal{D}^+(\zeta(\tau))
	\end{equation}
	as a union of $2$-simplices.
\end{lemma}

\begin{proof}
	The idea is similar to that of Lemma \ref{lem:local-around-A}. The only complete admissible line chart that contains $B(1,1,1)$ as a line subchart is $D(1,2,3)$. Moreover, specialization does not alter the value of $\tau_i$'s. Therefore the $2$-simplices which contain $B_1^+(\tau)$ are exactly given by $\mathcal{D}^+$ parametrized by any permutation of $\tau$. The claim follows.
\end{proof}

It is clear from Lemma \ref{lem:local-around-B} that the shape of $\mathcal{C}^+(\tau)$ depends on the number of distinct values among the components of $\tau$. Hence we still distinguish three cases.

\subsubsection{The set $\{ \tau_1, \tau_2, \tau_3 \}$ contains $3$ distinct values}

In such a case, the $2$-simplices on the right side of \eqref{eq:config2-C} are distinct, hence $\mathcal{C}^+(\tau)$ consists of $6$ distinct $2$-simplices. Let $(x,y,z)$ be any permutation of $(\tau_1, \tau_2, \tau_3)$. By manipulating the corresponding line charts, we observe the following rules concerning the relative positions of the $2$-simplicies
\begin{itemize}
	\item[(a)] the $2$-simplices $\mathcal{D}^+(x,y,z)$ and $\mathcal{D}^+(y,x,z)$ share a common edge labelled by $C(x,y,z;1,1,2)$;
	\item[(b)] the $2$-simplices $\mathcal{D}^+(x,y,z)$ and $\mathcal{D}^+(x,z,y)$ share a common edge labelled by $C(x,y,z;1,2,2)$;
	\item[(c)] the edge of $\mathcal{D}^+(x,y,z)$ opposite to the vertex $B_1^+(\tau)$ is labelled by $C(x,y,z;1,2,3)=C(x,y,z+1;1,2,-3)$, which is also an edge of $\mathscr{D}^+(x,y,z+1)$.
\end{itemize}

By the first two rules, also for the convenience later, we can view the local chart around $B_1^+(\tau)$ as an equilateral triangle with curved edges, which splits into $6$ sectors; see Figure \ref{fig:fulldisk6-C}. The last rule determines all $2$-simplicies that are attached to this local chart by edges. In particular, none of the boundary edges of this local chart is on the boundary of the dual complex $\Delta(K^2_{\mathcal{Y}/C})$.

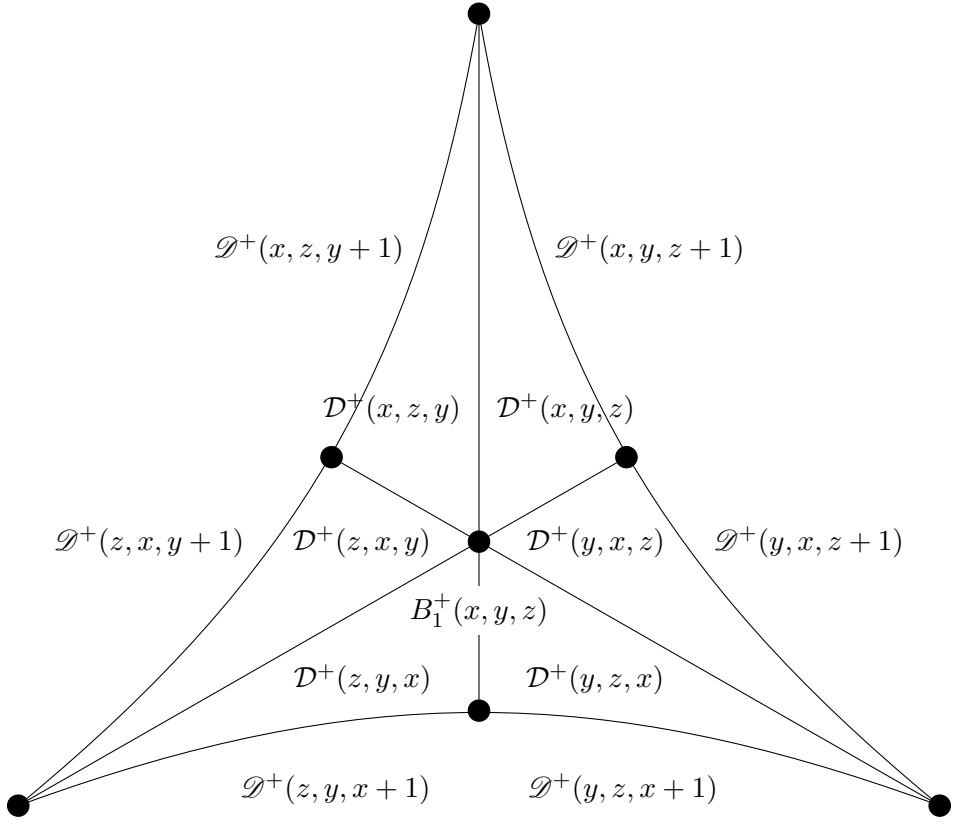
\begin{figure}[h]
	\begin{tikzpicture}[inner sep=1mm, scale=1.4]
		\node[circle,draw,fill] (a) at (90:5) {};
		\node[circle,draw,fill] (b) at (210:5) {};
		\node[circle,draw,fill] (c) at (330:5) {};
		\node[circle,draw,fill] (o) at (0,0) {};
		\draw (a) to [bend left=20] (b);
		\draw (b) to [bend left=20] (c);
		\draw (c) to [bend left=20] (a);
		\node[circle,draw,fill] (ab) at (150:1.6) {};
		\node[circle,draw,fill] (bc) at (270:1.6) {};
		\node[circle,draw,fill] (ca) at (30:1.6) {};
		\draw (a) to (bc);
		\draw (b) to (ca);
		\draw (c) to (ab);
		\draw (0,-0.65cm) node [fill=white] {$B_1^+(x,y,z)$};
		\draw (0:1.1cm) node {$\mathcal{D}^+(y,x,z)$};
		\draw (57:1.5cm) node {$\mathcal{D}^+(x,y,z)$};
		\draw (123:1.5cm) node {$\mathcal{D}^+(x,z,y)$};
		\draw (180:1.1cm) node {$\mathcal{D}^+(z,x,y)$};
		\draw (230:1.7cm) node {$\mathcal{D}^+(z,y,x)$};
		\draw (310:1.7cm) node {$\mathcal{D}^+(y,z,x)$};
		\draw (0:3.1cm) node {$\mathscr{D}^+(y,x,z+1)$};
		\draw (60:3.2cm) node {$\mathscr{D}^+(x,y,z+1)$};
		\draw (120:3.2cm) node {$\mathscr{D}^+(x,z,y+1)$};
		\draw (180:3.1cm) node {$\mathscr{D}^+(z,x,y+1)$};
		\draw (240:2.7cm) node {$\mathscr{D}^+(z,y,x+1)$};
		\draw (300:2.7cm) node {$\mathscr{D}^+(y,z,x+1)$};
	\end{tikzpicture}
	\caption{Local chart $\mathcal{C}^+(\tau)$ for pairwise distinct $\tau_i$'s}
    \label{fig:fulldisk6-C}
\end{figure}

\subsubsection{The set $\{ \tau_1, \tau_2, \tau_3 \}$ contains $2$ distinct values}

In such a case, the right side of \eqref{eq:config2-C} consists of $3$ distinct $2$-simplices forming $\mathcal{C}^+(\tau)$. The rules concerning the relative positions of $2$-simplices are similar to those in the above case, namely
\begin{itemize}
	\item[(a1)] when $x \neq y$, the $2$-simplices $\mathcal{D}^+(x,y,z)$ and $\mathcal{D}^+(y,x,z)$ share a common edge labelled by $C(x,y,z;1,1,2)$;
	\item[(a2)] when $x = y$, the $2$-simplex $\mathcal{D}^+(x,y,z)$ is the only one that contains the edge labelled by $C(x,y,z;1,1,2)$;
	\item[(b1)] when $y \neq z$, the $2$-simplices $\mathcal{D}^+(x,y,z)$ and $\mathcal{D}^+(x,z,y)$ share a common edge labelled by $C(x,y,z;1,2,2)$;
	\item[(b2)] when $y = z$, the $2$-simplex $\mathscr{D}^+(x,y,z)$ is the only one that contains the edge labelled by $C(x,y,z;1,2,2)$;
	\item[(c)] the edge of $\mathcal{D}^+(x,y,z)$ opposite to the vertex $B_1^+(\tau)$ is labelled by $C(x,y,z;1,2,3) = C(x,y,z+1;1,2,-3)$, which is also an edge of $\mathscr{D}^+(x,y,z+1)$.
\end{itemize}

By the rules (a1), (b1) and (c), also for the convenience later, we can view the local chart around $B_1^+(\tau)$ as half of an equilateral triangle with curved edges, which splits into $3$ sectors; see Figure \ref{fig:semidisk3-C}. The rules (a2) and (b2) determine all $2$-simplicies that are attached to this local chart by edges, as well as both boundary edges of this local chart that also lie on the boundary of the dual complex $\Delta(K^2_{\mathcal{Y}/C})$; see the thick edges in Figure \ref{fig:semidisk6-C}.

\begin{figure}[h]
	\begin{tikzpicture}[inner sep=1mm, scale=1.4]
		\node[circle,draw,fill] (a) at (0:5) {};
		\node[circle,draw,fill] (b) at (180:1.6) {};
		\node[circle,draw,fill] (c) at (120:5) {};
		\node[circle,draw,fill] (o) at (0,0) {};
		\draw (b) to [bend right=10] (c);
		\draw (c) to [bend right=20] (a);
		\node[circle,draw,fill] (ca) at (60:1.6) {};
		\draw [ultra thick] (a) to (b);
		\draw (o) to (ca);
		\draw (c) to (o);
		\draw (0,-0.4cm) node [fill=white] {$B_1^+(x,y,y)$};
		\draw (90:1.3cm) node {$\mathcal{D}^+(y,x,y)$};
		\draw (20:1.3cm) node {$\mathcal{D}^+(x,y,y)$};
		\draw (155:1.2cm) node {$\mathcal{D}^+(y,y,x)$};
		\draw (90:2.9cm) node {$\mathscr{D}^+(y,x,y+1)$};
		\draw (20:3.1cm) node {$\mathscr{D}^+(x,y,y+1)$};
		\draw (160:3cm) node {$\mathscr{D}^+(y,y,x+1)$};
	\end{tikzpicture}
	\caption{Local chart $\mathcal{C}^+(\tau)$ for $\tau_i$'s taking two distinct values; thick edges are on the boundary of $\Delta(K^2_{\mathcal{Y}/C})$.}
    \label{fig:semidisk3-C}
\end{figure}
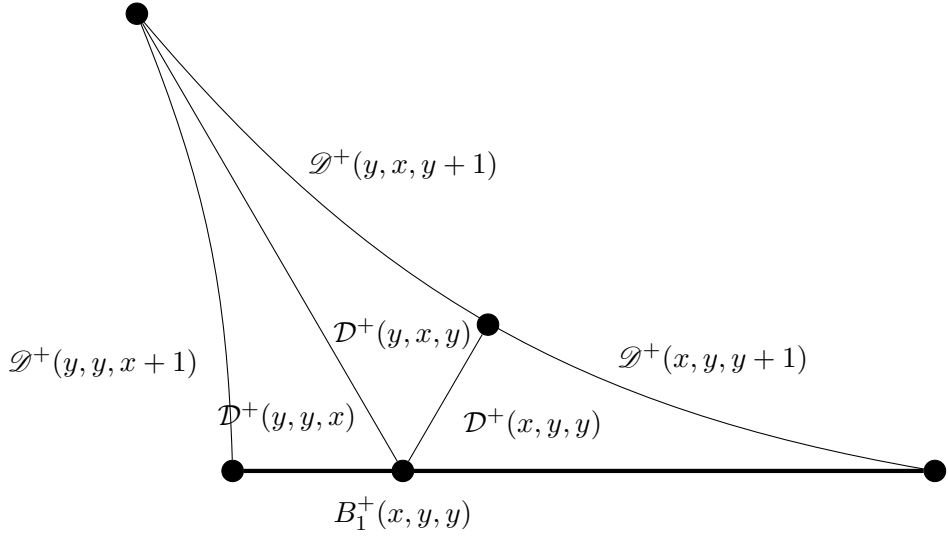

\subsubsection{The set $\{ \tau_1, \tau_2, \tau_3 \}$ contains $1$ distinct value}

In such a case, the right side of \eqref{eq:config2-C} consists of only $1$ distinct $2$-simplices forming $\mathcal{C}^+(\tau)$. The rules concerning the relative positions of $2$-simplices are similar to those in both above cases, namely
\begin{itemize}
	\item[(ab)] the $2$-simplex $\mathcal{D}^+(x,x,x)$ is the only one that contains the edge labelled by $C(x,y,z;1,1,2)$ or $C(x,y,z;1,2,2)$;
	\item[(c)] the edge of $\mathcal{D}^+(x,x,x)$ opposite to the vertex $B_1^+(\tau)$ is labelled by $C(x,x,x;1,2,3) = C(x,x,x+1;1,2,-3)$, which is also an edge of $\mathscr{D}^+(x,x,x+1)$.
\end{itemize}

The first rule determines that the local chart around $B_1^+(\tau)$ can be viewed as a sector of the above mentioned equilateral triangle; see Figure \ref{fig:sector1-C}. The second rule determines the only $2$-simplex that is attached to this local chart by edges, as well as both boundary edges of this local chart that also lie on the boundary of the dual complex $\Delta(K^2_{\mathcal{Y}/C})$; see the thick edges in Figure \ref{fig:sector1-C}.

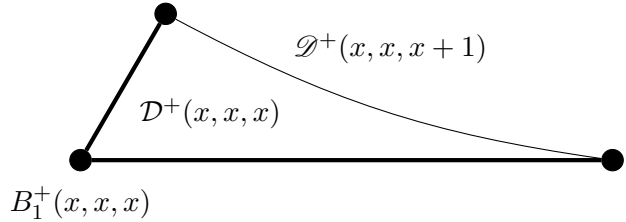
\begin{figure}[h]
	\begin{tikzpicture}[inner sep=1mm, scale=1.4]
		\node[circle,draw,fill] (a) at (0:5) {};
		\node[circle,draw,fill] (ca) at (60:1.6) {};
		\node[circle,draw,fill] (o) at (0,0) {};
		\draw (ca) to [bend right=10] (a);
		\draw [ultra thick] (a) to (o);
		\draw [ultra thick] (o) to (ca);
		\draw (0,-0.4cm) node [fill=white] {$B_1^+(x,x,x)$};
		\draw (20:1.3cm) node {$\mathcal{D}^+(x,x,x)$};
		\draw (20:3.1cm) node {$\mathscr{D}^+(x,x,x+1)$};
	\end{tikzpicture}
    \caption{Local chart $\mathcal{C}^+(\tau)$ for $\tau_i$'s taking identical values; thick edges are on the boundary of $\Delta(K^2_{\mathcal{Y}/C})$.}
    \label{fig:sector1-C}
\end{figure}

We summarize the above discussion by the following result

\begin{prop} \label{prop:local-charts-B}
	The local chart $\mathcal{C}^+(\tau)$ around the vertex $B_1^+(\tau)$ can be one of the three forms, as shown in Figures \ref{fig:fulldisk6-C}, \ref{fig:semidisk3-C} or \ref{fig:sector1-C}. Moreover, the thick edges in these figures exhibit the edges in $\mathcal{C}^+(\tau)$ that occur on the boundary of the dual complex $\Delta(K^2_{\mathcal{Y}/C})$. \qed
\end{prop}

\subsection{The local chart $\mathcal{C}^-(\tau)$ around $B_1^-(\tau)$}

This situation is completely parallel to the one in Subsection \ref{subsec:local-chart-B111-C}. For the vertex $ B_1^-(\tau) \coloneqq B(\tau;-1,-1,-1)$, the admissibility condition \eqref{eqn:admissible-second} requires
\begin{equation} \label{eqn:adm-B-}
	\tau_1 + \tau_2 + \tau_3  \equiv 1 \pmod N.
\end{equation}
We denote the local chart around the vertex $B_1^-(\tau)$ as $\mathcal{C}^-(\tau)$. Then we can similarly obtain both the following results

\begin{lemma} \label{lem:local-around-C}
	The local chart around the vertex $B_1^-(\tau)$ can be given by
	\begin{equation} \label{eq:config3-C}
		\mathcal{C}^-(\tau) = \bigcup_{\zeta \in \mathcal{S}_3} \mathcal{D}^-(\zeta(\tau))
	\end{equation}
	as a union of $2$-simplicies. \qed
\end{lemma}

\begin{prop}
	The local chart $\mathcal{C}^-(\tau)$ around the vertex $B_1^-(\tau)$ can be one of the three forms, as shown in Figures \ref{fig:fulldisk6-C}, \ref{fig:semidisk3-C} or \ref{fig:sector1-C}, with $\mathcal{D}^+$ replaced by $\mathcal{D}^-$, and $\mathscr{D}^+$ replaced by $\mathscr{D}^-$. Moreover, the thick edges in these figures exhibit the edges in $\mathcal{C}^-(\tau)$ that occur on the boundary of the dual complex $\Delta(K^2_{\mathcal{Y}/C})$. \qed
\end{prop}

\subsection{Glueing local charts} \label{subsec:glue-local-charts}

Based on the above discussion, we will glue the local charts to obtain the full dual complex $\Delta(K^2_{\mathcal{Y}/C})$.

\subsubsection{General principle} \label{subsubsec:general-principle}

In the above discussion we have completely classified all local charts around vertices of Type 4, as well as the $2$-simplicies that share common edges with them. It follows that we can easily find all neighboring local charts of any particular one.

By Lemma \ref{lem:glue-type-4}, the full dual complex $\Delta(K^2_{\mathcal{Y}/C})$ is obtained by glueing these local charts along their boundary edges. Therefore, we can start from any fixed local chart, and recursively find all neighboring local charts to expand our knowledge about the dual complex, until we reach the boundary. In this way we obtain a connected component of the dual complex.

In particular, if we start from a local chart of type $\mathscr{C}$ in any of the three cases discussed in Subsection \ref{subsec:3-charts-A}, the following Figures \ref{fig:neighbor-scr1}, \ref{fig:neighbor-scr2} and \ref{fig:neighbor-scr3} show how to find its neighboring local charts (which are necessarily of type $\mathcal{C}^\pm$), and further the neighboring local charts of them (which are necessarily of type $\mathscr{C}$ again). For simplicity, we only draw the center and the boundary of each local chart; other vertices and edges are all omitted.

\begin{figure}[h]
	\begin{tikzpicture}[scale=0.8]
		\draw (0,0) circle (3);
		\foreach \x in {0,...,5}
			\draw (60*\x-30:6)+(60*\x+90:3) arc (60*\x+90:60*\x+210:3);
		\fill (0,0) circle (4pt);
		\foreach \x in {0,...,5}
			\fill (60*\x+30:3)+(60*\x-60:{sqrt(3)}) circle (4pt);
		\foreach \x in {0,...,5}
			\fill (60*\x+30:6) circle (4pt);
		\draw (0,-0.5) node {$A_1(x,y,z)$};
		\draw (0:{2*sqrt(3)})+(0,-0.5) node {$B_1^-(x+1,y,z)$};
		\draw (60:{2*sqrt(3)})+(0,-0.5) node {$B_1^+(x,y-1,z)$};
		\draw (120:{2*sqrt(3)})+(0,-0.5) node {$B_1^-(x,y,z+1)$};
		\draw (180:{2*sqrt(3)})+(0,-0.5) node {$B_1^+(x-1,y,z)$};
		\draw (240:{2*sqrt(3)})+(0,-0.5) node {$B_1^-(x,y+1,z)$};
		\draw (300:{2*sqrt(3)})+(0,-0.5) node {$B_1^+(x,y,z-1)$};
		\draw (30:6)+(0,0.5) node {$A_1(x+1,y-1,z)$};
		\draw (90:6)+(0,-0.5) node {$A_1(x,y-1,z+1)$};
		\draw (150:6)+(0,0.5) node {$A_1(x-1,y,z+1)$};
		\draw (210:6)+(0,0.5) node {$A_1(x-1,y+1,z)$};
		\draw (270:6)+(0,0.5) node {$A_1(x,y+1,z-1)$};
		\draw (330:6)+(0,0.5) node {$A_1(x+1,y,z-1)$};
	\end{tikzpicture}
	\caption{Neighboring local charts of $\mathscr{C}(\tau)$ in Subsection \ref{subsubsec:3-values-C}. }
	\label{fig:neighbor-scr1}
\end{figure}
\begin{figure}[h]
	\begin{tikzpicture}[scale=0.8]
		\draw (3,0) arc (0:180:3);
		\foreach \x in {1,...,3}
			\draw (60*\x-30:6)+(60*\x+90:3) arc (60*\x+90:60*\x+210:3);
		\fill (0,0) circle (4pt);
		\foreach \x in {0,...,3}
			\fill (60*\x+30:3)+(60*\x-60:{sqrt(3)}) circle (4pt);
		\foreach \x in {0,...,2}
			\fill (60*\x+30:6) circle (4pt);
		\draw [ultra thick] ({-3*sqrt(3)},0) -- ({3*sqrt(3)},0);
		\draw (0,-0.5) node {$A_1(x,y,y)$};
		\draw (0:{2*sqrt(3)})+(0,-0.5) node {$B_1^-(x+1,y,y)$};
		\draw (60:{2*sqrt(3)})+(0,-0.5) node {$B_1^+(x,y,y-1)$};
		\draw (120:{2*sqrt(3)})+(0,-0.5) node {$B_1^-(x,y,y+1)$};
		\draw (180:{2*sqrt(3)})+(0,-0.5) node {$B_1^+(x-1,y,y)$};
		\draw (30:6)+(0,0.5) node {$A_1(x+1,y,y-1)$};
		\draw (90:6)+(0,-0.5) node {$A_1(x,y+1,y-1)$};
		\draw (150:6)+(0,0.5) node {$A_1(x-1,y,y+1)$};
	\end{tikzpicture}
	\caption{Neighboring local charts of $\mathscr{C}(\tau)$ in Subsection \ref{subsubsec:2-values-C}. }
	\label{fig:neighbor-scr2}
\end{figure}
\begin{figure}[h]
	\begin{tikzpicture}[scale=0.8]
		\draw (3,0) arc (0:60:3);
		\draw (60-30:6)+(60+90:3) arc (60+90:60+210:3);
		\fill (0,0) circle (4pt);
		\fill (30:6) circle (4pt);
		\foreach \x in {0,1}
			\fill (60*\x+30:3)+(60*\x-60:{sqrt(3)}) circle (4pt);
		\draw [ultra thick] (0,0) -- ({3*sqrt(3)},0);
		\draw [ultra thick] (0,0) -- (60:{3*sqrt(3)});
		\draw (0,-0.5) node {$A_1(x,x,x)$};
		\draw (0:{2*sqrt(3)})+(0,-0.5) node {$B_1^-(x,x,x+1)$};
		\draw (60:{2*sqrt(3)})+(-1.8,0) node {$B_1^+(x,x,x-1)$};
		\draw (30:6)+(0,-0.5) node {$A_1(x,x+1,x-1)$};
	\end{tikzpicture}
	\caption{Neighboring local charts of $\mathscr{C}(\tau)$ in Subsection \ref{subsubsec:1-values-C}. }
	\label{fig:neighbor-scr3}
\end{figure}

It is helpful to observe that the centers of all local charts are located on an equilateral triangle lattice. Therefore in the following discussion, we will further omit the edges and only draw the centers of local charts.

\subsubsection{Algorithm for building the dual complex} \label{subsubsec:algorithm}

In practice, with the above principle, we can run the following algorithm to find one connected component of the dual complex

\bigskip

\begin{quotation}

\quad \textbf{Initialization.} The local chart $\mathscr{C}(0,0,0)$ around the vertex $A_1(0,0,0)$ exists in the dual complex $\Delta(K^2_{\mathcal{Y}/C})$ for any value of $N$. We view it as a partial dual complex, denoted by $\Delta_0(K^2_{\mathcal{Y}/C})$.

\textbf{Step $k$ ($k \geq 1$).} Assume that the partial dual complex after Step ($k-1$) is denoted by $\Delta_{k-1}(K^2_{\mathcal{Y}/C})$. We attach to it all its immediate neighboring local charts, and denote this expansion by $\Delta_k(K^2_{\mathcal{Y}/C})$.

\textbf{Termination.} The algorithm terminates until no more local chart can be further attached to the existing partial dual complex.

\end{quotation}

\bigskip

The above algorithm determines a connected component of the dual complex in a recursive manner. 

\subsubsection{Construction of the dual complex}

First of all, we give the reader an idea of the component of the dual complex constructed by the above algorithm. As an example, the following Figure \ref{fig:centers-4-levels} shows the centers of the local charts in $\Delta_4(K^2_{\mathcal{Y}/C})$ for $N>4$, from which we can observe the following pattern 
\begin{itemize}
	\item The centers of all local charts are precisely the lattice points located in a sector of an equilateral triangle lattice;
	\item After \textbf{Step $k$ ($k \geq 1$)} of the algorithm, we obtain the centers of all local charts in $\Delta_k(K^2_{\mathcal{Y}/C})$, which are precisely given by all lattice points in an equilateral triangle, whose sides contain exactly $k$ edges;
	\item Each step of the algorithm adds a new row at the bottom of the picture, hence enlarge the size of the equilateral triangle by $1$ edge on each side.
\end{itemize}

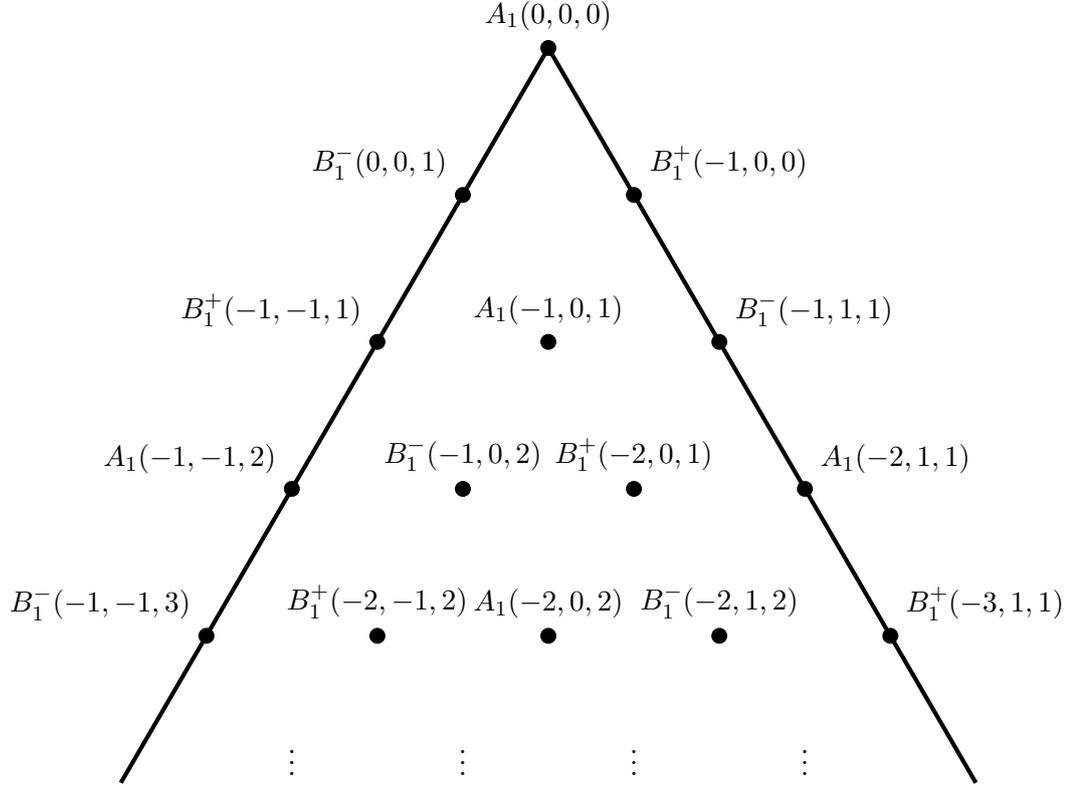
\begin{figure}[h]
	\begin{tikzpicture}[scale=0.75]
		\fill (0:0) circle (4pt) node[above=0.1] {$A_1(0,0,0)$};
		\fill (240:3) circle (4pt) node[above left=0.1] {$B_1^{-}(0,0,1)$};
		\fill (240:3)+(0:3) circle (4pt) node[above right=0.1] {$B_1^{+}(-1,0,0)$};
		\fill (240:6) circle (4pt) node[above left=0.1] {$B_1^{+}(-1,-1,1)$};
		\fill (240:6)+(0:3) circle (4pt) node[above=0.1] {$A_1(-1,0,1)$};
		\fill (240:6)+(0:6) circle (4pt) node[above right=0.1] {$B_1^{-}(-1,1,1)$};
		\fill (240:9) circle (4pt) node[above left=0.1] {$A_1(-1,-1,2)$};
		\fill (240:9)+(0:3) circle (4pt) node[above=0.1] {$B_1^-(-1,0,2)$};
		\fill (240:9)+(0:6) circle (4pt) node[above=0.1] {$B_1^+(-2,0,1)$};
		\fill (240:9)+(0:9) circle (4pt) node[above right=0.1] {$A_1(-2,1,1)$};
		\fill (240:12) circle (4pt) node[above left=0.1] {$B_1^-(-1,-1,3)$};
		\fill (240:12)+(0:3) circle (4pt) node[above=0.1] {$B_1^+(-2,-1,2)$};
		\fill (240:12)+(0:6) circle (4pt) node[above=0.1] {$A_1(-2,0,2)$};
		\fill (240:12)+(0:9) circle (4pt) node[above=0.1] {$B_1^-(-2,1,2)$};
		\fill (240:12)+(0:12) circle (4pt) node[above right=0.1] {$B_1^+(-3,1,1)$};
		\draw[ultra thick] (240:15) -- (0:0) -- (300:15);
		\node at ($(240:15)+(0:3)+(0,0.5)$) {$\vdots$};
		\node at ($(240:15)+(0:6)+(0,0.5)$) {$\vdots$};
		\node at ($(240:15)+(0:9)+(0,0.5)$) {$\vdots$};
		\node at ($(240:15)+(0:12)+(0,0.5)$) {$\vdots$};
	\end{tikzpicture}
	\caption{Centers of local charts in $\Delta(K^2_{\mathcal{Y}/C})$ for large $N$}
	\label{fig:centers-4-levels}
\end{figure}

Indeed, we can give a precise description of the dual complex by specifying the new row of centers added in \textbf{Step $k$} for each value of $k$. 

\begin{prop} \label{prop:dual-complex-n3}
	Following the algorithm in Subsection \ref{subsubsec:algorithm}, \textbf{Step $k$} adds a new row of centers of local charts at the bottom of Figure \ref{fig:centers-4-levels}, which are determined by both of the following rules
	\begin{itemize}
		\item The row starts with $B_1^+(-q,-q,2q-1)$ if $k=3q-1$, $A_1(-q,-q,2q)$ if $k=3q$, and $B_1^-(-q,-q,2q+1)$ if $k=3q+1$;
		\item Assuming $a \leq b \leq c$, the one on the right side of $B_1^+(a,b,c)$ is $A_1(a,b+1,c)$; the one on the right side of $A_1(a,b,c)$ is $B_1^-(a,b+1,c)$; and the one on the right side of $B_1^-(a,b,c)$ is $B_1^+(a-1,b,c-1)$;
		\item The row ends with $B_1^-(-2q+1,q,q)$ if $k=3q-1$, $A_1(-2q,q,q)$ if $k=3q$, and $B_1^+(-2q-1,q,q)$ if $k=3q+1$;
		\item The new row adds exactly $k+1$ centers of local charts to the partial dual complex.
	\end{itemize}
	Moreover, the algorithm terminates precisely after \textbf{Step $N$}, and the connected component $\Delta_N(K^2_{\mathcal{Y}/C})$ that we obtain is the full dual complex $\Delta(K^2_{\mathcal{Y}/C})$. 
\end{prop}

\begin{proof}
	The statement about the new row added in Figure \ref{fig:centers-4-levels} in \textbf{Step $k$} can be checked inductively, following the pattern exhibited in Figures \ref{fig:neighbor-scr1}, \ref{fig:neighbor-scr2}, \ref{fig:neighbor-scr3}. This is a tedious but straightforward calculation, hence we omit the details.
	
	For the statement about the termination of the algorithm, we need to find out all local charts that lie on the boundary of $\Delta(K^2_{\mathcal{Y}/C})$. For the center of any local chart $A_1(\tau)$ or $B_1^\pm(\tau)$ where $\tau=(\tau_1, \tau_2, \tau_3)$ described above, we can prove by induction that $\tau_1=\tau_2$ if and only if the center is the first one in a row, and $\tau_2=\tau_3$ if and only if the center is the last one in a row. Moreover, we can also prove by induction that $\tau_3 - \tau_1 = k$ for all centers in the row added in \textbf{Step $k$}. When we view the components of $\tau$ as residue classes modulo $N$, it is clear that the entire row added in \textbf{Step $N$} is on the boundary of the dual complex $\Delta(K^2_{\mathcal{Y}/C})$. Therefore the algorithm terminates after \textbf{Step $N$}, and $\Delta_N(K^2_{\mathcal{Y}/C})$ becomes a connected component of the full dual complex $\Delta(K^2_{\mathcal{Y}/C})$.
	
	Finally we show that $\Delta(K^2_{\mathcal{Y}/C})$ does not contain any other component. By Proposition \ref{prop:stratadeterm} we have seen that $\Delta(K^2_{\mathcal{Y}/C})$ is given by the union of all its $2$-simplicies. Moreover by Lemma \ref{lem:glue-type-4} we have also seen that each $2$-simplex contains a vertex of Type 4. Hence each connected component $\Delta(K^2_{\mathcal{Y}/C})$ must have vertices of Type 4. However, \eqref{eqn:sum-of-three} in the proof of Proposition \ref{prop:number} shows that $\Delta(K^2_{\mathcal{Y}/C})$ contains exactly $\frac{3}{N}\binom{N+2}{3}$ vertices of Type 4. And the above construction shows that the number of vertices of Type 4 in $\Delta_N(K^2_{\mathcal{Y}/C})$ is given by
	$$ 1 + 2 + \dots + (N+1) = \binom{N+2}{2} = \frac{3}{N}\binom{N+2}{3}. $$
	Therefore $\Delta_N(K^2_{\mathcal{Y}/C})$ already contains all vertices of Type 4 in $\Delta(K^2_{\mathcal{Y}/C})$, which must be the unique component of it.
\end{proof}

The following result can be viewed as a summary of the above construction

\begin{theorem} \label{thm:dual-complex-n3}
	The dual complex $\Delta(K^2_{\mathcal{Y}/C})$ is PL-homeomorphic to the standard $2$-simplex. In particular, it is  contractible and of dimension $2$.
\end{theorem}

\begin{proof}
	Indeed, Proposition \ref{prop:dual-complex-n3} shows inductively that $\Delta(K^2_{\mathcal{Y}/C})$ can be viewed as a triangulation of an equilateral triangle, consisting of $N$ horizontal strips of $2$-simplicies. Whenever $N$ is increased by $1$, an extra strip has to be attached to the bottom of the dual complex. In other words, if we extend the complex in Figure \ref{fig:centers-4-levels} downwards infinitely, then $\Delta(K^2_{\mathcal{Y}/C})$ is given by the top $N$ horizontal strips. Hence this statement follows immediately from Proposition \ref{prop:dual-complex-n3}.
\end{proof}

The statement in Theorem \ref{thm:dual-complex-n3} is not a coincidence. In fact, a result of Brown-Mazzon proved in \cite{BrownMazzon} says that the essential skeleton of the associated Berkovich analytic space of the generic fiber is PL homeomorphic to the standard  $(n-1)$-simplex in general. Our computation of the dual complex can be viewed as an explicit realization of this general result in a very explicit example.

Moreover, by the construction in Proposition \ref{prop:dual-complex-n3}, we can see that the pattern of the triangulation of $\Delta(K^2_{\mathcal{Y}/C})$ differ slightly depending on the the residue of $N$ modulo $3$. In particular, the dual complex admits an $\mathcal{S}_3$-symmetry if $N$ is divisible by $3$; while only an $\mathcal{S}_2$-symmetry if $N$ is not divisible by $3$. The following Figures \ref{fig:dual-4}, \ref{fig:dual-5} and \ref{fig:dual-6} illustrate the dual complex $\Delta(K^2_{\mathcal{Y}/C})$ in the case of $N=4$, $5$ and $6$ respectively.

\begin{figure}[h]
	\begin{tikzpicture}[scale=1.8]
		\fill (0,0) circle (1pt);
		\foreach \x in {0,1}
			\fill ($(240:{2/sqrt(3)})+({2/sqrt(3)*\x},0)$) circle (1pt);
		\foreach \x in {0,1,2}
			\fill ($(240:{4/sqrt(3)})+({2/sqrt(3)*\x},0)$) circle (1pt);
		\foreach \x in {0,1,2,3}
			\fill ($(240:{6/sqrt(3)})+({2/sqrt(3)*\x},0)$) circle (1pt);
		\foreach \x in {0,1,2,3,4}
			\fill ($(240:{8/sqrt(3)})+({2/sqrt(3)*\x},0)$) circle (1pt);
		
		\foreach \x in {0,1,2}
			\fill ($(0:0)+(240+30*\x:1)$) circle (1pt);
		\foreach \x in {0,...,11}
			\fill ($(270:2)+(30*\x:1)$) circle (1pt);
		\foreach \x in {0,...,6}
			\fill ($(240:{6/sqrt(3)})+(240+30*\x:1)$) circle (1pt);
		\foreach \x in {0,...,6}
			\fill ($(300:{6/sqrt(3)})+(120+30*\x:1)$) circle (1pt);
		\foreach \x in {0,...,6}
			\fill ($(270:4)+(30*\x:1)$) circle (1pt);
			
		\draw (240:1) arc (240:300:1);
		\draw (270:2) circle (1);
		\draw ($(240:{6/sqrt(3)})+(240:1)$) arc (240:420:1);
		\draw ($(300:{6/sqrt(3)})+(120:1)$) arc (120:300:1);
		\draw ($(270:4)+(0:1)$) arc (0:180:1);
		
		\foreach \x in {1,...,4}
			\draw (240:{2/sqrt(3)*\x}) -- (300:{2/sqrt(3)*\x});
		\foreach \x in {0,...,3}
			\draw (240:{2/sqrt(3)*\x}) -- ($(300:{8/sqrt(3)})+(180:{2/sqrt(3)*\x})$);
		\foreach \x in {0,...,3}
			\draw (300:{2/sqrt(3)*\x}) -- ($(240:{8/sqrt(3)})+(0:{2/sqrt(3)*\x})$);
			
		\draw (0,0) -- ++ (270:4);
		\draw (240:{6/sqrt(3)}) -- ++(270:1);
		\draw (240:{6/sqrt(3)}) -- ++(30:3);
		\draw (300:{6/sqrt(3)}) -- ++(270:1);
		\draw (300:{6/sqrt(3)}) -- ++(150:3);
		\draw (270:4) -- ++(150:2);
		\draw (270:4) -- ++(30:2);
	\end{tikzpicture}
	\caption{Example: dual complex $\Delta(K^2_{\mathcal{Y}/C})$ for $N=4$}
	\label{fig:dual-4}
\end{figure}
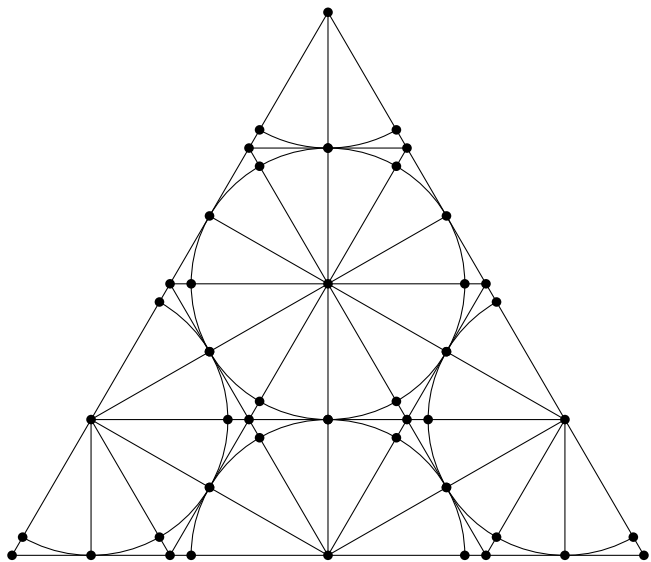

\begin{figure}[h]
	\begin{tikzpicture}[scale=1.8]
		\fill (0,0) circle (1pt);
		\foreach \x in {0,1}
			\fill ($(240:{2/sqrt(3)})+({2/sqrt(3)*\x},0)$) circle (1pt);
		\foreach \x in {0,1,2}
			\fill ($(240:{4/sqrt(3)})+({2/sqrt(3)*\x},0)$) circle (1pt);
		\foreach \x in {0,1,2,3}
			\fill ($(240:{6/sqrt(3)})+({2/sqrt(3)*\x},0)$) circle (1pt);
		\foreach \x in {0,1,2,3,4}
			\fill ($(240:{8/sqrt(3)})+({2/sqrt(3)*\x},0)$) circle (1pt);
		\foreach \x in {0,1,2,3,4,5}
			\fill ($(240:{10/sqrt(3)})+({2/sqrt(3)*\x},0)$) circle (1pt);
		
		\foreach \x in {0,1,2}
			\fill ($(0:0)+(240+30*\x:1)$) circle (1pt);
		\foreach \x in {0,...,11}
			\fill ($(270:2)+(30*\x:1)$) circle (1pt);
		\foreach \x in {0,...,6}
			\fill ($(240:{6/sqrt(3)})+(240+30*\x:1)$) circle (1pt);
		\foreach \x in {0,...,6}
			\fill ($(300:{6/sqrt(3)})+(120+30*\x:1)$) circle (1pt);
		\foreach \x in {0,...,11}
			\fill ($(270:4)+(30*\x:1)$) circle (1pt);
		\foreach \x in {0,...,6}
			\fill ($(240:{6/sqrt(3)})+(270:2)+(30*\x:1)$) circle (1pt);
		\foreach \x in {0,...,6}
			\fill ($(300:{6/sqrt(3)})+(270:2)+(30*\x:1)$) circle (1pt);
			
		\draw (240:1) arc (240:300:1);
		\draw (270:2) circle (1);
		\draw ($(240:{6/sqrt(3)})+(240:1)$) arc (240:420:1);
		\draw ($(300:{6/sqrt(3)})+(120:1)$) arc (120:300:1);
		\draw (270:4) circle (1);
		\draw ($(240:{6/sqrt(3)})+(270:2)+(0:1)$) arc (0:180:1);
		\draw ($(300:{6/sqrt(3)})+(270:2)+(0:1)$) arc (0:180:1);

		\foreach \x in {1,...,5}
			\draw (240:{2/sqrt(3)*\x}) -- (300:{2/sqrt(3)*\x});
		\foreach \x in {0,...,4}
			\draw (240:{2/sqrt(3)*\x}) -- ($(300:{10/sqrt(3)})+(180:{2/sqrt(3)*\x})$);
		\foreach \x in {0,...,4}
			\draw (300:{2/sqrt(3)*\x}) -- ($(240:{10/sqrt(3)})+(0:{2/sqrt(3)*\x})$);
			
		\draw (0,0) -- ++ (270:5);
		\draw (240:{6/sqrt(3)}) -- ++(270:2);
		\draw (240:{6/sqrt(3)}) -- ++(30:3);
		\draw (300:{6/sqrt(3)}) -- ++(270:2);
		\draw (300:{6/sqrt(3)}) -- ++(150:3);
		\draw ($(240:{6/sqrt(3)})+(270:2)$) -- ++(150:1);
		\draw ($(240:{6/sqrt(3)})+(270:2)$) -- ++(30:4);
		\draw ($(300:{6/sqrt(3)})+(270:2)$) -- ++(150:4);
		\draw ($(300:{6/sqrt(3)})+(270:2)$) -- ++(30:1);
	\end{tikzpicture}
	\caption{Example: dual complex $\Delta(K^2_{\mathcal{Y}/C})$ for $N=5$}
	\label{fig:dual-5}
\end{figure}

\begin{figure}[h]
	\begin{tikzpicture}[scale=1.8]
		\fill (0,0) circle (1pt);
		\foreach \x in {0,1}
			\fill ($(240:{2/sqrt(3)})+({2/sqrt(3)*\x},0)$) circle (1pt);
		\foreach \x in {0,1,2}
			\fill ($(240:{4/sqrt(3)})+({2/sqrt(3)*\x},0)$) circle (1pt);
		\foreach \x in {0,1,2,3}
			\fill ($(240:{6/sqrt(3)})+({2/sqrt(3)*\x},0)$) circle (1pt);
		\foreach \x in {0,1,2,3,4}
			\fill ($(240:{8/sqrt(3)})+({2/sqrt(3)*\x},0)$) circle (1pt);
		\foreach \x in {0,1,2,3,4,5}
			\fill ($(240:{10/sqrt(3)})+({2/sqrt(3)*\x},0)$) circle (1pt);
		\foreach \x in {0,1,2,3,4,5,6}
			\fill ($(240:{12/sqrt(3)})+({2/sqrt(3)*\x},0)$) circle (1pt);
		
		\foreach \x in {0,1,2}
			\fill ($(0:0)+(240+30*\x:1)$) circle (1pt);
		\foreach \x in {0,...,11}
			\fill ($(270:2)+(30*\x:1)$) circle (1pt);
		\foreach \x in {0,...,6}
			\fill ($(240:{6/sqrt(3)})+(240+30*\x:1)$) circle (1pt);
		\foreach \x in {0,...,6}
			\fill ($(300:{6/sqrt(3)})+(120+30*\x:1)$) circle (1pt);
		\foreach \x in {0,...,11}
			\fill ($(270:4)+(30*\x:1)$) circle (1pt);
		\foreach \x in {0,...,11}
			\fill ($(240:{6/sqrt(3)})+(270:2)+(30*\x:1)$) circle (1pt);
		\foreach \x in {0,...,11}
			\fill ($(300:{6/sqrt(3)})+(270:2)+(30*\x:1)$) circle (1pt);
		\foreach \x in {0,...,5}
			\fill ($(270:6)+(30*\x:1)$) circle (1pt);
		\foreach \x in {0,1,2}
			\fill ($(240:{12/sqrt(3)})+(30*\x:1)$) circle (1pt);
		\foreach \x in {0,1,2}
			\fill ($(300:{12/sqrt(3)})+(120+30*\x:1)$) circle (1pt);
			
		\draw (240:1) arc (240:300:1);
		\draw (270:2) circle (1);
		\draw ($(240:{6/sqrt(3)})+(240:1)$) arc (240:420:1);
		\draw ($(300:{6/sqrt(3)})+(120:1)$) arc (120:300:1);
		\draw (270:4) circle (1);
		\draw ($(240:{6/sqrt(3)})+(270:2)$) circle (1);
		\draw ($(300:{6/sqrt(3)})+(270:2)$) circle (1);
		\draw ($(270:6)+(0:1)$) arc (0:180:1);
		\draw ($(240:{12/sqrt(3)})+(0:1)$) arc (0:60:1);
		\draw ($(300:{12/sqrt(3)})+(120:1)$) arc (120:180:1);

		\foreach \x in {1,...,6}
			\draw (240:{2/sqrt(3)*\x}) -- (300:{2/sqrt(3)*\x});
		\foreach \x in {0,...,5}
			\draw (240:{2/sqrt(3)*\x}) -- ($(300:{12/sqrt(3)})+(180:{2/sqrt(3)*\x})$);
		\foreach \x in {0,...,5}
			\draw (300:{2/sqrt(3)*\x}) -- ($(240:{12/sqrt(3)})+(0:{2/sqrt(3)*\x})$);
			
		\draw (0,0) -- ++ (270:6);
		\draw (240:{6/sqrt(3)}) -- ++(270:3);
		\draw (240:{6/sqrt(3)}) -- ++(30:3);
		\draw (240:{6/sqrt(3)}) -- ++(330:6);
		\draw (300:{6/sqrt(3)}) -- ++(270:3);
		\draw (300:{6/sqrt(3)}) -- ++(150:3);
		\draw (300:{6/sqrt(3)}) -- ++(210:6);
		\draw (270:6) -- ++(150:3);
		\draw (270:6) -- ++(30:3);
	\end{tikzpicture}
	\caption{Example: dual complex $\Delta(K^2_{\mathcal{Y}/C})$ for $N=6$}
	\label{fig:dual-6}
\end{figure}
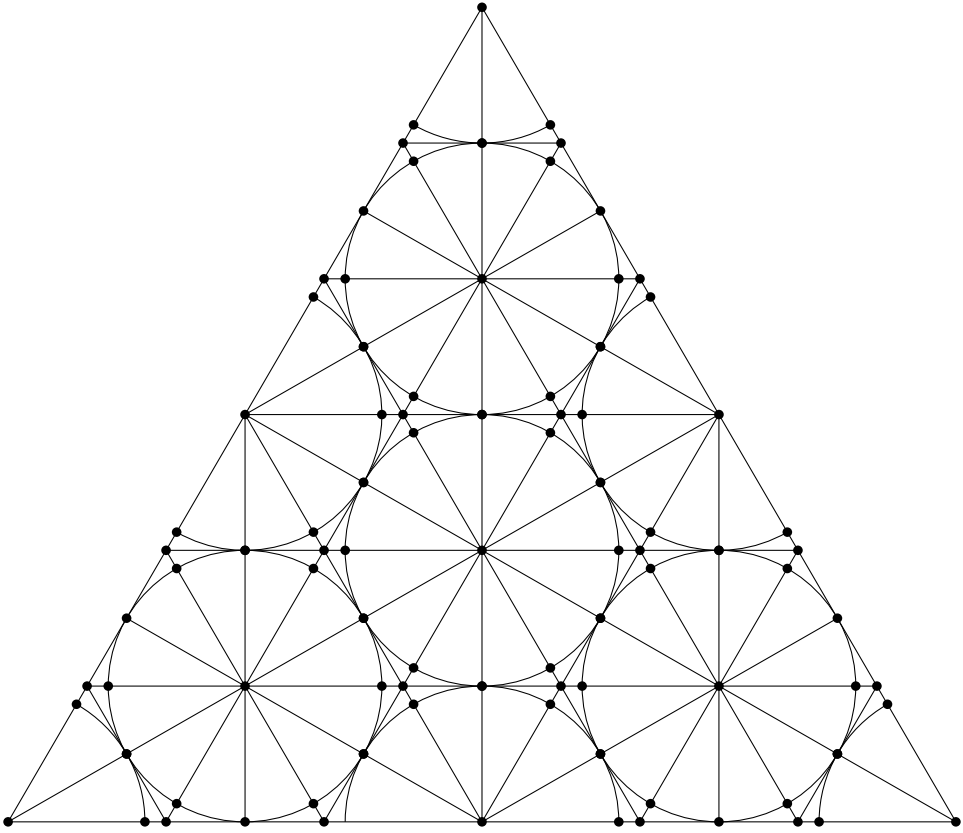

\section{Geometry of the strata}\label{sec:stratageom}

In this section, we discuss the geometry of the strata of the degenerate fiber of the GIT quotient $K^{n-1}_{\mathcal{Y}/C}$. We will primarily focus on the case $n=3$, but we also briefly mention the (much) simpler situation when $n=2$. The methods we use work in principle for arbitrary $n \geq 2$. 

\subsection{Preliminaries}
To begin with, we do not make any assumptions on $n$. The geometry of the strata depends on the distribution of the $n$ points on original components and inserted components, and how the torus $\mathbb G[n]$ acts on the latter. 

\subsubsection{}
Using the notation and results from Subsection \ref{sec:stratfinal}, we see that each stratum in the quotient can be written 
$$ \mathcal{K}(b,\mathbf{m})/\mathbb G[n] = O(b,\mathbf{m})/\mathbb G[b], $$ 
where $b \leq n$ and $\mathbf{m}$ is a numerical vector indexing the distribution of points in $\mathcal{Y}[b]_0$. Indeed, we have splittings $ \mathcal{K}(b,\mathbf{m}) = O(b,\mathbf{m}) \times \mathcal{U}(I)_0$ and $\mathbb G[n] = \mathbb G[b] \times \mathbb G[n-b]$, where $\mathcal{U}(I)_0$ is a $\mathbb G[n-b]$-torsor. We will use the notation $ A(x_1,x_2,x_3)$, $B(x_1,x_2,x_3)$, \emph{etc.}, introduced in Section \ref{subsec:prelsstratadescr} for the strata $\mathcal{K}(b,\mathbf{m})$ of the Kummer locus. Like in the previous section, we leave out the triple $\tau = (\tau_1, \tau_2,\tau_3)$ from the notation. 

\subsubsection{}
We represent by $\circ$ a point on an original component, and by $\bullet$ a point on an inserted component. We also need to specify the action of $\mathbb G[b]$ on inserted components. Let $\sigma_1, \ldots, \sigma_b$ be coordinates for $\mathbb G[b]$. Then the torus acts by $\pm \sigma_i$ on an inserted component, for some $i$. Lastly, we indicate by $\mathfrak w$ and $\mathfrak n$ whether a stratum is wide or narrow, respectively.

\subsubsection{} To give an example when $n=3$, the notation
$$ [(\circ, \bullet, \bullet), (-, \sigma_1, -\sigma_1), \mathfrak n]$$
means a narrow stratum for $b=1$, with one point on an original component, one point on an inserted component where the action is $P \mapsto P + \sigma_1$ and one inserted component where the action is $P \mapsto P - \sigma_1$. As another example, 
$$ [(\bullet, 2 \bullet), (\sigma_1, \sigma_2), \mathfrak w]$$
means a wide stratum for $b=2$, where one point is on an inserted component with action $+\sigma_1$ and $2$ points are located on the same inserted component where the action is by $+\sigma_2$. 

\subsubsection{} We denote for simplicity by $G^{\circ}$ the identity component of $\mathcal{A}_0 \cong \mathcal{Y}_0^{sm}$. Then $ O(b,\mathbf{m}) $ can be identified with the $\mathbb G[b]$-orbit of the kernel of the addition map
$$ \prod_l \mathrm{Hilb}^{m_l}(G^{\circ}) \to G^{\circ}, $$
where, for a given $l$, $\mathbb G[b]$ either acts trivially on $G^{\circ}$, or the action is identified with plus or minus the action of the rank $1$ subtorus $T \subset G^{\circ}$. 

\subsection{The case of three points}
We only treat the strata corresponding to cells in $\mathcal{D}^+$ and $\mathscr{D}^+$, strata corresponding to their negatives follow in the same way. 
\subsubsection{} First we discuss the $4$-dimensional strata without inserted components, which are obtained from $A(0,0,0)$, $A(0,1,-1)$ and $A(0,1,1)$ after taking the quotient. In this case we have three possibilities, depending on the  number of components on which the points lie, namely 3,2 or 1 different components, i.e. $(\circ,\circ,\circ)$, $(\circ,2\circ)$ or $(3\circ)$. There is no torus action and all strata are narrow. 

For $(\circ,\circ,\circ)$, the stratum is identified with the kernel of the  homomorphism 
\[
\mu: G^{\circ} \times G^{\circ} \times G^{\circ} \to G^{\circ},
\]
which is isomorphic to $G^{\circ} \times G^{\circ}$. For $(\circ,2\circ)$, the stratum is identified with the kernel of 
\[
\mu: G^{\circ} \times \mathrm{Hilb}^2(G^{\circ}) \to G^{\circ},
\]
which is isomorphic to $\mathrm{Hilb}^2(G^{\circ})$. For the case $(3\circ)$, the stratum is isomorphic to the kernel of the map
\[
\mu: \mathrm{Hilb}^3(G^{\circ}) \to G^{\circ}.
\]
This is the generalized Kummer variety for the semi-abelian surface $G^{\circ}$, i.e. $\mathrm{Kum}^2(G^{\circ})$. 

\subsubsection{} We next discuss the $4$-dimensional strata with inserted components.
 Here we start with $B(1,1,2)$. This can lead to two possibilities, namely $[(\circ,\bullet,\bullet),(-,\sigma_1,\sigma_1),\mathfrak w]$ and  $[(\circ,2\bullet),(-,\sigma_1,),\mathfrak w]$. In the first case we have to look at
 \[
\mu: G^{\circ} \times G^{\circ} \times G^{\circ} \to G^{\circ},
\]
and since we are in the wide case, the stratum is $O/\mathbb G[1] = \mu^{-1}(T)/\mathbb G[1]$, where $T \subset G^{\circ}$ is the rank $1$ subtorus. 
Throughout this section we denote by $O$ the inverse image of either the origin (narrow case) or of $T$ (wide case).
Note that the map $\mu$ is not equivariant with respect to the action of $\mathbb G[1]$ on the source and of $T$ on the target. 

The $\mathbb G[1]$-action has the form 
\[
(P,Q,R) \mapsto (P,Q+\sigma_1,R+\sigma_1)
\]  
where the latter triple is mapped to $P+Q+R+2\sigma_1$ under the addition map. This shows that every $\mathbb G[1]$-orbit of $O$ 
has a representative in $\mu^{-1}(e)$ where $e$ is the origin of $G^{\circ}$. 
The action of $\mathbb G[1]$ moreover maps fibers of the map $O \to T$ to fibers, hence two points in $\mu^{-1}(e)$ are identified if and only if they are equivalent under the stabilizer of $\mu^{-1}(e)$ in $\mathbb G[1]$. Clearly this is the $2$-torsion subgroup $\mu_2=\{\pm1\}$.

To conclude, we identify the kernel $\mu^{-1}(e)\cong G^{\circ} \times G^{\circ}$ with the last two components of $(G^{\circ})^3$, and claim that 
\[ O/\mathbb G[1] \cong G^{\circ} \times G^{\circ}/ \mu_2\]
where $\mu_2$ acts diagonally. Indeed,  The map $G^{\circ}\times G^{\circ} \to O/\mathbb G[1]$ mapping $(P,Q)$ to the class of $(-P-Q,P,Q)$ is surjective. The above discussion shows that, after dividing $G^{\circ} \times G^{\circ}$ by $\mu_2$ this becomes a bijective morphism. As both varieties are normal, this is an isomorphism. We will use similar arguments in other cases below without spelling them out each time. 

The case $[(\circ,2\bullet),(-,\sigma_1), \mathfrak w]$ is similar and an analogous discussion gives 
\[
O/\mathbb G[1] \cong \mathrm{Hilb}^2(G^{\circ}) / \mu_2.
\] 
 
 A further possibility is $B(1,1,1)$, where we again have two possibilities. The first is $[(\bullet,\bullet,\bullet),(\sigma_1,\sigma_1,\sigma_1),\mathfrak w]$. Arguing as above, we find that 
 \[
O/\mathbb G[1] \cong G^{\circ} \times G^{\circ}/ \mu_3
\] 
where the group $\mu_3$ of third roots of unity acts diagonally.  
Also possible is $[(\bullet,2\bullet),(\sigma_1,\sigma_1),\mathfrak w]$, leading to 
 \[
O/\mathbb G[1] \cong \mathrm{Hilb}^2(G^{\circ})/ \mu_3
\] 
where the $\mu_3$-action is induced by a diagonal action on $G^{\circ} \times G^{\circ}$. Finally, the case where all 3 points lie on the same component does not occur, as there is a numerical obstruction.

\subsubsection{$3$-dimensional strata}
 We start with the unique narrow stratum, which is $B(0,1,-1)$, which is of type $[(\circ,\bullet,\bullet),(-,\sigma_1,-\sigma_1),\mathfrak n]$. The standard arguments, identifying the kernel of $\mu$ with the first two factors of $(G^{\circ})^3$, show that this stratum is
 \[
O/\mathbb G[1] \cong G^{\circ} \times G^{\circ}/\mathbb G[1] \cong G^{\circ} \times D,
\]  
where we recall that $ D = G^{\circ}/T $.

Next we consider $C(1,2,3)$ which is $[(\circ,\bullet,\bullet),(-,\sigma_1,\sigma_2),\mathfrak w]$, leading to
 \[
O/\mathbb G[2] \cong G^{\circ} \times G^{\circ}/ T_1 \cong G^{\circ} \times D,
\]  
where $T_1$ is the rank $1$ subtorus of $\mathbb G[2]$ with coordinate $\sigma_1$. The case $C(1,2,-3)$ is identical.  The case $C(1,2,2)$ however allows two possibilities, the first being $[(\bullet,\bullet,\bullet),(\sigma_1,\sigma_2,\sigma_2),\mathfrak w]$. We can identify the kernel $\mu^{-1}(e)$ with the two last factors of $(G^{\circ})^3$. The rank $1$ subtorus $T_2 \subset \mathbb G[2]$ 
then acts on
the kernel and the stratum is isomorphic to 
 \[
O/\mathbb G[2] \cong G^{\circ} \times G^{\circ}/ T_2 \cong G^{\circ} \times D.
\]   
where the last isomorphism again follows from a suitable change of coordinates. 
The second possibility is $[(\bullet,2\bullet),(\sigma_1,\sigma_2),\mathfrak w]$ where we get 
 \[
O/\mathbb G[2]  \cong \mathrm{Hilb}^2(G^{\circ})/ T_2
\] 
with the action of $T_2$ on $\mathrm{Hilb}^2(G^{\circ})$ induced from the action $(\sigma_2,\sigma_2)$ on $G^{\circ}\times G^{\circ}$.

We also have the case $C(1,1,2)$. This is, after relabeling the same as $C(1,2,2)$ and, finally, $C(1,1,-2)$ is, modulo a change of coordinates $\sigma_2 \mapsto -\sigma_2$, also identical.

\subsubsection{$2$-dimensional strata}
 Here we have $D(1,2,3)$ and $D(1,2,-3)$ which are, modulo a change of coordinates, the same. In our notation this corresponds to $[(\bullet,\bullet,\bullet),(\sigma_1,\sigma_2,\sigma_3),\mathfrak w]$ and the same arguments as above give
  \[
O/\mathbb G[3] \cong G^{\circ} \times G^{\circ}/ (T_1 \times T_2) \cong D \times D.
\]  

\subsection{Kummer surfaces $(n=2)$}
\label{subsection:geomstrataKum}
Lastly, we discuss the case when $n=2$, i.e. Kummer surfaces.
Here we take up and recover some of the results of Subsection \ref{subsec:Kummersurfaces}.
There we proved in Theorem \ref{theorem:Kummerdeg} that  
$ K^1_{\mathcal{Y}/C} \to C $ is a projective Kulikov model. 
 Kulikov's theory of type II degenerations of K3 surfaces tells us what to expect, namely a chain of surfaces where the two ends are rational surfaces and the surfaces inbetween are 
$\mathbb P^1$-bundles over an elliptic curve $D$. Any two consecutive surfaces intersect transversally along sections of the $\mathbb P^1$-bundles.  

Indeed, our construction results in  a chain of $N+1$ surfaces 
$$S_0 \cup S_1 \cup \ldots \cup S_{N}$$ 
intersecting along curves $S_i \cap S_{i+1} \cong D_{i,i+1}$ for $i=0, \ldots, N-1$. Accordingly, the dual complex is a simple tree with $N+1$ vertices $v_i, i=0, \ldots, N$ and $N$ edges connecting $v_i$ and $v_{i+1}$ for $i=0, \ldots, N-1$.
This follows from Kulikov's theory of type II degenerations, but it can also be deduced by the same methods which we used to describe the dual complex for $n=3$ in Section \ref{sec:dualcplxIII} (the case $n=2$ is much easier).

We claim that the open surfaces $S_i \setminus (D_{i-1,i} \cup D_{i,i+1}), i \neq 0, N$ are isomorphic to a semi-abelian variety $G$ over the elliptic curve $D \cong D_{i,i+1}$ , whereas $S_0$ and $S_{N}$ are rational surfaces. 

We shall first treat the 2-dimensional strata. This is the situation also discussed in Subsection \ref{subsection:stratificationKumersurfaces} (the case $b=0$), i.e., no inserted components.
 The 2-dimensional strata are all narrow. The first case is that the two points belong to different components. and in the notation introduced above, this is the case $[(\circ,\circ),(-,-),\mathfrak n]$. 
This means that we look at the addition map $\mu: G^{\circ} \times G^{\circ} \to G^{\circ}$ and the kernel 
\[
O =\mu^{-1}(e) \cong G^{\circ},
\]
which is obviously isomorphic to $G^{\circ}$, as any point $P$ in, say the first component of $G^{\circ} \times G^{\circ}$, defines a unique point $(P,-P) \in O$. 
The resulting strata 
are open parts of the surfaces $S_i, i \neq 0, N$.
Indeed, the semi-abelian surfaces have a natural compactification as a  $\mathbb P^1$-bundle over the elliptic curve $D$, given by adding a $0$-section and an $\infty$-section,m thus giving the surfaces $S_i, i\neq 0,N$. The open stratum is given by removing these two sections.

The surfaces $S_0$ and $S_{N}$ correspond to the case where the two points lie on the same component, i.e., to $[(2\circ),(-),\mathfrak n]$. The interior of these surfaces, i.e., away from the intersection with $S_1$ and $S_{N-1}$ respectively, is now given by 
\[
\mu^{-1}(e) \cong\mathcal K^1(G^{\circ})\cong \widetilde G^{\circ}/{\tilde{\iota}} \cong \widetilde{G^{\circ}/\iota} 
\]
where $\mathcal K^1(G^{\circ})$ is the Kummer surface of the semi-abelian surface $G^{\circ}$. This is isomorphic to the blowup $\widetilde{G^{\circ}/\iota}$ of the 8 $A_1$-singularities  of  $G^{\circ}/\iota$ or, equivalently, to
$\widetilde G^{\circ}/{\tilde{\iota}}$ where $\widetilde G^{\circ}$ is the blowup of $G^{\circ}$ in the 8 fixed points of the involution $\iota$, and $\tilde \iota$ is the induced involution on  $\widetilde G^{\circ}$. This
is a rational surface.  Indeed, as this is a birational statement, we can consider the projective surface which is given as a resolution of the quotient $G/\iota$ in its 8 $A_1$-singularities, where $G$ 
is the $\mathbb P^1$-bundle 
compactifying $G^{\circ}$. Since $G$ is a $\mathbb P^1$-bundle it has Kodaira dimension $-\infty$ and hence the same holds for the quotient and its resolution. 
By surface theory, it then remains to show that the 
irregularity of the quotient is 0, in other words, that it does not admit a holomorphic $1$-form. Now $G$ has regularity 1 where the 1-form comes as a pullback of the $1$-form on the elliptic curve $D$.
Since the involution $\iota$ acts by multiplication by $-1$, this form does not descend to the quotient.  

The remaining case is that we have two points on inserted components. This is the case treated in Lemma \ref{lemma:KumSurb1} (the case $b=1$). We have a narrow stratum with opposite action of the torus on the two points, i.e., $[(\bullet,\bullet),(\sigma_1,-\sigma_1),\mathfrak n]$. The resulting stratum then is the quotient
\[
O/\mathbb G[1] = G^{\circ}/\mathbb G[1] \cong D
\]
which is the intersection  of two consecutive surfaces $S_i$. Clearly, this fits with the geometry of type II Kulikov models.

\appendix

\section{Closure of strata} \label{appendix-new}

\subsection{Overview}

We consider the degeneration family $$f \colon \mathcal{Y} \longrightarrow C.$$ For any admissible line chart $\mathcal{L}$, let $\mathcal{X}_n(\mathcal{L})$ and $X_n(\mathcal{L})$ be the Hilbert and Kummer strata for length $n$ subschemes associated to $\mathcal{L}$.  Moreover, let $\mathcal{X}_n(\mathcal{L})^\circ$ and $X_n(\mathcal{L})^\circ$ be the corresponding open subvariety parametrizing tuples of $n$ distinct points. The goal of this appendix is to prove the following result

\begin{prop} \label{prop:C-goal}
	Assume $\mathcal{L}_1$ and $\mathcal{L}_2$ are admissible line charts, such that $\mathcal{L}_1 \prec \mathcal{L}_2$, then
	$$ \overline{X_n(\mathcal{L}_1)} \cap \mathcal{X}_n(\mathcal{L}_2) = X_n(\mathcal{L}_2). $$
\end{prop}

To start with, it is clear that $ \overline{X_n(\mathcal{L}_1)} \cap \mathcal{X}_n(\mathcal{L}_2) \subseteq X_n(\mathcal{L}_2). $ Since $X_n(\mathcal{L}_2)^\circ$ is dense in $X_n(\mathcal{L}_2)$, it suffices to show that
\begin{equation}\label{eqn:C-goal}
	\overline{X_n(\mathcal{L}_1)} \supseteq X_n(\mathcal{L}_2)^\circ.
\end{equation}

Moreover, it suffices to assume that the inclusion $\mathcal{L}_1 \prec \mathcal{L}_2$ is proper and simple, namely, $\mathcal{L}_1 \neq \mathcal{L}_2$, and there is no intermediate admissible line chart between them. 

In the following we assume that $\mathcal{L}_2$ has $b+1$ vertices, namely, any configuration of type $\mathcal{L}_2$ can only be realized in the central fiber $\mathcal{Y}[b]_0$ of the expansion $$f[b] \colon \mathcal{Y}[b] \longrightarrow C[b].$$ 
To prove \eqref{eqn:C-goal}, we consider an arbitrary tuple 
\begin{equation}\label{eqn:C-config}
	\{P_1, \dots, P_n\} \in X_n(\mathcal{L}_2)^\circ.
\end{equation}
We need to construct a family
\begin{equation}\label{eqn:C-family}
	\alpha_i \colon \Spec k[[\pi]] \longrightarrow \mathcal{Y}[b]
\end{equation}
satisfying $\alpha_i(0) = P_i$ for each $1 \leq i \leq n$ and 
\begin{equation}\label{eqn:C-in-L1}
	\{ \alpha_1(\pi), \dots, \alpha_n(\pi) \} \in X_n(\mathcal{L}_1)
\end{equation}
for $\pi \neq 0$.

The proof will occupy the rest of the section, which will be divided into three cases. In Subsection \ref{subsec:C-local-coord}, we will recall the local coordinates in the expanded degeneration for specifying points. In Subsections \ref{subsec:C-same-status-wide} and \ref{subsec:C-same-status-narrow}, we will prove \eqref{eqn:C-goal} in the cases when $\mathcal{L}_1$ and $\mathcal{L}_2$ are both wide and both narrow respectively. In Subsection \ref{subsec:C-diff-status}, we will prove \eqref{eqn:C-goal} in the case when $\mathcal{L}_1$ is narrow while $\mathcal{L}_2$ is wide.

\subsection{Preliminaries}

\subsubsection{Conventions}
For any $b$, we denote the components of $\mathcal{Y}[b]_0$ by $G_i$, where $0 \leq i < 2N(b+1)$. Recall that each $G_i$ is a $\mathbb{P}^1$-bundle, which, if we remove the $0$ and $\infty$ sections, is isomorphic to the semi-abelian surface 
$$ 0 \to \mathbb{G}_m \to G^{\circ} \to D \to 0.$$

For each $0 \leq i \leq 2N(b+1) - 2 $, $G_i$ intersects $G_{i+1}$ in $D_i$ which we think of both as the $\infty$-section on $G_i$ and the $0$-section on $G_{i+1}$. The latter is naturally identified with the $\infty$-section of $G_{i+1}$ via the $\mathbb{P}^1$-bundle structure. Note that the same identification can not be done for the intersection $G_{2N(b+1)-1} \cap G_0$, since there is an, in general, non-trivial shift along $D$. 

In any case, we can now associate to any point $ P \in \mathcal{Y}[b]_0$ a well defined projection $d = \mathrm{pr}(P) \in D$. The only potential ambiguity is when $ P \in G_{2N(b+1)-1} \cap G_0 $, in which we view $P$ as a point on the $\infty$-section on $G_{2N(b+1)-1}$. For any $n$-tuple $P_1, \ldots, P_n$, it is a necessary condition for being in the Kummer locus that the corresponding $d_1, \ldots, d_n$ add to the identity $e_D$ in $D$.

\subsubsection{}
Each $ \alpha_i $ will be a lift of a certain map $ \overline{\alpha_i} \colon S \to \mathcal{Y} $, where $S$ maps into a $\mathbb{P}^1$-fiber $\mathbb{P}_d \subset G $ where $G$ is a component of $\mathcal{Y}_0$ and $ d \in D $ a chosen point in the base. Depending on the context we will also interpret $d$ as either $0$ or $\infty$ on $\mathbb{P}_d$. 

The closed point $s \in S$ will always map to $d$ in the double locus $ \mathrm{Sing}(\mathcal{Y}_0)$. The generic point $\eta$, however, will either map to $\mathbb{P}_d \setminus \{0, \infty\} $ (when $b=0$), or to $d$ as well (when $b>0$). 

\subsubsection{}
We now explain how to find a suitable open affine $\mathcal{U} \subset \mathcal{Y}$ such that $\overline{\alpha}_i$ factors through $\mathcal{U}$. First of all, we remove all components of $\mathcal{Y}_0$ except for $G$ and the unique other component $G'$ passing through $d$. 
If we further remove finitely many fibers over points $d' \neq d$ from $G$ and $G'$, we can assume that both $G$ and $G'$ are trivial fibrations over an open affine subset of $D$. Finally, inside this open subset of $\mathcal{Y}$, we let $\mathcal{U}$ be an open affine subset containing $d$, and where both $G$ and $G'$ are principal divisors. 

\subsubsection{}
We fix an \'etale coordinate $ C \to \mathbb{A}^1 = \mathrm{Spec}~k[t] $, and define $X = \mathrm{Spec}~k[x,y,z]$ (the variable $z$ is only present if the relative dimension $\delta$ equals $2$). It is a scheme over $\mathbb{A}^1$ via the map $t \mapsto xy$. Let $w$, resp.~$w'$ be a local equation of $G'$, resp.~$G$. Then we can find an \'etale map $ \mathcal{U} \to X $, where $ x \mapsto w$ and $ y \mapsto w'$.

We write $A = \mathcal{O}(\mathcal{U})$. Multiplying, if necessary, $w$ by a unit, we can assume that $w$ maps to the coordinate of the $\mathbb{P}^1$-fibers in $A/(w')$. The composition
$$ k[x,y,z] \to A \to R $$
then necessarily factors through
$$ k[x,z] \to A/(w') \to R,$$
where $z \mapsto 0$, and $ x $ is mapped either to $c \pi^a $ (with $c$ a unit and $a \geq 0$) or to $0$, depending on the limit we want to study. This describes the map $S \to X$ induced by $ \overline{\alpha_i} $.

\subsubsection{}
For each $b \geq 0$ the diagram
\begin{equation}
\xymatrix{
    \mathcal{U}[b]  \ar[r] \ar[d] & \mathcal{U}   \ar[d] \\
   \mathcal{Y}[b]   \ar[r] & X
}
\end{equation}
is Cartesian (\cite[Prop.~1.13]{GHH}). Given $\overline{\alpha}_i$, the lift $ \alpha_i \colon S \to \mathcal{U}[b] \subset \mathcal{Y}[b] $ is therefore uniquely determined by a lift of the composition $S \to X $ to $\mathcal{Y}[b]$. This can be effectively constructed using the local equations from \cite[Prop.~1.7]{GHH}. 

\subsubsection{Local equations} \label{subsec:C-local-coord}

For the expansion family $f[b] \colon \mathcal{Y}[b] \to C[b]$, we consider the coordinates in the \'etale neighborhood $\mathcal{U}[b]$ of $\mathcal{Y}[b]$. We can understand $C[b]$ as $\mathbb{A}^{b+1}$ with coordinates $(t_1, \dots, t_{b+1})$.
To describe any point $P \in \mathcal{Y}[b]$, we assume 
\begin{equation}\label{eqn:C-tcoord}
	f[b](P) = (t_1, \dots, t_{b+1}),
\end{equation}
then we have the local coordinates
\begin{equation}\label{eqn:C-coord}
	P=(d, \theta_0, \theta_1, \dots, \theta_b, \theta_{b+1})
\end{equation}
where $d = \mathrm{pr}(P) \in D$ is the projection of $P$ to the double locus $D$, and
$$ \theta_j = \frac{t_1 \dots t_j}{x} = \frac{y}{t_{j+1} \dots t_{b+1}} $$
for $0 \leq j \leq b+1$. Recall from Subsection \ref{subsec:torus-action-expansion} that $\theta_j$'s are indeed homogeneous coordinates $$ \theta_j = \frac{u_j}{v_j} \in \mathbb{P}^1, $$ among which $\theta_0 = \frac{1}{x} \in \mathbb{P}^1 \setminus \{0\}$ and $\theta_{b+1} = y \in \mathbb{P}^1 \setminus \{\infty\} = k$. Note that these homogeneous coordinates are not independent. They are related by the equations
\begin{equation}\label{eqn:C-compat}
	\theta_j = t_j \theta_{j-1}
\end{equation}
for $1 \leq j \leq b+1$.

\begin{example}
	Assume that the point $P$ lies in the smooth locus of the central fiber $\mathcal{Y}[b]_0$, then we have $P \in G^b_{\tau, \varepsilon\xi}$ for some $\varepsilon \in \{ \pm 1 \}$ and $0 \leq \xi \leq b+1$. In such a case the $\theta$-coordinates of $P$ can be given as
$$ \theta_j= \begin{cases}
	\infty, & j<\xi; \\
	\lambda\in k^\ast, & j=\xi; \\
	0, & j>\xi.
\end{cases} $$
This will be the starting point for the construction of the families of $\alpha_i$'s.
\end{example}

\subsection{The case when $\mathcal{L}_1$ and $\mathcal{L}_2$ are both wide} \label{subsec:C-same-status-wide}

Since we assumed that the inclusion $\mathcal{L}_1 \prec \mathcal{L}_2$ is simple, in such a case $\mathcal{L}_1$ is obtained by removing one vertex from $\mathcal{L}_2$, which we assume to be the $\ell$-th vertex counted from the right side.

\subsubsection{Coordinates of the limit}\label{subsubsec:coord-limit}

We consider the set of $n$ points given in \eqref{eqn:C-config}. For each $1 \leq i \leq n$, we assume that $$P_i \in G^b_{\tau_i, \varepsilon_i \xi_i},$$ where $\varepsilon_i = \pm 1$. Then the local coordinates of $P_i$ can be set as $d(P_i) = d_i$, and for $0 \leq j \leq b+1$
$$ \theta_j(P_i) = \begin{cases}
 \infty, & j < \xi_i; \\
 \lambda_i \in k^\ast, & j = \xi_i; \\
 0, & j > \xi_i.	
 \end{cases} $$
 
 \subsubsection{Construction of the family} \label{subsubsec:constr-family-wide}
 
 Now we construct the required family $\alpha_i$ for each $1 \leq i \leq n$. We require that the image of each $\alpha_i$ is in the same \'etale neighborhood as $P_i$, therefore it can be given in the same set of local coordinates. First of all, the $t$-coordinates as in \eqref{eqn:C-tcoord} are given as
\begin{equation}\label{eqn:family-both-same}
	t_j(\alpha_i(\pi)) = \begin{cases}
 	\pi, & j = \ell; \\ 0, & \text{otherwise}.
	\end{cases}
\end{equation}
The $d$-coordinate as in \eqref{eqn:C-coord} is given as a constant value
 $$ d(\alpha_i(\pi)) = d_i. $$
 Depending on $\xi_i$, there are three cases for the $\theta$-coordinates in \eqref{eqn:C-coord} 
 \begin{itemize}
 	\item If $\xi_i \neq \ell - 1$ or $\ell$, then $$ \theta_j(\alpha_i(\pi)) = \begin{cases}
 		\infty, & j < \xi_i; \\ \lambda_i, & j = \xi_i; \\ 0, & j > \xi_i. 
 	\end{cases}$$
 	\item If $\xi_i = \ell -1$, then  $$ \theta_j(\alpha_i(\pi)) = \begin{cases}
 		\infty, & j < \ell-1; \\ \lambda_i, & j = \ell-1; \\ \lambda_i t_\ell, & j = \ell; \\ 0, & j > \ell. 
 	\end{cases}$$
 	\item If $\xi_i = \ell$, then $$ \theta_j(\alpha_i(\pi)) = \begin{cases}
 		\infty, & j < \ell-1; \\ \frac{\lambda_i}{t_\ell}, & j = \ell-1; \\ \lambda_i, & j = \ell; \\ 0, & j > \ell. 
 	\end{cases}$$
 \end{itemize}
 
It is immediate to see that the above coordinates satisfy the compatibility condition \eqref{eqn:C-compat}, hence they are legitimate coordinates for each family $\alpha_i$. In other words, we have constructed the family $\alpha_i(\pi)$ converging to the point $P_i$ for each $1 \leq i \leq n$, as required in \eqref{eqn:C-family}. It remains to show \eqref{eqn:C-in-L1} is satisfied.
 
\subsubsection{Verifying the Kummer condition} \label{subsubsec:C-both-wide}

The Kummer condition for a wide stratum is given by only one equation, using the group law on the double locus $D$. More precisely, \eqref{eqn:C-config} implies that 
$$d_1 + \dots + d_n = d(P_1) + \dots + d(P_n) = 0.$$ 
Since $d(\alpha_i(\pi)) = d(P_i) = d_i$ for $1 \leq i \leq n$, we still have
$$ d(\alpha_1(\pi)) + \dots + d(\alpha_n(\pi)) = d_1 + \dots + d_n = 0 $$
which implies \eqref{eqn:C-in-L1}; in other words the tuple $\{ \alpha_1(\pi), \dots, \alpha_n(\pi) \}$ is in the Kummer locus.

\subsection{The case when $\mathcal{L}_1$ and $\mathcal{L}_2$ are both narrow} \label{subsec:C-same-status-narrow}

In this situation, we still assume that $\mathcal{L}_1$ is obtained by removing the $\ell$-th vertex in $\mathcal{L}_2$, counted from the right side. Given a tuple of $n$ distinct points $\{ P_i \mid 1 \leq i \leq n \}$, we still use the same notations for their coordinates as in Subsection \ref{subsubsec:coord-limit}, and construct the same family $\{ \alpha_i \mid 1 \leq i \leq n \}$ that converges to the given tuple. What remains to be done is to check that the family satisfies the Kummer condition for narrow strata.

The Kummer condition for a narrow stratum means that we take the inverse image of the origin under the summation map.  
More precisely, starting with $n$ points $\{P_1, P_2, \dots, P_n\}$ as in \eqref{eqn:C-config} the Kummer 
condition translates into two equations, namely
\begin{align}
	d_1 + d_2 + \dots + d_n &= 0, \label{eqn:narrow1}\\
	\lambda_1^{\varepsilon_1} \cdot \lambda_2^{\varepsilon_2} \cdot \dots \cdot \lambda_n^{\varepsilon_n} &= 1, \label{eqn:narrow2}
\end{align}
where the first equation is the Kummer condition in the direction of the double locus, and the second equation is the Kummer condition in the direction of the $\mathbb{P}^1$-fiber.

\subsubsection{Verifying the Kummer condition}

To verify the Kummer condition for the family $\{ \alpha_1(\pi), \alpha_2(\pi), \dots, \alpha_n(\pi) \}$, we also need to check two equations similar to the above ones. First of all, we observe that the condition in the direction of the double locus 
\begin{equation}\label{eqn:next1}
	\sum_{i=1}^n d_i = 0
\end{equation}
holds by \eqref{eqn:narrow1} immediately. To check the condition in the direction of the $\mathbb{P}^1$-fiber, we have to distinguish three cases.

When $\ell \neq 1$ or $b+1$, the coordinates $\theta_{\ell-1}$ and $\theta_{\ell}$ are proportional to each other on each fiber of $f[n]$ with $t_\ell \neq 0$, hence they are related by the action $G[b]=\mathbb{G}_m^b$ on the fiber. Correspondingly, there are two equivalent formulations for the Kummer condition in the direction of the $\mathbb{P}^1$-fiber. The formulation using the coordinate $\theta_{\ell-1}$ is
\begin{equation}\label{eqn:next2}
	\prod_{\substack{i=1 \\ \xi_i \neq \ell }}^n \lambda_i^{\varepsilon_i} \cdot \prod_{\substack{i=1 \\ \xi_i = \ell }}^n \left( \frac{\lambda_i}{t_\ell} \right)^{\varepsilon_i} = 1,
\end{equation}
while the formulation using the coordinate $\theta_\ell$ is
\begin{equation}\label{eqn:narrow-second}
	\prod_{\substack{i=1 \\ \xi_i \neq \ell-1 }}^n \lambda_i^{\varepsilon_i} \cdot \prod_{\substack{i=1 \\ \xi_i = \ell-1 }}^n \left( \lambda_i t_\ell \right)^{\varepsilon_i} = 1.
\end{equation}

To prove \eqref{eqn:next2} and \eqref{eqn:narrow-second}, and to explain why they are equivalent, we introduce the notation
\begin{equation}\label{eqn:C-seg-diff}
	m_j = \# \{ i \mid \xi_i = j \text{ and } \varepsilon_i = 1 \} - \# \{ i \mid \xi_i = j \text{ and } \varepsilon_i = -1 \}
\end{equation}
for each $1 \leq j \leq b$. The narrowness of $\mathcal{L}_2$ implies that
\begin{equation}\label{eqn:narrowness_L2}
	m_j = 0
\end{equation}
for each $j$. Recall that we also have $t_\ell = \pi$. Then \eqref{eqn:next2} follows from
$$ \prod_{\substack{i=1 \\ \xi_i \neq \ell }}^n \lambda_i^{\varepsilon_i} \cdot \prod_{\substack{i=1 \\ \xi_i = \ell }}^n \left( \frac{\lambda_i}{t_\ell} \right)^{\varepsilon_i} = \pi^{-m_\ell} \cdot \prod_{i=1}^n \lambda_i^{\varepsilon_i} = 1, $$
where the last equality holds due to \eqref{eqn:narrow2} and \eqref{eqn:narrowness_L2}. We can similarly prove \eqref{eqn:narrow-second} by observing that
$$ \prod_{\substack{i=1 \\ \xi_i \neq \ell-1 }}^n \lambda_i^{\varepsilon_i} \cdot \prod_{\substack{i=1 \\ \xi_i = \ell-1 }}^n \left( \lambda_i t_\ell \right)^{\varepsilon_i} = \pi^{m_{\ell-1}} \cdot \prod_{i=1}^n \lambda_i^{\varepsilon_i} = 1, $$
whose last equality holds also due to \eqref{eqn:narrow2} and \eqref{eqn:narrowness_L2}. In fact, the above calculation also shows that the left sides of \eqref{eqn:next2} and \eqref{eqn:narrow-second} differ by a factor of $$\pi^{m_{\ell-1} + m_\ell}=1$$ due to \eqref{eqn:narrowness_L2}, which explains the equivalence of the formulations \eqref{eqn:next2} and \eqref{eqn:narrow-second}.

When $\ell = 1$, $\theta_0$ and $\theta_1$ are no longer related by the action of $G[b]$, hence we can only use the original coordinate $\theta_0$ to verify the Kummer condition, which is formulated as \eqref{eqn:next2}, whose proof is still valid since it only requires $m_1=0$ and \eqref{eqn:narrow2}.

When $\ell=b+1$, $\theta_{b}$ and $\theta_{b+1}$ are no longer related by the action of $G[b]$, hence we can only use the original coordinate $\theta_{b+1}$ to verify the Kummer condition, which is formulated as \eqref{eqn:narrow-second}, whose proof is still valid since it only requires $m_\ell=0$ and \eqref{eqn:narrow2}.

Combining all the above possible values of $\ell$, we have verified the Kummer condition for narrow strata given by \eqref{eqn:next1} and either \eqref{eqn:next2} or \eqref{eqn:narrow-second}. Therefore, we conclude \eqref{eqn:C-in-L1} under the assumption that both $\mathcal{L}_1$ and $\mathcal{L}_2$ are narrow.

\subsection{The case when $\mathcal{L}_1$ is narrow and $\mathcal{L}_2$ is wide} \label{subsec:C-diff-status}

In such a case all vertices of $\mathcal{L}_1$ lie on the neutral line, while $\mathcal{L}_2$ has two extra vertices, one above and the other below the neutral line. We assume that the vertices of $\mathcal{L}_2$ that are higher and lower than the neutral line are the $\ell_+$-th and $\ell_-$-th respectively, when counted from the right side. 

Given a tuple of distinct points $\{ P_1, \dots, P_n \} \in \mathrm{Kum}(\mathcal{L}_2)^\circ$, we still assume that their coordinates are given as in Subsection \ref{subsubsec:coord-limit}. In the following we will produce a family $\{ \alpha_1(\pi), \dots, \alpha_n(\pi) \} \in \mathrm{Kum}(\mathcal{L}_1)$ for $\pi \neq 0$, such that $\alpha_i(\pi)$ converges to $P_i$ as $\pi$ approaches $0$ for each $1 \leq i \leq n$.

\subsubsection{Construction of the family}

To construct the families $\alpha_i(\pi)$ for $1 \leq i \leq n$, the $d$-coordinates are given as constant values as usual, namely
\begin{equation}\label{eqn:C-family-d-coord}
	d(\alpha_i(\pi)) = d_i.
\end{equation}    
We also give the $t$-coordinates as
\begin{equation}\label{eqn:C-family-base}
t_j(\alpha_i(\pi)) = \begin{cases}
c_- \pi^{a_-}, & j = \ell_-; \\
c_+ \pi^{a_+}, & j = \ell_+; \\
0, & \text{otherwise}.	
\end{cases}
\end{equation}
We will determine the parameters $a_\pm \in \mathbb{Z}^+$ and $c_\pm \in k^\ast$ later by the Kummer condition that the family has to satisfy. It is worth noting that there is no restriction for the limit tuple $\{ P_1, \dots, P_n \}$ in the fiber direction because $\mathcal{L}_2$ is wide, but we still need to consider the Kummer condition in the fiber direction for the family $\{ \alpha_1(\pi), \dots, \alpha_n(\pi) \}$ when $\pi \neq 0$ because $\mathcal{L}_1$ is narrow.

Depending on the relative position of the two vertices that are removed from $\mathcal{L}_2$ to obtain $\mathcal{L}_1$, we separate two cases.

\subsubsection{The Kummer condition when $\lvert \ell_+ - \ell_- \rvert = 1$} \label{subsubsec:C-diff-equal}

Without loss of generality, we assume that $1 \leq \ell_- < \ell_+ \leq b+1$, then we can further write $$\ell_- = \ell  \quad \text{and} \quad \ell_+ = \ell+1$$ where $1 \leq \ell \leq b$.

Since the $d$- and $t$-coordinates are already given in \eqref{eqn:C-family-d-coord} and \eqref{eqn:C-family-base}, it remains to give the $\theta$-coordinates. Depending on the value of $\xi_i$, they are given as
\begin{itemize}
	\item if $\xi_i \notin \{ \ell-1, \ell, \ell+1 \}$, then $$ \theta_j(\alpha_i(\pi)) = \begin{cases}
			\infty, & j < \xi_i; \\
			\lambda_i, & j = \xi_i; \\
			0, & j > \xi_i.
		\end{cases} $$
	\item if $\xi_i = \ell-1$, then $$ \theta_j(\alpha_i(\pi)) = \begin{cases}
			\infty, & j < \ell-1; \\
			\lambda_i, & j = \ell-1; \\
			\lambda_i t_{\ell}, & j = \ell; \\
			\lambda_i t_{\ell} t_{\ell+1}, & j = \ell+1; \\
			0, & j > \ell+1.
		\end{cases} $$
	\item if $\xi_i = \ell$, then $$ \theta_j(\alpha_i(\pi)) = \begin{cases}
			\infty, & j < \ell-1; \\
			\frac{\lambda_i}{t_{\ell}}, & j = \ell-1; \\
			\lambda_i, & j = \ell; \\
			\lambda_i t_{\ell+1}, & j = \ell+1; \\
			0, & j > \ell+1.
		\end{cases} $$
	\item if $\xi_i = \ell+1$, then $$ \theta_j(\alpha_i(\pi)) = \begin{cases}
			\infty, & j < \ell-1; \\
			\frac{\lambda_i}{t_{\ell} t_{\ell+1}}, & j = \ell-1; \\
			\frac{\lambda_i}{t_{\ell+1}}, & j = \ell; \\
			\lambda_i, & j = \ell+1; \\
			0, & j > \ell+1.
		\end{cases} $$
\end{itemize}

It is immediate to see that these assignments satisfy the relation \eqref{eqn:C-compat}, therefore they are valid coordinates for points on $\mathcal{Y}[b]$.

It remains to check the Kummer condition \eqref{eqn:C-in-L1}, which decomposes into two equations. Along the double locus $D$, we have
$$ \sum_{i=1}^n d_i = 0 $$
as usual. To guarantee the Kummer condition along the direction of the $\mathbb{P}^1$-fiber, we have to observe that $\theta_{\ell-1}$, $\theta_{\ell}$ and $\theta_{\ell+1}$ are coordinates on the same component. If this component is one of the inserted components, then they are related by $G[b]$-action, so we can use any of them to examine the Kummer condition. However, if this component is one of the components of the original degeneration family, only the original coordinates can be used; namely, $\theta_{\ell-1}$ if $\ell = 1$, or $\theta_{\ell+1}$ if $\ell = b$. 

Therefore, in the discussion below, we will first state the approach using the coordinate $\theta_{\ell-1}$, which is valid for $1 \leq \ell \leq b-1$; and then state the approach using the coordinate $\theta_{\ell+1}$, which is valid for $2 \leq \ell \leq b$. Both approaches together cover all possibilities of $\ell$.

For $1 \leq \ell \leq b-1$, since the vertices in the line chart labelled by $\ell+1$ and $\ell$ are above and below the neutral line respectively, we see that $$m_{\ell +1} \in \mathbb{Z}^+ \qquad \text{and} \qquad m_{\ell} + m_{\ell+1} \in \mathbb{Z}^-.$$
In such a case we state the Kummer condition in the fiber direction using the coordinate $\theta_{\ell-1}$, which is given by the equation
\begin{equation*}
	\prod_{\xi_i \neq \ell-1, \ell, \ell+1} \lambda_i^{\varepsilon_i} \cdot \prod_{\xi_i = \ell -1} \lambda_i^{\varepsilon_i} \cdot \prod_{\xi_i = \ell} \left( \frac{\lambda_i}{t_\ell} \right)^{\varepsilon_i} \cdot \prod_{\xi_i = \ell+1} \left( \frac{\lambda_i}{t_\ell t_{\ell+1}} \right)^{\varepsilon_i} = 1.
\end{equation*}
By the notation \eqref{eqn:C-seg-diff}, we can rewrite the above equation as
\begin{equation*}
	\prod_{i=1}^n \lambda_i^{\varepsilon_i} \cdot t_{\ell}^{- \left( m_\ell + m_{\ell+1} \right) } = t_{\ell+1}^{m_{\ell+1}}.
\end{equation*}

There are many choices for the above equation to hold. For example, we can choose the solution
$$ t_{\ell_-} = t_{\ell}=\pi^{m_{\ell+1}} \qquad \text{and} \qquad t_{\ell_+} = t_{\ell+1} = \left( \prod_{i=1}^n \lambda_i^{\varepsilon_i} \right)^{\frac{1}{m_{\ell+1}}} \cdot \pi^{-\left( m_\ell + m_{\ell+1} \right)}. $$
In other words, in the parametrization \eqref{eqn:C-family-base}, it suffices to choose
$$ c_+=\left( \prod_{i=1}^n \lambda_i^{\varepsilon_i} \right)^{\frac{1}{m_{\ell+1}}} \qquad \text{and} \qquad c_-=1  $$
as well as
$$ a_+=- \left( m_{\ell}+m_{\ell+1} \right) \qquad \text{and} \qquad a_-=m_{\ell-1} $$
to guarantee that the family $\{ \alpha_1(\pi), \dots, \alpha_n(\pi) \}$ satisfies the Kummer condition for narrow strata.

For $2 \leq \ell \leq b$, similar to the above case we have
$$ m_{\ell-1} \in \mathbb{Z}^+ \qquad \text{and} \qquad m_{\ell-1} + m_\ell \in \mathbb{Z}^-. $$
Under the assumption we state the Kummer condition in the fiber direction using the coordinate $\theta_{\ell+1}$, which is given by the equation
\begin{equation*}
	\prod_{\xi_i \neq \ell-1, \ell, \ell+1} \lambda_i^{\varepsilon_i} \cdot \prod_{\xi_i = \ell -1} \left( \lambda_i t_\ell t_{\ell+1} \right)^{\varepsilon_i} \cdot \prod_{\xi_i = \ell} \left( \lambda_i t_{\ell+1} \right)^{\varepsilon_i} \cdot \prod_{\xi_i = \ell+1} \lambda_i^{\varepsilon_i} = 1.
\end{equation*}
By the notation \eqref{eqn:C-seg-diff}, we can rewrite the above equation as
\begin{equation*}
	\prod_{i=1}^n \lambda_i^{\varepsilon_i} \cdot t_{\ell}^{m_{\ell-1}} = t_{\ell+1}^{- \left( m_{\ell-1} + m_\ell \right)}.
\end{equation*}
As a solution to the above equation we can choose
$$ t_{\ell_-} = t_{\ell}=\pi^{- (m_{\ell-1} + m_\ell)} \qquad \text{and} \qquad t_{\ell_+} = t_{\ell+1} = \left( \prod_{i=1}^n \lambda_i^{\varepsilon_i} \right)^{\frac{1}{-(m_{\ell-1}+m_\ell)}} \cdot \pi^{ m_{\ell-1} }. $$
In other words, in the parametrization \eqref{eqn:C-family-base}, it suffices to choose
$$ c_+=\left( \prod_{i=1}^n \lambda_i^{\varepsilon_i} \right)^{\frac{1}{-(m_{\ell-1}+m_\ell)}} \qquad \text{and} \qquad c_-=1  $$
as well as
$$ a_+=m_{\ell-1} \qquad \text{and} \qquad a_-=- \left( m_{\ell-1}+m_{\ell} \right) $$
to guarantee that the family $\{ \alpha_1(\pi), \dots, \alpha_n(\pi) \}$ satisfies the Kummer condition for narrow strata.

Combining both approaches, for each possible value of $\ell$, we can find an appropriate family that converges to the given limit tuple.

\subsubsection{The Kummer condition when $\lvert \ell_+ - \ell_- \rvert > 1$} \label{subsubsec:C-diff-large}

We follow a similar strategy as in the previous case, and still assume that $1 \leq \ell_- < \ell_+ \leq b+1$ without loss of generality. Following the notations in \eqref{eqn:C-seg-diff}, this assumption implies that 
\begin{equation}\label{eqn:case-both-negative}
	m_{\ell_-} \in \mathbb{Z}^- \qquad \text{and} \qquad m_{\ell_+-1} \in \mathbb{Z}^-.
\end{equation}

Depending on the value of $\xi_i$, the $\theta$-coordinates of each $\alpha_i(\pi)$ is given as follows
\begin{itemize}
	\item If $\xi_i \neq \ell_\pm -1$ or $\ell_\pm$, then $$ \theta_j(\alpha_i(\pi)) = \begin{cases}
			\infty, & j < \xi_i; \\
			\lambda_i, & j = \xi_i; \\
			0, & j > \xi_i.
		\end{cases} $$
	\item If $\xi_i = \ell_--1$, then $$ \theta_j(\alpha_i(\pi)) = \begin{cases}
 		\infty, & j < \ell_--1; \\ \lambda_i, & j = \ell_--1; \\ \lambda_i t_{\ell_-}, & j = \ell_-; \\ 0, & j > \ell_-.
 		\end{cases}$$
 	\item If $\xi_i = \ell_-$, then $$ \theta_j(\alpha_i(\pi)) = \begin{cases}
 		\infty, & j < \ell_--1; \\ \frac{\lambda_i}{t_{\ell_-}}, & j = \ell_--1; \\ \lambda_i, & j = \ell_-; \\ 0, & j > \ell_-. 
 	\end{cases}$$
 	\item If $\xi_i = \ell_+-1$, then $$ \theta_j(\alpha_i(\pi)) = \begin{cases}
 		\infty, & j < \ell_+-1; \\ \lambda_i, & j = \ell_+-1; \\ \lambda_i t_{\ell_+}, & j = \ell_+; \\ 0, & j > \ell_+.
 		\end{cases}$$
 	\item If $\xi_i = \ell_+$, then $$ \theta_j(\alpha_i(\pi)) = \begin{cases}
 		\infty, & j < \ell_+-1; \\ \frac{\lambda_i}{t_{\ell_+}}, & j = \ell_+-1; \\ \lambda_i, & j = \ell_+; \\ 0, & j > \ell_+. 
 	\end{cases}$$
\end{itemize}

We see again that these assignments satisfy the relation \eqref{eqn:C-compat}, therefore they are valid coordinates for points on $\mathcal{Y}[b]$. The Kummer condition \eqref{eqn:C-in-L1} along the double locus $D$ is still valid. Moreover, although $\theta_{\ell_\pm-1}$ and $\theta_{\ell_\pm}$ are coordinates on the same component, we observe that the Kummer condition along the direction of the $\mathbb{P}^1$-fiber can always be formulated using the $\theta_{\ell_--1}$ and $\theta_{\ell_+}$ coordinates, no matter whether such components are original components or inserted components. Therefore the condition is given by the equation
\begin{equation*}
	\prod_{\xi_i \neq \ell_\pm-1, \ell_\pm} \lambda_i^{\varepsilon_i} \cdot \prod_{\xi_i = \ell_--1} \lambda_i^{\varepsilon_i} \cdot \prod_{\xi_i = \ell_-} \left( \frac{\lambda_i}{t_{\ell_-}} \right)^{\varepsilon_i} \cdot \prod_{\xi_i = \ell_+-1} \left( \lambda_i t_{\ell_+} \right)^{\varepsilon_i} \cdot \prod_{\xi_i = \ell_+} \lambda_i^{\varepsilon_i} = 1.
\end{equation*}
Following the notations in \eqref{eqn:C-seg-diff}, the above equation becomes
\begin{equation*}
	\prod_{i=1}^n \lambda_i^{\varepsilon_i} \cdot t_{\ell_-}^{-m_{\ell_-}} \cdot t_{\ell_+}^{m_{\ell_+-1}} = 1,
\end{equation*}
which can be further rewritten as
\begin{equation*}
	\prod_{i=1}^n \lambda_i^{\varepsilon_i} \cdot t_{\ell_-}^{-m_{\ell_-}} = t_{\ell_+}^{-m_{\ell_+-1}}.
\end{equation*}
As a solution to the above equation we choose
$$ t_{\ell_-} = \pi^{-m_{\ell_+-1}} \qquad \text{and} \qquad t_{\ell_+} = \left( \prod_{i=1}^n \lambda_i^{\varepsilon_i} \right)^{\frac{1}{-m_{\ell_+-1}}} \cdot \pi^{-m_{\ell_-}}. $$
Given the condition \eqref{eqn:case-both-negative}, it suffices to choose in the parametrization \eqref{eqn:C-family-base} that
$$ c_+ = \left( \prod_{i=1}^n \lambda_i^{\varepsilon_i} \right)^{\frac{1}{-m_{\ell_+-1}}} \qquad \text{and} \qquad c_-=1 $$
as well as
$$ a_+ = -m_{\ell_-} \qquad \text{and} \qquad a_-=-m_{\ell_+-1} $$
to guarantee that the family $\{ \alpha_1(\pi), \dots, \alpha_n(\pi) \}$ satisfies the Kummer condition for narrow strata.

\begin{proof}[Proof of Proposition \ref{prop:C-goal}]
	There are three possibilities for the pair $\mathcal{L}_1$ and $\mathcal{L}_2$, namely both being wide, or both being narrow, or the former being narrow while the latter being wide. The three cases have been treated in Subsections \ref{subsec:C-same-status-wide}, \ref{subsec:C-same-status-narrow} and \ref{subsec:C-diff-status}. Therefore \eqref{eqn:C-goal} holds in general.
\end{proof}

\section{A lemma in convex geometry}

In this section we prove the following combinatorial result, which is used in our study of the dual complex encoding 
the intersection relation of irreducible components in the degeneration of generalized Kummer varieties.
This lemma is very likely well known to specialists.

\begin{lemma}\label{lemma:descriteratedconeprod}
	Let $\delta_n$ be the $n$-simplex embedded in $\mathbb{R}^{n+1}$ with vertices given by the standard basis vectors
	$$ V_i = (0, \dots, 0, 1, 0, \dots, 0) \in \mathbb{R}^{n+1} $$ where $1$ is the $i$-th coordinate for $1 \leq i \leq n+1$. Assume that $n_+$ and $n_-$ are positive integers, and $n_0$ is a non-negative integer, such that $$n_++n_-+n_0 = n+1.$$ For 
    $z = (z_1, \dots, z_{n+1}) \in \mathbb{R}^{n+1}$, consider the linear function
	$$ \ell(z) = z_1 + \dots + z_{n_+} - z_{n_++1} - \dots - z_{n_++n_-} $$
	and let $\Lambda$ be the hyperplane in $\mathbb{R}^{n+1}$ defined by $\ell(z)=0$. Then we have an affine
    equivalence of polytopes
    $$ \delta_n \cap \Lambda \ \cong \ C^{n_0}(\delta_{n_+-1} \times \delta_{n_--1}), $$
	where $C^{n_0}$ denotes taking the cone $n_0$ times. 
\end{lemma}

\begin{proof}
	Using the notations in the statement we have
	$$ \ell(V_i) = \begin{cases}
        1, & \quad\text{if}\quad 1 \leq i \leq n_+; \\
        -1, & \quad\text{if}\quad n_++1 \leq i \leq n_++n_-; \\
        0, & \quad\text{if}\quad n_++n_-+1 \leq i \leq n+1.
    \end{cases} $$
    Moreover
    $$ \delta_n \cap \Lambda = \left\{ Q=(\theta_1, \dots, \theta_{n+1}) \ \middle| \ \sum \theta_i = 1, \ \theta_i \in [0,1], \ \ell(Q) = 0 \right\}. $$
    We also denote
    \begin{equation}\label{eqn:def-Sk-convex}
        S_k = \left\{ Q_k=(\theta_1, \dots, \theta_{n_++n_-+k}, 0, \dots, 0) \ \middle| \ \sum \theta_i = 1, \ \theta_i \in [0,1], \ \ell(Q_k)=0 \right\}
    \end{equation}
    for $0 \leq k \leq n_0$. In particular, we have $\delta_n \cap \Lambda = S_{n_0}$.
    
    \textbf{Claim 1.} We first show that there is an affine equivalence
    \begin{equation*}\label{eqn:first-polyhedron}
        S_0 \cong \delta_{n_+-1} \times \delta_{n_--1}.
    \end{equation*}
    For this we identify $\delta_{n_+-1}$ and $\delta_{n_--1}$ as subsets of $\delta_n$ by
    \begin{align*}
        \delta_{n_+-1} &= \left\{ Q_+ = (\theta_1, \dots, \theta_{n_+}, 0, \dots, 0, 0, \dots, 0) \ \middle| \ \sum \theta_i = 1, \ \theta_i \in [0,1] \right\}; \\
        \delta_{n_--1} &= \left\{ Q_- = (0, \dots, 0, \theta_{n_++1}, \dots, \theta_{n_++n_-}, 0, \dots, 0) \ \middle| \ \sum \theta_i = 1, \ \theta_i \in [0,1] \right\}. \\
    \end{align*}
    To construct the desired affine equivalence, we first consider the  
    map $\varphi_1 \colon \delta_{n_+-1} \times \delta_{n_--1} \to S_0$ given by $\varphi_1(Q_+, Q_-) = \frac{Q_+ + Q_-}{2}$. Indeed, for any given points $Q_+ \in \delta_{n_+-1}$ and $Q_- \in \delta_{n_--1}$, since $\ell(Q_+)=1$ and $\ell(Q_-)=-1$, we find that $\varphi_1(Q_+, Q_-) \in S_0$. Note that the map $\varphi_1$ 
    can be interpreted as the restriction of a linear map $\mathbb R^{n_+} \times \mathbb R^{n_-} \to \mathbb R^{n+1}$.

    We then construct an inverse map $\varphi_2 \colon S_0 \to \delta_{n_+-1} \times \delta_{n_--1}$. Given any point $ Q_0 = (\theta_1, \dots, \theta_{n_++n_-}, 0, \dots, 0) \in S_0 $, since
    \begin{align*}
    	\sum \theta_i &= \theta_1 + \dots + \theta_{n_+} + \theta_{n_++1} + \dots + \theta_{n_++n_-} = 1, \\
    	\ell(Q_0) &= \theta_1 + \dots + \theta_{n_+} - \theta_{n_++1} - \dots - \theta_{n_++n_-} = 0,
    \end{align*}
    we have that
    $$ \theta_1 + \dots + \theta_{n_+} = \theta_{n_++1} + \dots + \theta_{n_++n_-} = \frac{1}{2}. $$
    Then we define $\varphi_2(Q_0) = (Q_+, Q_-)$ where
    \begin{align*}
        Q_+ &= \left( 2\theta_1, \dots, 2\theta_{n_+}, 0, \dots, 0, 0, \dots, 0 \right) \in \delta_{n_+-1}; \\
        Q_- &= \left( 0, \dots, 0, 2\theta_{n_++1}, \dots, 2\theta_{n_++n_-}, 0, \dots, 0 \right) \in \delta_{n_--1}.
    \end{align*}

    It is straightforward to check that $\varphi_1$ and $\varphi_2$ are mutually inverse to each other. Hence Claim 1 holds.
    
    \textbf{Claim 2.} We show that $S_k = C(S_{k-1})$ for each $1 \leq k \leq n_0$.

    On the one hand, for any $Q_k \in S_k$ as given in \eqref{eqn:def-Sk-convex}, if $Q_k \neq V_{n_++n_-+k}$, then
    $$ \theta_1 + \dots + \theta_{n_++n_-+k-1} = 1 - \theta_{n_++n_-+k} >0. $$ We define
    $$ Q_{k-1} = \frac{1}{1 - \theta_{n_++n_-+k}} (\theta_1, \dots, \theta_{n_++n_-+k-1}, 0, 0, \dots, 0). $$
    Then we have $$ Q_k = (1 - \theta_{n_++n_-+k}) Q_{k-1} + \theta_{n_++n_-+k} V_{n_++n_-+k}. $$
    Moreover, since $\ell(Q_k) = \ell(V_{n_++n_-+k}) = 0$, we have $\ell(Q_{k-1})=0$. It follows that $Q_{k-1} \in S_{k-1}$, which shows that $S_k \subseteq C(S_{k-1})$ where the vertex of the cone is $V_{n_++n_-+k}$.

    On the other hand, given any $Q_{k-1} \in S_{k-1}$ and $\lambda \in [0,1]$, it is easy to check that the convex combination $$ (1-\lambda) Q_{k-1} + \lambda V_{n_++n_-+k} \in S_k, $$ which implies that $C(S_{k-1}) \subseteq S_k$. This completes the proof of Claim 2.

    By Claim 1 and repeatedly applying Claim 2 for $k=1, \dots, n_0$, we obtain
    $$ \delta_n \cap \Lambda = S_{n_0} = C^{n_0}(\delta_{n_+-1} \times \delta_{n_--1}) $$
    as desired.
\end{proof}


\begin{thebibliography}{}



\bibitem[BHPT]{CY3}
F.~Bogomolov, L.~H.~Halle, F.~Pazuki, and S.~Tanimoto. 
\newblock{Abelian Calabi-Yau threefolds: N\'eron models and rational points,} 
\newblock{\em Math. Res. Lett.} \textbf{25}, no. 2, 367--392 (2018).

\bibitem[BLR90]{neron}
S.~Bosch, W.~{L\"u}tkebohmert, and M.~Raynaud.
\newblock {N\'eron models,} 
\newblock {\em Ergebnisse der Mathematik und ihrer Grenzgebiete, Springer-Verlag} (21), x + 325 pp. (1990).

 \bibitem[Bo14]{Boschbook}
S.~Bosch.
\newblock {Lectures on formal and rigid geometry,}
 \newblock{\em Lecture Notes in Mathematics 2105. Cham: Springer,} 
 viii+254 pp. (2014).

\bibitem[BM19]{BrownMazzon}
M.V.~Brown and E.~Mazzon.
\newblock{The essential skeleton of a product
of degenerations}, 
\newblock{\em Compos. Math.} \textbf{155}, no. 7, 1259--1300 (2019).

\bibitem[CL16]{ChiLaz}
B.~Chiarellotto and C.~Lazda.
\newblock{Combinatorial degenerations of surfaces and Calabi-Yau threefolds}, 
\newblock{\em Algebra Number Theory,} \textbf{10}, no. 10, 2235 -- 2266 (2016).

\bibitem[DM69]{DM69}
P.~Deligne and D.~Mumford, 
\newblock{The irreducibility of the space of curves of given genus},
\newblock{\em Inst. Hautes Etudes Sci. Publ. Math.,} \textbf{36}, 75--109 (1969).

\bibitem[dFKX17]{dFKX}
T.~de Fernex, J.~Koll\'{a}r and C.~Xu.
\newblock{The dual complex of singularities},
\newblock{\em Higher dimensional algebraic geometry---in honour of
              {P}rofessor {Y}ujiro {K}awamata's sixtieth birthday},
\newblock{\em Adv. Stud. Pure Math., Math. Soc. Japan, Tokyo},
\newblock{74, pp.~103--129 (2017).}

\bibitem[FC90]{FaltingsChai}
G.~Faltings and C.-L.~Chai.
\newblock{Degeneration of Abelian varieties.} 
\newblock{\em Ergeb. Math. Grenzgeb., Springer-Verlag} (3), xii + 316 pp. (1990).

\bibitem[Fel68]{Feller-1968}
Feller, W. 
\newblock{An Introduction to Probability Theory and Its Applications}, 
\newblock{\em Wiley, New York,} 1, 3rd ed., xviii+509 pp. (1968).

\bibitem[Fo68]{Fogarty}
J.~Fogarty.
\newblock {Algebraic families on an algebraic surface}
\newblock{\em Amer. J. Math}, \textbf{90}, 511--521 (1968).



\bibitem[GHH19]{GHH}
M.~G.~Gulbrandsen, L.~H.~Halle, and K.~Hulek. 
\newblock {A GIT construction of degenerations of Hilbert schemes of points}. 
\newblock {\em Doc. Math.}, \textbf{24}, 421--472 (2019).

\bibitem[GHHZ21]{GHHZ21}
	M.~G.~Gulbrandsen,  L.~H.~Halle, K.~Hulek and Z.~Zhang.
	\textit{The geometry of degenerations of Hilbert schemes of points}.
	J. Algebraic Geom. \textbf{30}, 1--56 (2021).

\bibitem[HN10]{HaNi-comp}
L.~H.~Halle and J.~Nicaise.
\newblock{The N\'eron component series of an abelian variety.}
\newblock{\em Math.~Ann.}, \textbf{348}, no. 3, 749--778 (2010).

\bibitem[HN18]{HN}
L.~H.~Halle and J.~Nicaise.
\newblock{Motivic zeta functions of degenerating Calabi-Yau varieties.}
\newblock{\em Math.~Ann.,} \textbf{370}, no. 3-4, 1277--1320 (2018).

\bibitem[Ha77]{Hartshorne}
R.~Hartshorne.
\newblock{Algebraic Geometry.}
\newblock{\em Graduate Texts in Mathematics, 52, Springer Verlag, New York -- Heidelberg,}
\newblock{xvi+496 pp. (1977).}

\bibitem[HR74]{Hochster-Roberts}
 M.~Hochster and J.L.~Roberts. 
 \newblock{Rings of invariants of reductive groups acting on regular rings are Cohen-Macaulay}. 
 \newblock{\em Advances in Mathematics},  13 (2), 115--175 (1974).



\bibitem[KLSV18]{KLSV}
J.~Koll{\'a}r, R.~Laza, G.~Sacc{\`a}, and C.~Voisin. 
 \newblock{Remarks on degenerations of hyper-K\"ahler manifolds},
\newblock{\em Ann. Inst. Fourier} (Grenoble), \textbf{68}, no. 7, 2837--2882 (2018). 

\bibitem[K\"u98]{Kunnemann}
K.~K\"unnemann.
\newblock{Projective regular models for abelian varieties, semistable reduction, and the height pairing}.
\newblock{\em Duke Math. J.}, \textbf{95}, 161--212 (1998).

\bibitem[Ku77]{Ku77}
 V.~S.~Kulikov
 \newblock{Degenerations of {{\(K_3\)}} surfaces and {Enriques} surfaces}.
 \newblock{\em Math. USSR, Izv.}, \textbf{11}, 957--989 (1977).  



\bibitem[LL01]{LiuLor}
Q.~Liu and D.~Lorenzini.
\newblock{Special fibers of N\'eron models and wild ramification.}
\newblock{\em J.~reine angew.~Math.}, \textbf{532}, 179--222 (2001). 



\bibitem[Liu06]{liu}
Q.~Liu.
\newblock {Algebraic geometry and arithmetic curves.} 
\newblock {\em Oxford Graduate Texts in Mathematics, Volume~6},
\newblock {\em Oxford University Press}, xv + 577 pp. (2006). 

\bibitem[Li01]{Li01}
J.~Li.
\newblock{Stable morphisms to singular schemes and relative stable morphisms.}
\newblock{\em J. Differential Geom.}, \textbf{57}, no. 3, 509--578 (2001).

\bibitem[Li13]{Li13}
J.~Li.
\newblock{Good degenerations of moduli spaces.}
\newblock{\em Handbook of moduli. {V}ol. {II}}, Adv. Lect. Math. (ALM) \textbf{25}, 299--351 (2013).

\bibitem[GIT]{GIT}
D.~Mumford, J.~Fogarty, F.~Kirwan.
\newblock {Geometric invariant theory. Third edition.}
\newblock {\em Ergebnisse der Mathematik und ihrer Grenzgebiete (2) [Results in Mathematics and Related Areas (2)], 34,}
\newblock {\em Springer-Verlag, Berlin}, xiv+292 pp. (1994).


\bibitem[Na08]{Na08}
Y.~Nagai. 
\newblock{On monodromies of a degeneration of irreducible symplectic {K{\"a}hler} manifolds}.
\newblock{\em Math. Z.}, 
\textbf{258}, no. 2, 407--426 (2008).

\bibitem[Na18]{Na18}
Y.~Nagai. 
\newblock{Symmetric products of a semistable degeneration of surfaces}.
\newblock{\em Math. Z.}, 
\textbf{289}, no. 3-4, 1143--1168 (2018).


\bibitem[Na21]{Nagai-GHH}
Y.~Nagai. 
\newblock{Gulbrandsen–Halle–Hulek degeneration and Hilbert-Chow morphism}.
\newblock{\em Pure Appl. Math. Q.}, 
\textbf{17}, no. 1, 401--422 (2022).




\bibitem[Ov21]{Overkamp}
O.~Overkamp.
\newblock {Degenerations of Kummer surfaces}.
\newblock{\em Math. Proc. Cambridge Philos. Soc.}, \textbf{171}, no. 1, 65--97 (2021).

 \bibitem[PP81]{PP81}
 U.~Persson, H.~Pinkham.
 \newblock{Degeneration of surfaces with trivial canonical bundle}.
 \newblock{\em Ann. Math. (2)},\textbf{113}, 45--66 (1981). 

 \bibitem[ST25]{ST25}
 Q.~Shafi, C.~Tschanz.
\newblock{From logarithmic Hilbert schemes to degenerations of hyperkähler varieties}.
\newblock{\url{arXiv:2512.21190}.}

\bibitem[Stacks]{Sta}
The Stacks Project Authors.
\newblock{\em Stacks Project}.
\newblock{\url{https://stacks.math.columbia.edu}} (2024).

\bibitem[Tsch24]{Tsch24}
 C.~Tschanz.
\newblock{Good models of Hilbert schemes of points over semistable degenerations}.
\newblock{\url{arXiv:2402.10209}.}

\bibitem[Tsch26]{Tsch23}
 C.~Tschanz.
\newblock{Expansions for Hilbert schemes of points on semistable degenerations}.
\newblock{\em Forum of Mathematics, Sigma}, \textbf{14}, e53 (2026).

\bibitem[Zie95]{Zie}
G.~M.~Ziegler.
\newblock{Lectures on polytopes}.
\newblock{\em Graduate Texts in Mathematics, 152, Springer Verlag, New York,}
\newblock{x+370 pp. (1995).}









\end{thebibliography}
\end{document}